\documentclass [11pt]{article}
\usepackage{verbatim}
\usepackage{amssymb,amsmath}
\usepackage{enumerate}
\usepackage{amsfonts}
\usepackage{amsmath}
\usepackage{color}
\usepackage{graphicx}
\usepackage{hyperref}
\usepackage{amsthm}

\usepackage{geometry}

\newtheorem{corollary}{Corollary}
\newtheorem{definition}{Definition}
\newtheorem{theorem}{Theorem}
\newtheorem{lemma}[corollary]{Lemma}
\newtheorem{proposition}{Proposition}

\newcommand\numberthis{\addtocounter{equation}{1}\tag{\theequation}}
\newcommand{\PP} {{\mathbb P}}
\newcommand{\Z}{{\mathbb Z}}

\newcommand{\EE}{\mathbb{E}}
\newcommand{\hcap}{\text{hcap}}

\newcommand{\C}{{\mathbb C}}

\newcommand{\dist}{{\rm dist}}
\newcommand{\x}{{\bf x}}
\newcommand{\y}{{\bf y}}

\newcommand{\hp}{\mathbb{H}}
\newcommand{\bigo}{\text{O}}
\newcommand{\smo}{\text{o}}

\newcommand {{\cent}} {{\bf c}}

\def \bgamma {{\pmb \gamma}}

\def \sle {SLE_\kappa}
\def \Half {\mathbb{H}}
\def \sm {\setminus}

\def \Im {{\rm Im}}
\def \Re {{\rm Re}}
\def \p {\partial}
\def \hcap {{\rm hcap}}
\def \cc {\textbf{c}}
\def \AA {\textbf{A}}

\def \diam {\text{diam}}
\def \HH {\textbf{H}}
\def \eset {\emptyset}
\def \LL {\textbf{L}}
\def \disk {\mathbb{D}}
\def \Disk {{\disk}}
\def \xbar {\bar{u}'}
\def \x {u'}
\def \ybar {\bar{w}'}
\def \y {w'}
\def \ubar {\bar{u}}
\def \u {u}
\def \wbar {\bar{w}}
\def \w {w}
\def \Xbar {\bar{U}'}
\def \X {U'}
\def \Ybar {\bar {W}'}
\def \Y {W'}
\def \Ubar {\bar {U}}
\def \U {U}
\def \Wbar {\bar {W}}
\def \W {W}

\def \xx {\bar x}
\def \XX {\bar X}
\def \yy {\bar y}
\def \YY {\bar Y}

\def \hbar {\bar {h}}
\def \tg { {g}}
\def \zbar {\bar {z}}

\def \sm {\setminus}
\def \bgamma {\boldsymbol{\gamma}}
\def \prd {\text{prod}}
\def \pf {V}
\def \gg {\bar{g}_t}
\def \bfz {\textbf{z}}
\def \bfw {\textbf{w}}
\def \bfu {\textbf{u}}
\def \bfz {\textbf{z}}
\def \bfwb {\bar{\textbf{w}}}
\def \bfub {\bar{\textbf{u}}}

\title{Multiple-Paths $\sle$ in Multiply Connected Domains}
\author{Mohammad Jahangoshahi\\ Gregory F. Lawler}

\begin{document}

\maketitle
\begin{abstract}
We define multiple-paths Schramm-Loewner evolution ($\sle$) in multiply connected domains when $\kappa\leq 4$ and prove that in annuli, the partition function is smooth. Moreover, we give up-to-constant estimates for the partition function of bi-chordal annulus $\sle$ measure and we establish a connection between this measure   and two-sided $\sle$ in the unit disk.
\end{abstract}
\tableofcontents
\section{Introduction}



The Schramm-Loewner evolution ($\sle$) is a one parameter family of measures on planar curves discovered by Oded Schramm \cite{oded}.  It was created as a candidate for the scaling
limit of measures on lattice paths arising in statistical physics.  It is a family of  probability measures
$\mu_D^\#(z,w)$ indexed by domains $D \subset \C$ and 
points $z,w$ that can be boundary or interior points.  They
are supported
  on curves $\gamma:(0,t_\gamma) \rightarrow
D$ with $\gamma(0+) = z, \gamma(t_\gamma-) = w$. Here
$z \in \p D$ and $w$ can be in $D$ or $\partial D$.  He made the following
two assumptions based on the conjectured behavior of
the scaling limit of the lattice paths.  
\begin{itemize}
\item {\bf Conformal Invariance.}  If $f: D \rightarrow
f(D)$ is a conformal transformation, then the push forward
$f \circ \mu_D^\#(z,w)$ of the measure is the same as
$\mu_{f(D)}^\#(f(z),f(w))$.
\item  {\bf Domain Markov property.}  In the measure
$\mu_D^\#(z,w)$, given an initial segment of the path
$\gamma(s), 0 \leq s \leq t$, the distribution of the
remainder of the path is $\mu_{\tilde D}  ^\#
(\gamma(t), w)$.  where
$\tilde D$ is the connected component of $  D \setminus \gamma[0,t]$ with
$\gamma(t)$ and $w$ on its boundary.

\end{itemize}
If such a family of measures exists, 
 $D$ is simply connected, and $z \in \p D$, then  
$\tilde D$ is also simply connected and we can find
$f: D \rightarrow \tilde D$ with $f(z) = \gamma(t), f(w)
 =  w .$ Hence,  the measure $\mu_{\tilde D}^\#
(\gamma(t), w)$ is the  same as $f \circ \mu_D^\#(z,w)$.
This observation was the starting point for Schramm's
construction and he showed that there exists a one-parameter family of measures on curves (modulo reparametrization)
satisfying these conditions.  If $w \in \p D$ this is
called \textit{chordal} $\sle$ and if $w \in D$ this is called
\textit{radial} $\sle$. See \cite{greg_book} for the
basic facts about $\sle$ including the following.
\begin{itemize}
\item  If $0 < \kappa \leq 4$, $\mu_D^\#(z,w)$ is a measure
on simple (non self-intersecting) curves with $\gamma(0,t_\gamma) \subset D$.
\item  If $4 < \kappa < 8$, then $\mu_D^\#(z,w)$ is a measure
on intersecting curves in $\overline D$
 with $\gamma(0,t_\gamma) \cap \p D \neq \eset$.
 \item  If $\kappa \geq 8$, this gives a measure on
 plane filling curves with $\gamma(0,t_\gamma) \cap \p D \neq \eset$.
\end{itemize}

This paper will be concerned mostly with the $\kappa \leq 4$ case and let us assume that for the rest of this introduction.
Let us also be a little more precise about the parametrization
of the curves.  The measure $\mu_D^\#(z,w)$ can be given
in terms of \textit{naturally parametrized} curves \cite{rezaei} where the parametrization
is chosen so that the $(1 + \frac \kappa 8)$-Minkowski content
of $\gamma[0,t]$ is $t$.  Under conformal transformations,
the parametrization changes appropriately: the time
to traverse $f \circ \gamma[0,t]$ is
\[    \int_0^t |f'(\gamma(s))|^\alpha \, ds, \;\;\;\;
  \alpha = 1 + \frac \kappa 8.\]
This choice of parametrization has the property that  parametrization of a curve $\gamma$ does not depend on the domain $D$ in which lives; this is useful when comparing $\mu_D^\#(z,w)$ with $\mu_{\tilde D}^\#(z,w)$ with $\tilde D
\subset D$.   This is not the most convenient parametrization
when analyzing $\sle$ in a fixed domain $D$; here capacity
parametrizations as originally chosen by Schramm are more
convenient.  The capacity parametrizations do not
have the independence of domain property.    While we use the natural parametrization  in our definitions,
in our 
analyses we will use capacity parametrizations.  It follows
from the definition of $\sle$ and the strong Markov
property for Brownian motion, that it satisfies the ``strong'' domain Markov property, i.e., the $t$ in the definition
 can be chosen to be
a stopping time.

If we restrict to the $\kappa \leq 4$ case, then the
{\em boundary perturbation} or 
{\em generalized restriction} property, which we now describe, shows that if $D' \subset
D$ and $D,D'$ agree in neighborhoods of $D',D$ then
$\mu_{D'}^\#(z,w) \ll \mu_D^\#(z,w)$.  This is more nicely
expressed if we consider $\sle$ as a measure that is
not necessarily a probability measure.  Let us fix
some constants now that will be used throughout this paper,
\begin{equation}\label{constants}
a = \frac{2}{\kappa},\qquad b =  \frac{6 - \kappa}{2\kappa},\qquad \tilde{b} = \frac{b(\kappa-2)}{4},\qquad\cc = \frac{(6-\kappa)(3\kappa-8)}{2\kappa}.
\end{equation}
The measures $\mu_D(z,w)$ are defined if $z,w$ are  analytic boundary points (or $w \in D$).  They satisfy the conformal
covariance property
\begin{equation}  \label{sep19.1}
   f  \circ \mu_D(z,w) = |f'(z)|^b \, |f'(w)|^{b'}
 \, \mu_{f(D)}(f(z),f(w)), 
 \end{equation}
where $b' = b$ if $w \in  \p D$ and $b' = \tilde b$
if $w \in D$.  This  assumes that $f(z),f(w)$
are also nice boundary points in $f(D)$ (or $w \in f(D)$).  The total masses
can be determined by making the arbitrary choices
\[   \| \mu_\Half(0,1)\| = 1 , \;\;\;\; \|\mu_\Disk(1,0)\|
 = 1 , \]
 and using the implicit scaling rule from \eqref{sep19.1}.
Here $\Half,\Disk$ denote the upper half plane and unit
disk, respectively.
 We need to review the 
Brownian loop measure as first introduced in \cite{greg_loop}.
A loop $\ell \subset \mathbb{C}$ can be described by a triple $(z,t,\bar{\ell})$, where $z$ is the root $z = \ell(0)$, $t$ is the time duration of the loop and $\bar{\ell}(s) := (\ell(ts)-z)/\sqrt{t}$ is a loop of time duration 1. The Brownian loop measure $m_\mathbb{C}$ is the measure induced on unrooted loops by  
\[
\text{Area}\times \frac{1}{2\pi t^2}\,dt\times\,\text{Brownian bridge distribution}
\]
on triples $(z,t,\bar{\ell})$. 
 For a domain $D$, $m_D$ is restriction of $m_\mathbb{C}$ to the loops in $D$. 
  We can now state the generalized restriction or boundary perturbation rules.
\begin{itemize}
\item   If $D' \subset D$ and $D$ and $D'$ agree 
in neighborhoods of $z,w$, then $\mu_{D'}(z,w) \ll
\mu_D(z,w)$ with Radon-Nikodym derivative
\begin{equation}  \label{genres}
      1\{\gamma \subset D'\} \, \exp \left\{
\frac \cent 2 \, m_D(\gamma,D \setminus D') \right\}. 
\end{equation}
\end{itemize}
Here $\gamma =  \gamma(0,t_\gamma)$
   and  $m_D(\gamma,D \setminus D') $ denotes
the Brownian loop measure of loops in $D$ that
intersect both $\gamma$ and $D \setminus D'$. 
%

 The chordal
version of the boundary restriction property is a combination of results in \cite{greg_restrict,greg_loop,Parkcity}.
In this case for simply connected domains
$\|\mu_D(z,w)\| = H_D(z,w)^b$ where $H_D(z,w)$ is the boundary
Poisson kernel normalized so that $H_\Half(0,1) = 1$ (we use 
this normalization throughout this paper). 
As we were preparing with the paper, we could not find 
the radial version proved in the literature so we have
included a version here (the $\kappa = 8/3, \cent = 0$ radial case 
was done in \cite{greg_restrict}).  In this case, the
total mass $\|\mu_D(z,w)\|$ is not given by a power
of the Poisson kernel except for $\kappa = 2$ for which
$\tilde b = 0$.

The definition of $\sle$ uses the domain Markov property;
it describes the evolution of the path in terms of the path
up to that point.  One can similarly ask for the distribution
of the path given other parts of the path other than an
initial segment.  For chordal $\sle$, one can first consider the distribution given a final segment.  This is closely related to the reversibility of $\sle$ first proved by
 Zhan \cite{dapeng_2008}.  Roughly speaking it was shown
 how to grow a chordal $\sle$ path simultaneously at the
 starting point and the terminal point until they meet.
 If one stops this at a stopping time, the distribution
 of the remainder of the path is $\sle$ in the
 remaining domain between the two endpoints. 

A similar question can be asked for radial $SLE$.  The
definition gives the conditional distribution given an
initial segment.  The distribution given a final segment is a little trickier since a simply connected domain slit by an interior segment is no longer simply connected but rather conformally equivalent to an annulus.  One would like to extend the domain Markov property to state that the distribution of the remainder of the curve given a terminal segment is $SLE$ in the annulus.  This leads to the
general question of defining $SLE$ in annuli, and more
generally, in multiply connected domains.

%
  
Several authors have studied $\sle$ in multiply connected domains using different methods.
Bauer and Friedrich (\cite{bauer,bauer2}) used a generalization of the Loewner equation for multiply connected domains to describe the driving function of $\sle$. 
Zhan's approach in \cite{dapeng_2014} was similar, in that he used a generalization of Loewner equation to define annulus $\sle$. 
 In addition to conformal invariance and Markov property, he required $\sle$ to be reversible and used that to uniquely determine the driving function. 
These articles are based on the work of Komatu \cite{komatu} in 1950, who formulated a generalization of the Loewner equation in multiply connected domains.

Our approach goes back to \cite{michael} where it was
conjectured that   $\sle$ (either chordal or radial) is the scaling limit of a measure on self-avoiding paths weighted by
the measure of random walk loops that can be added.
This conjecture comes from the boundary perturbation
rule for $\sle$ involving the Brownian loop measure.
In \cite{greg_annulus}, it was suggested to use this
boundary perturbation rule as the definition of $\sle$ 
in all domains for $\kappa \leq 4$; see section \ref{mcsec}.
   It is not difficult to
show that it is well defined, but other issues arise as
we will see.  Finiteness of the measure for general domains
is still open for $8/3 < \kappa \leq 4$ ($\cent > 0$).
If the partition function is finite and signficantly smooth,
one can describe 
 the probability measure  by running $\sle$ and 
tilting by the normalized partition function. 
The Girsanov theorem then gives the drift of the driving
function in terms of a logarithmic derivative of the
partition function.   This analysis
requires the partition function to be sufficiently smooth
and this does not follow immediately from the
definition in \cite{greg_annulus}.   
In the case of the annulus, it
was shown there that the annulus $SLE$ with finite
partition function is well defined and sufficiently
smooth, and that the
corresponding probability measure on paths is the
same as discovered by Zhan in \cite{dapeng_2014}.  The distribution
of the (reversal of the) terminal part of radial $\sle$ 
is absolutely continuous with respect to whole
plane $\sle$ see \cite{field}.

Instead of growing a $\sle$ curve from one point to another, one might be interested in the simultaneous growth from the two ends,  which invites the study of  $\sle$ measures on multiple path
\[
\bgamma = (\gamma^1,\ldots,\gamma^n).
\] 
Unlike $\sle$ measures on single curves, conformal invariance and domain Markov property do not uniquely specify the measure when $2\leq n$.
In \cite{julien_2005, julien_2006}, Dub\'edat characterized multiple-paths $\sle$ in simply connected domains using a commutation relation for certain differential operators related to the driving functions. He also gave a discussion about multiply connected domains, but did not give a complete classification. In \cite{michael}, the process was required to satisfy the restriction property in addition to the domain Markov property and conformal invariance and  a global construction was given using the Brownian loop measure for $\kappa\leq4$.

A similar question concerns chordal $\sle$ conditioned
to go through an interior point.  For ease, let us consider
$\sle$ connecting two boundary points of the unit
disk conditioned to go through the origin.  While this
is conditioning on an event of probability zero, it
is not difficult to make sense of this for
$0 < \kappa < 8$ as $\sle$ tilted by the partition function.
The process is called two-sided radial $\sle$ since the
intuition is of a pair  of paths starting at the initial
and terminal points, both going to the interior point with an
interaction term.  The precise definition looks at one
path, tilts by the Green's function until it hits the
interior point, and then continues as a chordal $\sle$ for
the remainder.  This process motivates the results of
the paper.  In particular, we are interested in the
distribution of the process given an interior segment.
In other words, if $\eta$ is a curve in $\Disk$ going through
the origin, what is the distribution of chordal $\sle$ from $z$ to $w$ in $\Disk$ conditioned that it contains $\eta$?
The answer should be a pair of $\sle$ paths in the annulus
$\Disk \setminus \eta$.  The purpose of this paper is to make this idea precise.

 As mentioned before,  partition function of  chordal $\sle$ in simply connected domains is given by an exponent of the Poisson kernel. 
 In annuli, the partition function has been proved to be smooth and is described by a particular differential equation (\cite{greg_annulus, dapeng_2014}). 
 For multiple-paths $\sle$ in simply connected domains, Dub\'edat \cite{julien_2006} proved that partition function can be described as a family of Euler integrals taken on a specific set of cycles. In \cite{eve}, Peltola and Wu used technique from  partial differential equations  such as H\"ormander's theorem to show that the partition function satisfies a particular PDE when $\kappa\leq 4$. By only using techniques from probability, it was directly proved in \cite{multiple} that the partition function is smooth and satisfies the same PDE as in \cite{julien_2006,eve} when $\kappa<4$.

Our definition of multiple-paths $\sle$ in multiply connected domains is similar to the approach of \cite{multiple,michael} in simply connected domains. 
That is, we define it to be the measure absolutely continuous with respect to the product of chordal $\sle$ measures with a particular Radon-Nikodym derivative involving the Brownian loop measure. 
To that end, we build on the definition of annulus $\sle$ in \cite{greg_annulus}. 
We find this definition to be the most natural one because it provides a clear consistency with $\sle$ in simply connected domains.  We prove the following theorem regarding the partition function of this process.
\begin{theorem}\label{thmucsd}
Define $A_r = \{z\in\disk;\,|z|>e^{-r}\}$ and let $0\leq u_1,\ldots, u_n<2\pi,\,0\leq w_1,\ldots,w_n<2\pi$ be distinct numbers. Let $\bfub = (e^{iu_1},\ldots,e^{iu_n}),\,\bfwb = (e^{-r+iw_1},\ldots,e^{-r+iw_n})$. 
Let $\Psi_{A_r}(\bfub,\bfwb)$ denote the partition function of multiple-paths $\sle$ connecting $e^{iu_j}$ to $e^{-r+iw_j}$ for $1\leq j\leq n$. 
If $\kappa\leq 4$,  then $\Psi_{A_r}(\bfub,\bfwb)$ is a smooth function of $r,u_1,\ldots,u_n,w_1,\ldots,w_n$.

\end{theorem}
Even though we will only prove  the smoothness  for $n$ crossing paths, other cases can be proved in a similar manner and we will omit the details here.

 In this work, we give a construction of two-sided $\sle$
by describing a Radon-Nikodym derivative with respect to the product of two independent radial $\sle$ measures. This 
 allows us to grow the curves from the endpoints $z, w$ at the same time. 
 We will use this construction to prove the following.
\begin{theorem}\label{serenade2}
The distribution of bi-chordal $\sle$ in $A_r:=\{z\in\disk;\,|z|>e^{-r}\}$ is absolutely continuous with respect to the distribution of two-sided $\sle$ grown simultaneously from the endpoints and stopped before reaching the boundary. Moreover, we can estimate the Radon-Nikodym derivative up to multiplicative constants.  
\end{theorem}

%


We finish this introduction by giving an outline of this paper. 
 In section \ref{background}, we establish our notation, give a brief review of chordal
 $\sle$ in multiply connected domains and define the multiple-paths (chordal)
  $\sle$ measure in multiply connected domains. 
The definitions here are not difficult but rely on 
the descriptions of $\sle$ as a nonprobability measure. 
The particular case of the annulus is discussed.  Although the annulus Loewner equation will be used in analyzing this, it is not part of the definition.   This defines the partition function for the single path and
multiple-path $\sle$ measures.  

Section \ref{confsec} 
contains a number of deterministic estimates about Brownian motion.  This section can be skipped on a first reading
and referred to when needed.   The next subsection reviews
the version of the annulus Loewner equation we will use.  This
section is deterministic --- given a curve in the annulus, what
is the Loewner equation that conformal map satisfies.  There is a little work to be done here since the conformal map
onto an annulus is only defined up to a rotation. 
We finish this section by deriving the equation for 
the probability measure $\mu^\#_{A_r}(\ubar,\wbar)$ 
  in terms of the annulus Loewner equation
with a driving function equal to Brownian motion with
a drift.  This is not the definition; rather it is a
derivation using the previous definition.

We prove that the partition function is smooth in section \ref{proof}. The main idea of the proof is to define appropriate martingales and apply the It\^o's formula using the H\"ormander's theorem. 
Finally, section \ref{twosided} consists of a number of  results about two-sided $\sle$ and bi-chordal annulus $\sle$.
 In particular, we show that two-sided $\sle$ can be constructed by weighting two independent radial $\sle$  by an appropriate martingale. In addition, we derive asymptotic estimates for the partition function of bi-chordal annulus $\sle$. 
We use that to show  two-sided $\sle$ can be approximated by bi-chordal annulus $\sle$  and prove Theorem \ref{serenade2}.

\section{Preliminaries and Definitions}\label{background}
\subsection{Notation}\label{secnotation}
In this article, unless mentioned otherwise, we will use the following notation. 
\begin{itemize}
\item\textbf{Notation.} We denote the unit disk by $\disk$ and the upper-half plane by $\hp$. For $r>0$, define 
\[
C_r = \{z\in\mathbb{C};\,|z| = e^{-r}\},\qquad \disk_r = e^{-r}\,\disk,
\]
\[
A_r = \{z\in\mathbb{C};\,e^{-r}<|z|<1\},\qquad S_r = \{ z \in \hp:\, \Im(z) < r\}.
\]
Define
\[
\psi(z) = e^{iz}
\]
which is a many-to-one (covering) map from $S_r$ onto $A_r$. 

If $K_1,K_2\subset D$, then we denote by $m_D(K_1,K_2)$ the Brownian loop measure of loops in $D$ that intersect both $K_1,K_2$.
Suppose $\gamma:(0,t_\gamma)\to {A_r}$ is a simple curve with $\gamma(0+) = \ubar \in C_0,\,\gamma(t_\gamma-) = \wbar\in C_r$. We will write $\gamma_t$ for $\gamma((0,t])$.
For any $t<t_\gamma$, there exists a unique $0<r(t)\leq r$ and a conformal transformation $\hbar_t$ such that
\begin{equation}\label{eq2.5}
\hbar_t : A_r\sm \gamma_t \to A_{r(t)},\qquad \hbar_t(C_r) =C_{r(t)}.
\end{equation} 
(e.g., see \cite{greg_book} for a proof). 
It is easy to see that the function $\hbar_t$ is unique up to a rotation. We will uniquely specify our choice of rotation in section \ref{seclow}. 

Choose $0\leq u,w<2\pi$ such that $\psi(u)=\ubar,\,\psi(w+ir)=\wbar$. For $\theta\in \mathbb{R}$, we let $\sin_2(\theta) = \sin(|\theta|/2)$.
Let $\eta_t\subset S_r$ be the unique continuous curve satisfying $\gamma_t = \psi\circ \eta_t,\,  \eta(0) = u$ and let
\begin{equation}\label{eq2.55}
\tilde\eta_t=\bigcup_{k\in\mathbb{Z}}[\eta_t+2k\pi].
\end{equation}
Define
\[
S_{r,t} = S_r \sm \tilde{\eta}_t
\]
and note that the transformation $\hbar_t$ can be raised to the covering space $S_{r,t}$ to yield a conformal transformation 
\begin{equation}\label{eq2.6}
 h_t : S_{r,t} \to S_{r(t)},\qquad \psi \circ h_t = \hbar_t \circ \psi.
\end{equation}
Similarly to $\hbar_t$, the function $h_t$ is unique up to a translation. However, regardless of the choice of translations, $h_t(z) - z$ is $2\pi$-periodic. 

Let $\bar{g}_t: \disk \sm \gamma_t \to \disk$ be the unique conformal transformation with $\gg(0) = 0,\,\gg'(0)>0$ and let 
\[
\bar{\xi} _ t = \gg(\gamma(t)).
\]
Let $\Ubar_t=\hbar_t(\gamma(t))$ and define $\bar\phi_t$ to be the unique conformal transformation satisfying
\begin{equation}\label{eq2.7}
\hbar_t = \bar\phi_t \circ \gg, \qquad \bar\phi_t(\bar\xi_t) = \Ubar_t.
\end{equation}
Similarly to $\hbar_t$, we can raise $\gg$ to the covering space $\hp\sm\tilde{\eta}_t$ to get a conformal transformation
\[
\tilde g_t:\hp\sm\tilde{\eta}_t\to \hp,
\]
such that $\xi_t=\tilde g_t(\eta(t)) $ is continuous and $\xi_0 = u,\,\psi(\xi_t) = \bar\xi_t$. Define the conformal transformation $\phi_t$ with 
\[
h_t = \phi_t \circ \tilde g_t.
\]
We let $g_t:\hp\sm\eta_t\to\hp$ be the unique conformal transformation with $g_t(z)= z+ o(1)$ as $z\to \infty$. We say $\gamma_t$ has the \emph{radial parametrization} if $\gg'(0) = at/2$ for any $t<t_\gamma$ and $\gamma_t$ has the \emph{annulus parametrization} if $r(t) = r-t$.
\item \textbf{Poisson kernel and excursion measure.}
Let $D'\subset D$ be domains in $\mathbb{C}$. Assume $z\in D$ and $w\in \partial D\cap \partial D'$ is an analytic boundary point of $D,\,D'$ with $\dist(w,D\sm D')>0$.
That is, there exists an analytic function $f:\disk \to \mathbb{C}$ such that $f(0) = z$ and $f(\disk)\cap D = f(\hp\cap \disk)$. 
 We denote by $H_D(z,w)$ the Poisson kernel in $D$ normalized so that $H_\disk(0,1) = 1/2$. Define
\[
Q_D(z,w;D') = \frac{H_{D'}(z,w)}{H_D(z,w)}.
\]
In other words, $Q_D(z,w;D')$ is the probability that an $h$-process started from $z$ and conditioned to exit $D$ at $w$ does not hit $D\sm D'$. If $w'\in\partial D\cap\partial D'$ is another analytic boundary point with $\dist(w',D\sm D')>0$, then
\[
H_{ D}(w',w) := \partial_{\textbf n_{w'}}H_D(w',w),
\]
where $\partial_{\textbf n_{w'}}$ denotes the inward normal derivative at $w'$. Similarly, we will write
\[
Q_{ D}(w',w;D') = \frac{H_{ D'}(w',w)}{H_{ D}(w',w)},
\]
which is the probability that a Brownian excursion from $w'$ to $w$ in $D$ does not intersect $D\sm D'$. Suppose $f:D\to f(D)$ is a conformal transformation such that $f(w),f(w')$ are analytic boundary points of $f(D)$. Since the Poisson kernel satisfies the conformal convariance property, $Q_D$ is conformally invariant and 
\begin{align*}
Q_D(z,w;D') &= Q_{f(D)}(f(z),f(w);f(D')),\\
 Q_{ D}(w',w;D') &= Q_{ f(D)}(f(w'),f(w);f(D')).
\end{align*}

Suppose $B_t$ is a Brownian motion and $\tau_D$ is the firs time it exits $D$. 
Let $V_1,V_2\subset \partial D$ be segments of the boundary and assume $V_1$ is smooth. 
For any $z\in D$, define
\[
\rho(z;V_2) := \PP^z[B_{\tau_D}\in V_2].
\]
 With an abuse of notation, for  $z\in V_1$ we let $\rho_D(z;V_2)$ denote the probability that a Brownian excursion starting from $z$ exits $D$ at $V_2$. In other words
\begin{equation*}
\rho_D(z;V_2) = \partial_{\textbf{n}_z} \PP^z[B_{\tau_D}\in V_2],
\end{equation*}
where  $\textbf{n}_z$ denotes the inward normal derivative at $z$. 
Let $f:D\to\tilde D$ be a conformal transformation such that $f(z)$ is an analytic boundary point. Using conformal invariance of Brownian motion, we have
\begin{equation}\label{eqsecond}
\rho_D(z;V_2) = |f'(z)| \rho_{\tilde{D}}(f(z);f(V_2)).
\end{equation}
The \emph{excursion measure} between $V_1,V_2$ is defined by
\[
\mathcal{E}_D(V_1,V_2) = \int_{V_1} \rho_D(z;V_2) |dz|.
\]
(Excursion measure is normally defined as a measure induced by Brownian motion on curves connecting two points on the boundary. What we call the excursion measure is the total mass of the aforementioned measure).
If $V_1,V_2$ are both smooth, then we can write
\[
\mathcal{E}_D(V_1,V_2) = \frac{1}{\pi}\int_{V_1}\int_{V_2} H_{ D}(w_1,w_2) |dw_2||dw_1|.
\] 
Using conformal covariance of the Poisson kernel, we can see that if $f:D\to \tilde{D}$ is a conformal transformation such that $f(V_1),f(V_2)$ are smooth, then
\begin{equation*}
\mathcal{E}_D(V_1,V_2) = \mathcal{E}_{f(D)}(f(V_1),f(V_2)).
\end{equation*}
This equality can  be used to define $\mathcal{E}_D(V_1,V_2) $ even if $V_1,V_2$ are rough. 
Note that 
\[
\mathcal{E}_{A_r}(C_0,C_r) = \int_{C_0} \rho_{A_r}(z;C_r)|dz| = \int_{C_0}\frac{1}{r}|dz| = \frac{2\pi}{r}.
\]
Therefore, by conformal invariance of the excursion measure,
$
r(t) = {2\pi}/{\mathcal{E}_{A_r\sm\gamma_t}(C_0\cup\gamma_t,C_r)}$,
where $r(t)$ is defined in \eqref{eq2.5}.

\item \textbf{$\sle$ Measures.} We use $\mu_D(w,w')$ to denote the $\sle$ measure from $w$ to $w'$ in $D$. This is considered as a measure with partition function $\Psi_D(w,w')=\|\mu_D(w,w')\|$ satisfying
\begin{equation}\label{eqdef0}
\Psi_D(w,w') = |f'(w)|^b\,|f'(w')|^b \,\Psi_{f(D)}(f(w),f(w')).
\end{equation}
If $D$ is a simply connected domain, then $\Psi_D(w,w')=H_{ D}(w',w)^b$. 
While \eqref{eqdef0} holds even if $D$ is not simply connected, it is no longer true that $\Psi_D(w,w')=H_{ D}(w',w)^b$ for multiply connected domains. 

\item \textbf{Half-plane capacity.} Suppose $K$ is a compact subset of $\hp$ such that $\hp\sm K$ is simply connected. The half-plane capacity of $K$ is defined as 
\[
\hcap(K) = \lim_{y\to \infty}y\EE^{iy}[\Im[B_\tau]],
\]
where $\tau$ is the first time a Brownian motion $B_t$ exits $\hp\sm K$.  
If $0\in \overline{K}$ and $d=\diam[K]$, then for all $|z|>2d$,  
\begin{equation}\label{tchy}
\EE^z[\Im[B_\tau]] = \Im[-1/z]\,\hcap(K)\,[1+\bigo(d/|z|)].
\end{equation}
Suppose $\gamma_t\subset \hp$ is a simple curve with $\gamma(0+)=0$. Let $V\subset \hp$ be a neighborhood of 0 and assume $\Phi$ is a conformal transformation defined on $V$ with  $\Phi(V)\subset\hp$ and $\Phi(\mathbb{R}\cap \bar V) \subset \mathbb{R}$. Then at $t=0$,
\begin{equation}\label{eqhalf}
\partial_t\hcap[\Phi(\gamma_t)] = \Phi'(0)^2\, \partial_t\hcap[\gamma_t].
\end{equation}
See \cite{greg_book} for a detailed discussion about the half-plane capacity.

\item \textbf{Convention.} For a function of the form $f_t(z)$, we use the dot derivative $\dot{f}_t(z)$ to denote the derivative with respect to $t$ and we use $f'_t(z)$ to denote the derivative with respect to $z$.
\end{itemize}
\subsection{\texorpdfstring{$\sle$}{LG}  in multiply connected domains}   \label{mcsec}
In this subsection, we briefly review $\sle$ measures in multiply connected domains, as discussed in \cite{greg_annulus}.
Assume $\kappa\leq 4$ and let $z,w$ be analytic boundary points of a domain $D\subset \mathbb{C}$. Let $\mu_D(z,w)$ denote the chordal $\sle$ measure from $z$ to $w$ in $D$
which is defined for simply connected $D$ and we now extend
to the multiply connected case.
If $D$ is a (possibly multiply connected) domain and 
 $D'\subset D$ is simply connected, define $\mu_D(z,w;D')$ by
\[
\frac{d \mu_D(z,w;D')}{d \mu_{D'}(z,w)}(\gamma) = 1\{\gamma
\subset D'\}\,  \exp\left\{-\frac{\cc}{2}m_D(\gamma,D\sm D')\right\}.
\]
The measure
$\mu_D(z,w)$ is defined to be  the  unique measure on continuous curves $\gamma$ connecting $z,w$ in $D$ such that for every simply   connected domain $D'\subset D$, $\mu_D(z,w)$ restricted to the curves $\gamma\subset D'$ is $\mu_D(z,w;D')$. It is not hard to show that $\mu_D(z,w)$ satisfies the domain Markov property and conformal covariance. Moreover, if $\tilde D\subset D$ is another domain  that agrees with $D$ in neighborhoods of $z,w$, then \eqref{genres}
holds:
\begin{equation}\label{theq6}
\frac{d\mu_{\tilde D}(z,w)}{d\mu_{D}(z,w)}(\gamma) =1\{\gamma \subset \tilde D\} \, \exp\left\{\frac{\cc}{2}m_{D}(\gamma,D_2\sm D_1)\right\}. 
\end{equation}
 In addition, if $\kappa\leq 8/3$, then  $\Psi_D(z,w)<\infty$ for any multiply connected domain $D$. It is still unknown if $\Psi_D(z,w)<\infty$ for $8/3<\kappa\leq 4$ for all 
multiply connected domains.  However, this is known
in the case of annulus $\sle$ which we discuss now. 

Let $\ubar\in C_0$ and $\wbar\in C_r$ and for ease
assume $\ubar = 1, \wbar = e^{-r + i\theta}$.  If
$  \gamma$ is a simple curve from $1$ to $\wbar$ in 
$A_r$, then as in section \ref{secnotation},
let $\eta$ be the 
  continuous curve in $S_r$
   from $0$ to $\theta + 2\pi k  + ir$
for some integer $k$ such that $\psi \circ \eta =
  \gamma$.    We write
\[   \mu_{A_r}(1,\wbar) = \sum_{k \in \Z}
     \psi \circ \nu_{S_r}(0,\theta + 2\pi k  + ir), \]
     for measures $\nu_{S_r}(0,x + ir)$ that
 are absolutely continuous with respect to
 $\mu_{S_r}(0,x + ir)$ with Radon-Nikodym derivative
 \[     I(\eta) \, \exp\left\{ \frac \cent 2[
     m_{S_r}^*(\eta)- \hat m_r] \right\}. \]
 Here,
 \begin{itemize}
 \item  $I(\eta)$ is the indicator function that
 $\eta$ does not hit any of its $2\pi k $ translates,
 that is, $\psi \circ \eta$ is simple;
 \item   $\hat m_r$ is the measure of loops   in $A_r$ of
 nonzero winding number  (these are the loops
 in $A_r$ that are not of the form
 $\psi \circ \ell$ for a loop in $S_r$);
 \item   
 $m_{S_r}^*(\eta)$ is the measure of loops in $S_r$
 that intersect $\eta$ but intersect
  a $2k\pi$ translate of $\eta$ before the intersection
  (this takes care of the multiple counting of loops when
 one takes $\psi \circ \ell$ for loops in $S_r$). The time
 ordering of the curve $\eta$ is used to determine 
 ``before''. 
 \end{itemize}
 Using this comparison, it can be shown that $\Psi_{A_r}(\bar u,\bar w) < \infty$ for all $\kappa \leq 4$ and that
 the function is $C^2$ (actually more) in $\bar u, \bar w$
 and $C^1$ in $A_r$ when the  curve has a capacity
 parametrization.

Let $\ubar\in C_0$ and $\wbar\in C_r$. Then there exist positive constants $c_*,\,q$ such that as $r\to\infty$,
\begin{equation}\label{ucsd1}
\Psi_{A_r}(\ubar,\wbar) = c_*\,r^{\cc/2}\,e^{(b-\tilde{b})r}\left[1+\bigo(e^{-qr})\right].
\end{equation}
(See Theorem 7.6 in \cite{greg_annulus}).  In particular,
the aymptotics are independent of $\ubar,\wbar$,
as would be expected.    Suppose $\gamma_t\subset A_r$  is a simple curve with $\,\gamma(0+) = \ubar$.  Then
\begin{equation}\label{ucsd1.5}
\frac{d\mu_{A_r}(\ubar,\wbar)}{d\mu_\disk(\ubar,0)}(\gamma_t) = \frac{|\gg'(\wbar)^b|}{\gg'(0)^{\tilde{b}}}\exp\left\{\frac{\cc}{2}m_\disk(\gamma_t,C_r)\right\}\Psi_{A_{r(t)}}(\Ubar_t,\hbar_t(\wbar)).
\end{equation}
 By using the last formula and \eqref{ucsd1}, we can see that if $\gamma_t$ has the radial parametrization (so that $\gg'(0) = at/2$), then
\begin{equation}\label{ucsd2}
\frac{d\mu_{A_r}(\ubar,\wbar)}{d\mu_\disk(\ubar,0)}(\gamma_t) = c_*\,r^{\cc/2}\,e^{(b-\tilde{b})r}\left[1+\bigo(e^{-q(r-at/2)})\right].
\end{equation}

\subsection{Defining multiple-paths \texorpdfstring{$\sle$}{LG}}\label{defs2}
Suppose 	$\textbf{z} = (z^1,z^2,\ldots,z^n)$ and $\textbf{w} = (w^1, w^2, \ldots, w^n)$ are distinct analytic boundary points of a domain $D$.
For $1\leq j\leq n$, let $\gamma^j$ be a $\sle$ path from $z^j$ to $w^j$ in $D$ with corresponding $\sle$ measure $\mu_{D}(z^j,w^j)$.  Recall that
\[
\Psi_{D}(z^j,w^j)= \| \mu_{D}(z^j,w^j) \|.
\]
For a measure $\mu$, we will use $\mu^\#$ to represent the probability measure $\mu/\|\mu\|$, given that $\|\mu\|<\infty$.
We define multiple-paths $\sle$ measure in $D$ similar to \cite{multiple,michael}.
\begin{definition} \label{def1} For $\kappa\leq 4$, we define  $\mu_{D}(\textbf{z},\textbf{w})$  to be the measure on $n$-tuple of paths ${\bgamma} = (\gamma^1, \ldots,\gamma^n)$ that is absolutely continuous with respect to the product measure $\mu_\text{prod}(\textbf{z},\textbf{w}):=\mu_{D}(z^1,w^1)\times \ldots \times \mu_{D}(z^n,w^n)$ with Radon-Nikodym derivative 
 \[
Y(\bgamma) = I(\bgamma)\exp\left\{\frac{\cc}{2}\sum_{j = 2} ^ n m[K_j(\bgamma)]\right\}.
 \]
Here $I(\boldsymbol\gamma)$ is the indicator function of the event that for all $i\neq j$, $\gamma^i \cap \gamma^j = \emptyset$ and
$m[K_j(\bgamma)]$ is the Brownian loop measure of loops that intersect at least j of the paths.
As before, the partition function of the measure $\mu_{D}(\textbf{z},\textbf{w})$ is the total mass
\[
\Psi_{D}(\textbf{z}, \textbf{w}) = \|\mu_{D}(\textbf{z},\textbf{w})\|.
\]
\end{definition}
Note that $\mu_{D}(\textbf{z},\textbf{w})=0$ if there exists $k$ such that $z^k,w^k$ are not on the boundary of the same connected component of 
\[
D\sm \bigcup_{j\neq k} \gamma^j.
\]
Moreover, it is clear from the definition that if $\sigma:\{1,2,\ldots,n\}\to \{1,2,\ldots,n\}$ is a permutation, then $Y(\bgamma)=Y(\bgamma^\sigma)$ and $\Psi_{A_r}(\textbf{z},\textbf{w}) = \Psi_{A_r}(\textbf{z}^\sigma,\textbf{w}^\sigma)$.

Let $\bgamma=(\bgamma',\gamma^n),\, \textbf{z}'=(z^1,z^2,\ldots,z^{n-1}),\,\textbf{w}'=(w^1,w^2,\ldots,w^{n-1}),$ and 
\[
D_k = D\sm \bigcup_{j=1}^{k-1} \gamma^j.
\]
\begin{lemma}\label{lemdef1}
We have
\begin{equation}\label{eqdef1}
\sum_{j=2}^n m[K_j(\bgamma)] = \sum _{j=2}^nm_{D}\left(\gamma^j,\,\gamma^1\cup\cdots\cup\gamma^{j-1}\right).
\end{equation}

\end{lemma}
\begin{proof}
We can see that
\[
\sum_{j=2}^n m[K_j(\bgamma)] = \sum _{j=2}^{n-1}m[K_j(\bgamma')] + m_{D}(\gamma^n, \gamma^1\cup\cdots\cup\gamma^{n-1})
\]
(This is straightforward. See Lemma 1 in \cite{multiple} for a proof). The proof follows from this and induction.
\end{proof}
The right-hand side of \eqref{eqdef1} was used in \cite {greg_annulus} to define  multiple-paths $\sle$ measure. The last lemma shows that the two definitions are equivalent. 
Suppose $z^1,\ldots,z^n\in C_0$ and $w^1,\ldots,w^n\in C_r$. 
The following important properties can be seen from Lemma \ref{lemdef1}. Similar properties were  stated in \cite{michael} for  multiple-paths $\sle$ in simply connected domains.
\begin{itemize}
\item \textbf{Marginal Measure:} Let $\mu'_{A_r}(\textbf{z}',\textbf{w}')$ be the marginal measure on $\bgamma'$ induced by $\mu_{A_r}(\textbf{z},\textbf{w})$. Then
\[
\frac{d\mu'_{A_r}(\textbf{z}',\textbf{w}')}{d\mu_{A_r}(\textbf{z}',\textbf{w}')}(\bgamma') = H_{D_n}(z^n,w^n)^b.
\]
More generally, if $k<n$ and $\mu'_{A_r}((z^1,\ldots,z^{k-1}),(w^1,\ldots,w^{k-1}))$ is the marginal measure on $(\gamma^1,\ldots,\gamma^{k-1})$ induced by $\mu_{A_r}(\textbf{z},\textbf{w})$, then 
\[
\frac{d\mu'_{A_r}((z^1,\ldots,z^{k-1}),(w^1,\ldots,w^{k-1}))}{d\mu_{A_r}((z^1,\ldots,z^{k-1}),(w^1,\ldots,w^{k-1}))}(\gamma^1,\ldots,\gamma^{k-1}) = \Psi_{D_{k}}((z^k,\ldots,z^n),(w^k,\ldots,w^n)).
\]
\item \textbf{Conditional distribution:} Given $\bgamma'$, the conditional distribution of $\gamma^n$ is $\mu^\#_{D_n}(z^n,w^n)$.  More generally, conditioned on $(\gamma^1,\gamma^2,\ldots,\gamma^{k-1})$, the probability distribution of $(\gamma^{k},\ldots,\gamma^n)$ is $\mu^\#_{D_k}((z^k,\ldots,z^n),\,(w^k,\ldots,w^n)).$
\end{itemize}

Let 
\[
\tilde{\Psi}_{A_r}(\textbf{z},\textbf{w}) = \frac{\Psi_{A_r}(\textbf{z},\textbf{w})}{\prod_{j=1}^n\Psi_{A_r}(z^j,w^j)}.
\]
We will write $\EE_{\text{prod}}$ for the expectation with respect to the product measure $\mu^\#_\text{prod}$. Using \eqref{theq6}, it is easy to see that
\begin{equation}\label{eqdef2}
\EE_\prd\left[Y(\bgamma)\middle|\bgamma'\right] = Y(\bgamma')\frac{H_{D_n}(z^n,w^n)^b}{\Psi_{A_r}(z^n,w^n)}.
\end{equation}
Here, we are using the fact that $D_n$ is simply connected and $\Psi_{D_n}(z^n,w^n) = H_{D_n}(z^n,w^n)^b$. 
Let $m^*(r)$  be the Brownian loop measure of the loops in $A_r$ with nonzero winding number. One can see that $0<m^*(r)<r/6$ (e.g.  Proposition 3.11 in \cite{greg_annulus}). 
 Since $D_n\subset A_r$ is simply connected,
\[
\frac{d\mu_{D_n}(z^n,w^n)}{d\mu_{A_r}(z^n,w^n)}(\gamma^n) = 1\{\gamma^n\subset D_n\} \exp\left\{\frac{\cc}{2}m_{A_r}(\gamma,\,A_r\sm D_n)\right\}.
\]
If $\kappa\leq 8/3$, then $\cc\leq 0$ and 
\begin{equation*}\label{eqdef3}
H_{D_n}(z^n,w^n)^b=\Psi_{D_n}(z^n,w^n)\leq e^{\frac{c}{2}m^*(r)}\Psi_{A_r}(z^n,w^n).
\end{equation*}
In this case, we can use \eqref{eqdef2} to see
\[
\tilde \Psi_{A_r}(\textbf{z},\textbf{w})=\EE_\prd \left[\EE_\prd\left[Y(\bgamma)\middle|\bgamma' \right]\right]\leq e^{\frac{\cc}{2}m^*(r)} \tilde{\Psi}_{A_r}(\textbf{z}',\textbf{w}')\leq e^{(n-1)\frac{\cc}{2}m^*(r)}\leq 1.
\]
We still do not know if this holds for $8/3<\kappa\leq 4$.

\section{Single-Path Results}\label{secsingle}
In this section, we prove the necessary estimates for  our main results in the next two sections. 
We start with  proving a number of deterministic estimates. 
The main tools are the Koebe-$1/4$ theorem and the distortion estimates. 
Parts of these results already appear in the literature (e.g. in \cite{greg_book, greg_annulus}), but we prove them here for completeness. 
Next, we derive a version of  the annulus Loewner equation motivated by the corresponding choices is simply connected domains.
 Our approach is similar to \cite{bauer,bauer2}, but our choices for uniquely determining $\hbar_t$ in \eqref{eq2.5} are different. 
This version of the annulus Loewner equation is also used in the study of annulus $\sle$ in \cite{greg_annulus, dapeng_2014}.
Finally, we give a proof for  \eqref{ucsd1.5}.
\subsection{Conformal transformations}  \label{confsec}

\begin{lemma}\label{lemback0}
Let $z'\in C_0,\,z\in C_r,\,\bar{z}\in A_r$. Then 
\[
H_{A_r}(\bar{z},z') = \frac{r+\log|\bar{z}|}{2r} + O(|\bar{z}|),
\]
\[
H_{A_r}(\bar{z},z) = \frac{-e^{r}\,\log|\bar{z}|}{2r} + O(|\bar{z}|^{-1}),
\]
\[
H_{A_r}(z',z) = e^{r}\left[\frac{1}{2r}+\bigo(e^{-r})\right].
\]
\end{lemma}
\begin{proof}
By symmetry, it is enough to prove the lemma for $z'=1$. 
Note that
\begin{equation}\label{theq1}
H_{A_r}(1,z) = \frac{1}{\pi}\int _{C_{r-1}} H_{A_r\cap \disk_{r-1}}(z,v)\,H_{A_r}(v,1)\,|dv|.
\end{equation}
If $v\in C_{r-1}$, we can see from  the strong Markov property that
\[
H_{A_r}(v,1)=H_\disk(v,1) - \frac{1}{\pi}\int_{C_r} H_{A_r}(v,w)H_{\disk}(w,1)\,|dw|.
\]
Using the exact form of the Poisson kernel in $\disk$, we can see that for any $w\in \disk$
\[
H_\disk(w,1) = \frac{1}{2}[1+\bigo(|w|)].
\]
Therefore,
\begin{align*}
H_{A_r}(v,1) &= \frac{1}{2}[1+\bigo(e^{-r})] - \frac{1}{2\pi}[1+\bigo(e^{-r})] \int_{C_r} H_{A_r}(v,w) |dw|\\
& = \frac{1}{2}[1+\bigo(e^{-r})] - \frac{r-1}{2r}[1+\bigo(e^{-r})]\\
& = \frac{1}{2r} + \bigo(e^{-r}).
\end{align*}
Similarly,
\[
H_{A_r}(\bar{z},1) = \frac{1}{2}[1+\bigo(|\bar{z}|)] +\frac{\log |\bar{z}|}{2r}[1+\bigo(e^{-r})] = \frac{r+\log|\bar{z}|}{2r} + O(|\bar{z}|),
\]
which proves the first equality in the statement of the lemma. The second equality follows from this and the  transformation $z\mapsto e^{-r}/z.$
Using \eqref{theq1}, 
\begin{align*}
H_{A_r}(1,z) &= \left[\frac{1}{2r}+\bigo(e^{-r})\right] \frac{1}{\pi}\int_{C_{r-1}}H_{A_r\cap \disk_{r-1}}(z,v)|dv| \\
&= \left[\frac{1}{2r}+\bigo(e^{-r})\right] \mathcal{E}_{A_r\cap\disk_{r-1}}(z,C_{r-1})\\&
 = e^{r}\left[\frac{1}{2r}+\bigo(e^{-r})\right],
\end{align*}
which proves the last estimate.
\end{proof}
\begin{lemma}\label{lemback1}
Suppose $\gamma:(0,t_\gamma)\in\disk\sm\{0\}$ is a simple curve with $\gamma(0+) = 1$. Let $\Phi$ be a conformal transformation in a neighborhood of $1$ that locally maps $\disk$ to $\hp$ and $C_0$ to $\mathbb{R}$. Let $B_t$ be a complex Brownian motion and define $\tau_t$ to be the first time $B_t$ exits $\disk\sm\gamma_t$. Then at $t=0$,
\[
\partial_t\EE^0[\log|B_{\tau_t}|] =-\partial_t\log \gg'(0)= \frac{-1}{2|\Phi'(1)|^2}\partial_t\hcap[\Phi(\gamma_t)].
\]
\end{lemma}
\begin{proof}
By Schwarz lemma, $f_t(z)=\log(\gg(z)/z)$ is a well-defined bounded analytic function on $\disk\sm \gamma_t$ with $f_t(0) = \log\gg'(0)\in\mathbb{R}$. Hence, $\Re[f_t(z)]$ is a bounded harmonic function and
\[
\Re[f_t(z)] = \EE^z\left[\Re[f_t(B_{\tau_t})]\right] = -\EE^z[\log|B_{\tau_t}|].
\]
In particular,
\[
\log \gg'(0) = -\EE^0[\log|B_{\tau_t}|].
\]
Note that $\phi(z) = i(1-z)/(1+z)$ is a conformal transformation from $\disk$ to $\hp$ with $\phi(0) = i$. Let $\tilde{\gamma}_t = \phi(\gamma_t)$ and define $ g_t:\hp\sm\eta_t\to \hp$  to be the unique conformal transformation with ${g}_t (z) =  z+\smo(1)$ as $z\to\infty$. Let 
\[
\phi_t(z)=\frac{z-\Re[{g}_t(i)]}{\Im[{g}_t(i)]},
\] be a conformal transformation $\phi_t:\hp\to\hp$ satisfying $\phi_t\circ{g}_t(i) = i$.
Then 
\begin{equation}\label{eqdef4}
\gg'(0) = |\partial_z\,\phi^{-1}\circ\phi_t\circ{g}_t\circ\phi(0)| = \frac{{g}'_t(i)}{\Im[{g}_t(i)]}.
\end{equation}
If $|z|>2r_t$, then \eqref{tchy} and the fact that  $\Im[z-g_t(z)]$ is a bounded harmonic function imply that 
\begin{equation}\label{theq0.2}
\Im[z]-\Im[{g}_t(z)] = \Im[-1/z]\,\hcap[{\eta}_t]\,\left[1+\bigo(r_t/|z|)\right].
\end{equation}
Let $f_t(z) = g_t(z) - z - \hcap[\eta_t]/z$ and $v_t(z) = \Im[f_t(z)]$. 
Using \eqref{theq0.2}, we can see that there exists a constant $c$ such that for every $|z|>3\,r_t/2$,
\[
|v_t(z)| \leq c\,r_t\,\hcap[\eta_t] \,|z|^{-2}.
\]
By the mean value property of harmonic functions we can see that for $|z|>2r_t$,
\[
|f'_t(z)| = \lvert g'_t(z) - 1 + \hcap[\eta_t]/z^2\rvert \leq c\,r_t\,\hcap[\eta_t] \,|z|^{-3}.
\]
It follows from this and \eqref{theq0.2} that for small enough $t$,
\begin{align*}
\Im[{g}_t(i)] &= 1 - \hcap[{\eta}_t] + \bigo(r_t\hcap[{\eta}_t]),\\
|{g}'_t(i)| = & 1 + \hcap[{\eta}_t] + \bigo(r_t\hcap[{\eta}_t]).
\end{align*}
Substituting these into \eqref{eqdef4} gives
\[
\gg'(0) = 1+\hcap[\eta_t]\,\left[2 + \bigo(\hcap[\eta_t])\right].
\]
In addition, \eqref{eqhalf} implies that at $t=0$,
\[
\partial_t\hcap[\Phi(\gamma_t)] = (\Phi\circ \phi^{-1})'(0)^2\,\partial_t\hcap[{\eta}_t],
\]
from which the result follows.
\end{proof}
\begin{lemma}\label{lemback4}
Suppose $D\subset \disk$ is a simply connected domain with $C_r\subset D$. Let $K = \disk\sm D$ and $\hat{r}=2\pi/\mathcal{E}_{A_r\cap D}(C_r, \partial D).$  Define $g:D\to\disk$ to be the unique conformal transformation with $g(0) = 0,\,g'(0)>0$. Then 
\begin{equation}\label{theq2}
\hat{r} = r-\log g'(0) +\bigo(e^{-r+\log g'(0)}),
\end{equation}
and
\[
m_\disk(K, C_r) = \log({r}/{\hat r})+\bigo(e^{-\hat{r}}).
\]
\end{lemma}
\begin{proof}
Let $s=\log (4\,g'(0))$ and note that by Koebe-$1/4$ theorem, $C_s\subset D$. Without the loss of generality, assume $s<r$. Distortion theorem for the transformation $g$ restricted to $\disk_s$ implies that there exists a constant $c$ such that for any $|w|\leq e^{-r}$,
\[
|g'(w)-g'(0)| \leq cg'(0) e^{-r+s}.
\]
By taking integral of this we get
\[
|g(w)-w\,g'(0)| \leq c|w|\,g'(0)\,e^{-r+s},
\] 
from which \eqref{theq2} follows.

For $x\in C_t$, let $\Gamma_{A_t}(x,K)$ be the bubble measure of loops rooted at $x$ that intersect $K$ (That is, the measure induced by Brownian excursions rooted at $x$ in $A_t$ restricted to the loops that intersect $k$. See \cite{greg_loop} or Section 5.5 in \cite{greg_book} for more details). Then
\begin{equation}\label{theq3}
m_\disk(K,C_r) = \frac{1}{\pi}\int_r^\infty\int_0^{2\pi} e^{-2t}\,\Gamma_{A_t}(e^{-t+i\theta},K)\,d\theta\,dt.
\end{equation}
For $z\in A_r$, the probability that a Brownian motion starting at $z$ exits $A_r$ at $C_0$ is $1+\log|z|/t$. Using this and the strong Markov property for Brownian motion, we can see that
\begin{align*}
\frac{2\pi}{t}=\mathcal{E}_{A_t}(C_t,C_0) &= \int_{ \partial D}\left[1+\frac{\log|z|}{t}\right]\,\mathcal{E}_{A_t\cap D}(C_t, dz)\\
& = \mathcal{E}_{A_t\cap D}(C_t,\partial D) + \int_{ \disk \cap \partial D}\frac{\log|z|}{t}\,\mathcal{E}_{A_t\cap D}(C_t, dz).
\end{align*}
Using equation \eqref{theq2}, we can write
\[
\mathcal{E}_{A_t\cap D}(C_t,\partial D) =\int_{\partial D}\mathcal{E}_{A_t\cap D}(C_t,dz)= \frac{2\pi}{t-\log g'(0)}\left[1 + \bigo(e^{-t+\log g'(0)})\right].
\]
Therefore,
\[
\int_{\disk\cap \partial D} \frac{-\log|z|}{t}\,\mathcal{E}_{A_t\cap D}(C_t,dz) = \frac{2\pi}{t-\log g'(0)}-\frac{2\pi}{t} + \bigo\left(\frac{e^{-t+\log g'(0)}}{t-\log g'(0)}\right).
\]
Let $V\subset \partial D$ and recall that for any $z\in D$, $\rho_D(z;V)$ denotes the probability that a Brownian motion starting at $z$ exits $D$ at $V$. Since $C_s\subset D$, for any $w\in C_t$ we have 
\begin{align*}
\mathcal{E}_{A_t\cap D}(C_t,V) &= \frac{1}{\pi}\int_{C_t}\int_{C_s}H_{A_t\cap \disk_s}(v,u)\,\rho_D(u;V)\,|du|\,|dv|\\
& = \frac{\left[1+\bigo\left((t-s)e^{-t+s}\right)\right]}{\pi}\int_{C_t}\int_{C_s}H_{A_t\cap \disk_s}(w,u)\,\rho_D(u;V)\,|du|\,|dv|\\
& = e^{-t}\,\mathcal{E}_{A_t\cap D}(w,V)\,[1+\bigo\left((t-s)e^{-t+s}\right)],
\end{align*}
where the second equality follows from Lemma \ref{lemback0}. 

Using this, Lemma \ref{lemback0} and the fact that $\dist(0,\partial D) > 4\,g'(0)$, we can see that for $w\in C_t$,
\begin{align*}
\Gamma_{A_t}(w, K) &= \int_{\disk\cap\partial D} H_{A_t}(z,w) \,\mathcal{E}_{A_t\cap D}(w,dz) \\
& = \frac{e^{2t}}{2\pi}\,\left[1+\bigo\left((t-s)e^{-t+s}\right)\right]\int_{\disk\cap\partial D} \frac{-\log|z|}{2t}\mathcal{E}_{A_t\cap D}(C_t,dz)\\
& = \frac{e^{2t}}{2}\,\left[1+\bigo\left((t-s)e^{-t+s}\right)\right]\left[\frac{1}{t-\log g'(0)}-\frac{1}{t} + \bigo\left(\frac{e^{-t+\log g'(0)}}{t-\log g'(0)}\right)\right].
\end{align*}
Plugging this into \eqref{theq3}, we get
\[
m_\disk(K,C_r) = \log\left(\frac{r}{r-\log g'(0)}\right) + \bigo(e^{-r+s}).
\]
The result follows from this and \eqref{theq2}.
\end{proof}

\begin{lemma}  \label{lemback3}
Suppose $D\subset \disk$ is a simply connected domain such that $\disk\sm D\subset A_r$ and let $K=D\cap A_r$. Define $f:K\to A_{\hat r}$ to be a conformal transformation satisfying $f(C_r) = C_{\hat{r}}$ and let $g:D\to\disk$ be the unique conformal transformation with $g(0) = 0,\, g'(0)>0$. Then for $z\in C_r$,
\[
|f'(z)| = e^{r-\hat{r}}\left[1+O\left(e^{-\hat r}\right)\right]
= g'(0) \left[1+O\left(e^{-\hat r}\right)\right],
\] 
where the error term is independent of $z$.
\end{lemma}

\begin{proof}
  By conformal invariance of the excursion measure,
\begin{equation}\label{eqfirst}
\frac{2\pi}{\hat{r}}=\mathcal{E}_{A_{\hat{r}}}(C_0,C_{\hat{r}}) = \mathcal{E}_{g(K)}(C_0,g(C_r))=\mathcal{E}_K(\partial D,C_r).
\end{equation}
For $z\in C_r$, recall that $\rho_K(z;\partial D)$ is the probability that a Brownian excursion starting from $z$  exits  $K$ at $\partial D$.
Using  \eqref{eqfirst} gives us
\[
\frac{2\pi}{\hat{r}}=\int_{C_r} \rho_K(z;\partial D)|dz|.
\]
Equation  \eqref{eqsecond} indicates that
\begin{equation}\label{eqlem3-1}
\rho_K(z;\partial D) = |f'(z)|\,\frac{e^{\hat{r}}}{\hat{r}}.
\end{equation}
If $s = r-1$, we can  write
\[
\rho_K(z;\partial D) =\frac{1}{\pi}\int_{C_{{r-1}}} H_{\disk_{{r-1}}\cap A_{{r}}}(z,w)\,\rho_K(w;\partial D)\,|dw|. 
\]
If we prove that uniformly over $w\in C_{r-1}$,
\begin{equation}\label{ryan}
\rho_K(w;\partial D) = \frac{1}{\hat{r}}\left[1+\bigo(e^{-\hat{r}})\right],
\end{equation}
then 
\begin{align*}
\rho_K(z;\partial D) &=\frac{1}{\hat r}\left[1+\bigo(e^{-\hat{r}})\right]\, \frac{1}{\pi}\int_{C_{{r-1}}} H_{\disk_{{r-1}}\cap A_{{r}}}(z,w)\,|dw|\\
& = \frac{1}{\hat r}\left[1+\bigo(e^{-\hat{r}})\right]\, \mathcal{E}_{\disk_{r-1}\cap A_r}(z,C_{r-1}) \\
 & = \frac{e^{r}}{\hat r}\left[1+\bigo(e^{-\hat{r}})\right],
\end{align*}
and the first equality in the statement of the lemma follows from \eqref{eqlem3-1}.

To show \eqref{ryan}, let ${s} = \log(4g'(0))$.  Koebe-$1/4$ theorem implies that $C_{ s} \subset D$. It follows from the distortion estimates for the transformation $g$ restricted to $\disk_{ s}$ that there exists a constant $c$ such that for all $|w|=e^{-r+1}$,
\begin{gather*}
|g'(w)-g'(0)|\leq c\,g'(0)e^{-r+s},\\
|g(w) - w\,g'(0)|\leq c\,g'(0) e^{-2r+s},\\
 |g(w)| = g'(0)\,e^{-r+1}\left[ 1 +\bigo(e^{-r+s})\right]  .
\end{gather*}
Using this and \eqref{theq2}, we get
\[
\rho_K(w;\partial D) = \rho_{A_{\hat{r}}}(f(w); C_0) = 1+\frac{\log|f(w)|}{\hat{r}}=\frac{1}{\hat{r}}\left[1+\bigo(e^{-\hat{r}})\right].
\]
The second equality in the statement of the lemma follows from   \eqref{theq2}.
\end{proof}
\begin{corollary}\label{corback1}
Suppose $D\subset \disk$ is a simply connected domain containing the origin. Let $K=\disk \sm \bar D$ and choose $\hat r$ such that there exists a conformal transformation $f:K\to A_{\hat{r}}$  satisfying $f(C_0) = C_{0}$. Then for $z\in C_0$,
\[
|f'(z)| = \left[1+O(e^{-\hat{r}})\right].
\]
\end{corollary}
\begin{proof}
Since $0\in D$, we can find $0<r<\infty$ such that $C_r\subset D$. Let 
\[
g_1(z) = \frac{e^{-r}}{z},\qquad g_2(z) = \frac{e^{-\hat{r}}}{z},
\]
and define $\hat{f}=g_2\circ f\circ g_1$ to be a conformal transformation from $g_1(D)$ onto $A_{\hat{r}}$ with $\hat f(C_r) = C_{\hat{r}}$. The result follows from the chain rule and Lemma \ref{lemback3}.
\end{proof}
\begin{corollary}\label{corback2}
Recall the assumptions of Lemma \ref{lemback3}. Let $\theta \in [0,2\pi)$ and define the function  $\Theta(\theta) := \arg f(e^{-r+i\theta})$. Then
\[
\Theta'(\theta) = 1 + \bigo(e^{-\hat r}).
\]
\end{corollary}
\begin{proof}
Define the function $\varphi:C_0\to C_0$ with $\varphi(w) = e^{\hat r} f(e^{-r}\,w)$. Using Lemma \ref{lemback3},
\[
|\varphi'(w)| = 1 + \bigo(e^{-\hat r}).
\]
If $w= e^{i\theta}$, then  $\Theta(\theta) = \arg \varphi(w)$ and $|\varphi'(w)| = |\partial_w \,\arg\varphi(w)|.$ Therefore,
\[
\Theta'(\theta) = \partial_\theta \arg \varphi(e^{i\theta}) = ie^{i\theta} \,\partial_w\arg\varphi(w)
\]
and
\[
|\Theta'(\theta)| = |\partial_w\arg\varphi(w)| = |\varphi'(w)| = 1 + \bigo(e^{-\hat r}).
\]
Since $f(C_r) = C_{\hat r}$, we have $\Theta'(\theta)>0$ and the result follows.
\end{proof}
\begin{lemma}\label{lemback2}
Suppose $\gamma_t\subset A_r$ is a simple curve with $\gamma(0+)\in C_0$. If $\gamma_t$  has the radial parametrization, then 
\[
\dot{r}(t) = -\frac{a \bar\phi_t'(\bar\xi_t)^2}{2}.
\]

\end{lemma}
\begin{proof}
We first prove the lemma for the derivative at $t=0$. Since $\mathcal{E}_{A_r\sm\gamma_t}(C_0\cup\gamma_t,\,C_r) = 2\pi/r(t)$,  we have
\begin{equation}\label{eqdef4.5}
\partial_t \mathcal{E}_{A_r\sm\gamma_t}(C_0\cup\gamma_t,\,C_r) = \dot{r}(t)\frac{-2\pi  }{ r^2}.
\end{equation}
Since for all $z\in A_r\sm \gamma_t$,
\[
\rho_{A_r\sm\gamma_t}(z;\,C_0\cup\gamma_t) = \rho_{A_{r(t)}}(\hbar(z);\,C_{0}),
\]
we can  write
\[
 \mathcal{E}_{A_r\sm\gamma_t}(C_0\cup\gamma_t,\,C_r) = \int _{C_r} \partial_{\textbf{n}_z}\left(1 + \frac{\log |\hbar_t(z)|}{r(t)}\right)|dz|,
\]
where $\textbf{n}_z$ denotes the inward normal derivative.
Suppose $\tau_t$ is the first time a Brownian motion $B_t$ exits $A_r\sm\gamma_t$ and $\sigma_t$ is the first time $B_t$ exits $\disk\sm\gamma_t$.
Since $\log |\hbar_t(z)| - \log |z|$ is a bounded harmonic function on $A_r\sm\gamma_t$, we have
\begin{align*}
\log |\hbar_t(z)| -\log|z|& = (r-r(t))\rho_{A_r\sm\gamma_t}(z;\,C_r) - \EE^z[\log |B_{\tau_t}|;\,\tau_t = \sigma_t]\\
& = - (r - r(t)) \frac{\log|\hbar_t(z)|}{r(t)} - \EE^z[\log |B_{\tau_t}|;\,\tau_t = \sigma_t].
\end{align*}
 Hence,
\begin{equation}\label{eqdef5}
{r}\mathcal{E}_{A_r\sm\gamma_t}(C_0\cup\gamma_t,\,C_r) =  2\pi-\int_{C_r} \partial_{\textbf{n}_z} \EE^z[\log |B_{\tau_t}| ;\,\tau_t = \sigma_t]|dz|.
\end{equation}
 Assume $t$ is small enough so that $\diam[\gamma_t]<1/10$ and therefore $\gamma_t\subset A_1$. By conditioning on the first time the Brownian motion hits $C_1$ we get 
\begin{align*}
\int_{C_r} \partial_{\textbf{n}_z} \EE^z[\log |B_{\tau_t}| ;\,\tau_r = \sigma_t]\,|dz| &= \frac{ e}{r-1}\int_{C_1}\EE^z[\log |B_{\tau_t}|;\,\tau_t=\sigma_t ]|dz|,\\
&= \frac{2\pi}{r-1} \,\EE[\log |B_{\tau_t}|;\,\sigma_t = \tau_t].
\end{align*}
Here, $\EE$ denotes the expectation with respect to a Brownian motion starting uniformly on $C_1$.
 Note that 
\begin{align*}
\EE[\log |B_{\tau_t}|;\,\sigma_t = \tau_t] = \EE[\log |B_{\sigma_t}|] - \EE[\log |B_{\sigma_t}|;\,\tau_t <\sigma_t].
\end{align*}
Let $d_t = 2\,\diam[\gamma_t]$ and note that for small enough $t$, $\gamma_t\cap \disk_{d_t}=\emptyset$. If $T_s$ denotes the first time that a Brownian motion started uniformly on $C_1$ exits $\disk_{s}$, then
\[
\EE[\log |B_{\sigma_t}|;\,\tau_t <T_{d_t}]\leq \EE[\log |B_{\sigma_t}|;\,\tau_t <\sigma_t]\leq \EE[\log |B_{\sigma_t}|;\,\tau_t <T_0].
\]
Conditioned on $\tau_t <T_{d_t}$ (or conditioned on $\tau_t<T_0$), $B_{\tau_t}$ is uniformly distributed on $C_r$. Moreover, 
\[
\PP[\tau_t <T_{d_t}] = \frac{1}{r}\, [1+\bigo(d_t)],\qquad \PP[\tau_t <T_{0}] = \frac{1}{r}.
\]
 Therefore,
\[
\EE[\log |B_{\sigma_t}|;\,\tau_t <\sigma_t]  = \frac{1}{r}\,\EE[\log |B_{\sigma_t}| ]\,[1+\bigo(d_t)].
\]
and
\[
\EE[\log |B_{\tau_t}|;\,\sigma_t = \tau_t] = \EE[\log|B_{\sigma_t}|] \,\left[\frac {r-1} r + \bigo(d_t)\right].
\]
It follows from \eqref{eqdef4.5} and \eqref{eqdef5} that at $t=0$,
\begin{align*}
\dot{r}(t) =\frac{-r^2}{2\pi}\, \partial_t \mathcal{E}_{A_r\sm\gamma_t}(C_0\cup\gamma_t,\,C_r) = \partial_t \EE[\log |B_{\sigma_t}| ] =  \partial_t \EE^0[\log |B_{\sigma_t}| ].
\end{align*}
Now the proof for $t=0$ follows from this and Lemma \ref{lemback1}. Using \eqref{eqhalf}, we can see that at $s=0$,
\[
\partial_s\hcap[h_t(\eta_{t+s})] = |\bar{\phi}_t(\bar{\xi}_t)|^2\, \partial_s\hcap[\tilde g_t(\eta_{t+s})],
\]
from which we conclude the proof for all $t\geq 0$.
\end{proof}

\subsection{Annulus Loewner equation}\label{seclow}
Let  $\gamma_t\subset A_r,\,\gamma(0+) = \ubar$ be a simple curve with the annulus parametrization.  
Let $\bar h_t : A_r\sm\gamma_t\to A_{r-t}$ and $h_t:S_{r,t}\to S_{r-t}$ be as in section \ref{secnotation}.
 Using conformal invariance, one can see that for $z\in S_r$,
\begin{align*}
H_{S_r}(z,0) &= -\frac{\pi}{2r}\Im\left[\coth\left(\frac{\pi\,z}{2r}\right)\right].
\end{align*}
For $z_0\in\mathbb{R}$ define 
\begin{align*}
 \bar{H}_{S_r}(z,0) := &H_{A_r}(e^{iz},1) = \sum _{k\in\mathbb{Z}} H_{S_r}(z,\,2k\pi),\\
\mathcal{H}_{S_r}(z,0) := &-\frac{\pi}{2r}\coth\left(\frac{\pi\,z}{2r}\right),\\
\bar{\mathcal{H}}_{S_r}(z,0) :=&\sum^{PP}\mathcal{H}_{S_r}(z,\,2k\pi)=\mathcal{H}_{S_r}(z,0) + \sum_{k=1}^\infty\left[\mathcal{H}_{S_r}(z,\,2k\pi) + \mathcal{H}_{S_r}(z,\,-2k\pi)\right],
\end{align*}
\[
\bar H_{S_r}(z,z_0) := \bar H_{S_r}(z-z_0,0),\qquad\mathcal{H}_{S_r}(z,z_0) := \mathcal{H}_{S_r}(z-z_0,0),\qquad\bar{\mathcal{H}}_{S_r}(z,z_0) := \bar{\mathcal{H}}_{S_r}(z-z_0,0) .
\]
We had to be a little careful with the definition of $\bar{\mathcal{H}}_{S_r}(z,0)$ because the real parts  are not absolutely convergent. While $\bar{{H}}_{S_r}(z,0)$ is a $2\pi$-periodic function, it is not hard to see that
\begin{equation*}
\bar{\mathcal{H}}_{S_r}(z+2\pi,0) = \bar{\mathcal{H}}_{S_r}(z,0) + \frac{\pi}{r}.
\end{equation*}
\begin{lemma}\label{lem0.1}
Suppose $D_t\subset S_r$ is a half disk of radius $d_t$ centered at the origin. If $x\in \bar S_r,\,x\neq 0,\,\theta\in(0,\pi)$, then
\[
H_{S_r\sm D_t}(x,d_te^{i\theta}) = 2\,H_{S_r}(x,0)\,\sin\theta[1+\bigo(d_t)],
\]
where the error term is independent of $\theta$.
\end{lemma}
\begin{proof}
We prove the lemma for the case $x\in\mathbb{R}$. The case $x\in S_r$ can be proved in a similar way.
Define $f_t:\hp\sm D_t\to\hp$ with $f_t(z) = z+ d_t^2/z$. Then 
\begin{align*}
H_{S_r\sm D_t}(x,d_te^{i\theta}) &= |f'(x)|\,|f'(d_te^{i\theta})|\,H_{f_t(S_r\sm D_t)}(f_t(x), 2d_t\cos\theta)\\
&=2\sin\theta\,H_{f_t(S_r\sm D_t)}(f_t(x), 2d_t\cos\theta)\,[1+\bigo(d_t^2)].
\end{align*}
Note that $S_{r-d_t^2/r}\subset f_t(S_r\sm D_t)\subset S_r$. Therefore, 
\[
H_{f_t(S_r\sm D_t)}(f_t(x), 2d_t\cos\theta) = H_{S_r}(f_t(x), 2d_t\cos\theta)[1+\bigo(d_t^2)] = H_{S_r}(x, 0)[1+\bigo(d_t)].
\]
\end{proof}
\begin{lemma}\label{lem007}
Let $T$ be the first time a Brownian motion $B$ exits $S_{r,t}$. Then for any $z\in S_{r}$,
\[
\EE^z[\Im[B_T]1\{B_T\in\tilde\eta\}] = \hcap[\eta_t]\,\bar H_{S_r}(z,u)\,[1+\bigo(d_t)],
\]
as $t\to 0$. Moreover, for any $\epsilon>0$, the error term is uniform on $\{z;\,\forall k\in\mathbb{Z},\,|z-2k\pi|>\epsilon\}$.
\end{lemma}
\begin{proof}
Without the loss of generality we assume $u=0,\ubar = 1$. 
Define $d_t = 10\,\diam[\eta_t]$ and let $C_t\subset S_r$ be a half circle of radius $d_t$ centered at the origin. Let $D_t$ denote the unbounded connected component of $S_{r}\sm C_t$. Let $\tau$ be the first time the Brownian motion exits $S_r\sm\eta_t$. Define the  function $f$ on $  D_{t}$ with
\[
f(w)= \EE^w\left[\Im[B_\tau]1\{B_\tau\in\eta_t\}\right].
\] 
If $t$ is small enough so that $z\in D_t$, then
\begin{equation}\label{eq007}
f(z)= \frac{1}{\pi}\int_{C_{t}} H_{ D_{t}}(z,w) f(w)|dw|.
\end{equation}
Using Lemma \ref{lem0.1},  for $w\in C_{t}$ we have
\[
H_{D_{t}}(z,w) = 2\sin\theta_w\,H_{S_r}(z,0)\, [1+ \bigo(d_t)],
\]
where $\theta_w = \arg\,w$.
Let $\sigma$ be the first time the Brownian motion exits $\hp\sm\eta_t$ and define $\tilde{f}(w) = \EE^w[\Im[B_\sigma]]$.
Note that
\[
 \tilde{f}(w)-f(w) =\EE^w[\Im[B_\sigma]\,1\{\tau<\sigma\}].
\]
Since $w\in C_t$, it follows from \eqref{tchy} that 
\begin{align*}
\EE^w[\Im[B_\sigma]\,1\{\tau<\sigma\}] = \hcap[\eta_t]\bigo(d_t).
\end{align*}
Moreover, 
\[
\frac{2}{\pi}\int_{C_t}\sin\theta_w\,\tilde f(w)\,|dw|= \hcap[\eta_t].
\]
Therefore, \eqref{eq007} implies that
\[
\EE^z[\Im[B_\tau]1\{B_\tau\in\eta_t\}] = H_{S_r}(z,0)\hcap[\eta_t][1+\bigo(d_t)].
\]
Since this is true for any $z\in S_r$, we have
\begin{align*}
\EE^z[\Im[B_T]1\{B_\tau\in\tilde\eta_t\}]&=\hcap[\eta_t][1+\bigo(d_t)]\,\sum_{k\in\mathbb{Z}}H_{S_r}(z,2k\pi)  \\ 
&\quad +\bigo(\hcap[\eta_t]^2)\,\sum_{k\in\mathbb{Z}}\sum_{k'\neq k}H_{S_r}(z,2k\pi)H_{S_r}(2k\pi,2k'\pi).
\end{align*}
Since the double sum is finite for any $z\in S_r$, the result follows.
\end{proof}
\begin{lemma}\label{lemmoein}
Suppose $\gamma_t$ has annulus parametrization. If $\Ubar_t = \bar h_t(\gamma(t)),\,U_t = h_t(\eta(t))$, then 
for any $z\in S_{r,t}$ 
\begin{equation}\label{tom}
\partial_t \Im[h_t(z)] = -\frac{\Im[h_t(z)]}{r-t} -2 \bar H_{S_{r-t}}({h_t(z)},U_t).
\end{equation}
Moreover, if there exists  $w\in S_r$ such that $w\not\in\tilde\eta$ and  $\Re[h_{t}(w)]$ is differentiable with respect to $t$ for all $t<r$, then 
\[
\partial_t h_t(z) = -\frac{h_t(z)}{r-t} -2 \bar{\mathcal{H}}_{S_{r-t}}({h_t(z)},U_t) + \beta_t,
\]
for some $\beta_t$ independent of $z$.
\end{lemma}
\begin{proof}
It suffices to prove this for $t=0$.
The function
\[
I_t(z)=\Im[z -h_t(z)]
\]
 is a bounded harmonic function on $S_{r,t}$. Considering the values of $I_t(z)$ at the boundaries, we have
\begin{align*}
I_t(z) &= t\,\PP^z[B_\tau\in \mathbb{R}+ir]+ \EE^z\left[\Im[B_\tau]1\{B_\tau\in\tilde \eta_t\}\right]\\
&= t\,\frac{\Im[h_t(z)]}{r-t} + \EE^z\left[\Im[B_\tau]1\{B_\tau\in\tilde \eta_t\}\right],\numberthis\label{eq0025}
\end{align*}
where $\tilde{\eta}_t$ is defined in \eqref{eq2.55} and $\tau$ is the first time  Brownian motion $B$ exits $S_{r,t}$. It follows from Lemma \ref{lemback2} that under the annulus parametrization,
\[
\partial_t\,\hcap[\eta_t]|_{t=0} = 2.
\]
Equation \eqref{tom} follows from this, Lemma \ref{lem007} and \eqref{eq0025}.

To see the second equality in the statement of the lemma, define 
\[
f_t(z) =  \frac{r\,h_t(z)}{r-t}-z+ \hcap[\eta_t]\bar{\mathcal{H}}_{S_{r}}(z,u),
\]
and let $v_t(z) = \Im[f_t(z)]$. Then by using Lemma \ref{lem007} and \eqref{eq0025} we can see that for any $\epsilon>0$, there exists a constant $c_*$ such that for all $\{z\in S_r;\, \forall k,\, |z-2k\pi-u|>\epsilon\}$,
\[
|v_t(z)|<c_*\,d_t\,\hcap[\eta_t]\,\bar{H}_{S_r}(z,u).
\]
Since $v_t(z)$ is harmonic, there exists  $c=c(z)$ such that
\[
|\nabla v_t(z)|<c\,d_t\,\hcap[\eta_t]\,\bar{H}_{S_r}(z,u).
\]
Therefore,
\[
|f'_t(z)|=|\nabla v_t(z)|<c\,d_t\,\hcap[\eta_t]\,\bar{H}_{S_r}(z,u).
\]
Define $a_k = w+2k\pi $ and let $k^* = \arg\min_k |z-a_k|$.
Since $\bar{H}_{S_r}(z,u)$ is uniformly bounded on $\{z\in S_r;\, \forall k,\, |z-2k\pi-u|>\epsilon\}$ and $h_t(z),\bar{\mathcal{H}_{S_r}}(z,0)$ are quasi-periodic functions, for some constant $C$
\begin{align*}
\lvert f_t(z)-f_t(w)+f_t(w)-f_t(a_{k^*})\rvert&=\lvert f_t(z)-f_t(w) - \frac{2k^*\,t\pi}{r-t}-\frac{k^*\pi}{r}\hcap[\eta_t]\rvert\\
&<C\,d_t\,\hcap[\eta_t].
\end{align*}
Since $f_t(w)$ is differentiable with respect to $t$, we get
\[
\partial_t f_t(z)=\lim _{t\downarrow0}\frac{f_t(z)}{t} = \partial_t f_t(w).
\]
Therefore, $h_t(z)$ is differentiable with respect to $t$ and the second equality follows.

\end{proof}
\begin{lemma}\label{lemlink}
For any continuous path $\gamma$ with the annulus parametrization, there exists a collection of transformations $h_t:S_{r,t}\to S_{r-t}$ such that
\[
\partial_t h_t(z) = -\frac{h_t(z)-U_t}{r-t} -2 \bar{\mathcal{H}}_{S_{r-t}}({h_t(z)},U_t).
\]
\end{lemma}
\begin{proof}
Choose $w\in S_r$ such that $w\not\in \tilde{\eta}$. Let $h^*_t:S_{r,t}\to S_{r-t}$ be a conformal transformation and let $\U_t^* = h^*_t(\eta_(t))$. We can assume $h^*_t(w)$ is  differentiable with respect to $t$ (otherwise, we consider $h^*_t(w) + c_t$ for an appropriate $c_t\in\mathbb{R}$). Define
\[
h_t(z) = h^*_t(z) - \Re[h^*_t(w)] + \int_0^t\frac{U_s^*-\Re[h^*_s(w)]}{r-s}\,-2\,\Re[\bar{\mathcal{H}}_{S_{r-s}}(h^*_s(w),U^*_s)]\,ds.
\]
Note that this is well defined for all $t<r$ because $\bar{\mathcal{H}}_{S_{r-s}}(h^*_s(w),U^*_s)$ is a continuous function of $s$.
Using Lemma \ref{lemmoein}, we have
\[
\partial_th_t^*(z) = \partial_th^*_t(w) + \frac{h^*_t(w)-h^*_t(z)}{r-t} - 2[\bar{\mathcal{H}}_{S_{r-t}}(h^*_t(z),U_t)-\bar{\mathcal{H}}_{S_{r-t}}(h^*_t(w),U_t)].
\]
Therefore,
\[
\partial_t \Re[h_t(z)] = -\frac{\Re[h^*_t(z)]-U^*_t}{r-t} - 2\,\bar{\mathcal{H}}_{S_{r-t}}(h^*_t(z),U^*_t).
\]
Using this, the result follows from  the fact that $h_t(z)-U_t = h^*_t(z) - U^*_t$ and \eqref{tom}.
\end{proof}

In \cite{komatu,bauer2}, $h_t(z)$ is  specified by requiring $\hbar_t(e^{-r}) = e^{-(r-t)}$ and $\beta_t$ is determined according to this condition. 
Instead, we have uniquely specified $h_t(z)$ by requiring $\beta_t = U_t/(r-t)$ and $h_0(z) = z$ in Lemma \ref{lemlink}.
This is equivalent to requiring that 
\[
\partial_t \Re[h_t(z)] = 0, \text{ for }\, \{z\in S_{r,t};\,\exists k\in\mathbb{Z}\,\,\Re[h_t(z)] = U_t + 2k\pi\},\]
which is analogous to the usual conditions for chordal Loewner equation in $\hp$. 
We summarize our discussion with the following proposition.
\begin{proposition}\label{proplow}
For $z\in S_r,\,\zbar = e^{iz},\,x\in \mathbb{R}$ define
\begin{equation}\label{eq5.2}
\HH_r(z) = -\frac{z}{r} - 2\bar{\mathcal{H}}_{S_r}(z,0),
\end{equation}
\[
\HH^R_r(x) =\Re[\HH_r(x+ir)]=  -\frac{x}{r} + \frac{\pi}{r}\sum^{PP}_k\tanh\left(\frac{\pi(x+2k\pi)}{2r}\right),
\]
\[
\bar{\HH}_r(\zbar) = i\HH_r(z).
\]
Then
\[
\partial_t h_t(z) = \HH_{r-t}(h_t(z) - U_t),\qquad h_0(z) = z,
\]
\[
\partial_t \Re[h_t(x + ir)]  = \HH^R_{r-t}(\Re[h_t(x+ir)]-U_t),
\]
\begin{equation}\label{eqlow2}
\partial_t \hbar_t(z) = \hbar_t(z) \bar \HH_{r-t}(\hbar_t(z)/\Ubar_t),\qquad \hbar_0(z) =z.
\end{equation}
\end{proposition}
\begin{proof}
This is an straightforward consequence of Lemma \ref{lemlink}.
\end{proof}

The function $\HH_r(z)$ has several interesting properties.  
\begin{itemize}
\item $\HH_r(z)$ is an odd elliptic function. In other words, it is a meromorphic doubly periodic function in $\mathbb{C}$, with periods $2\pi, 2ir$. 
\item Let $\Gamma(r)$ be the measure of Brownian bubbles in $\disk$ that are rooted at 1 and intersect $\disk\sm A_r$. Since  $\disk$ and $A_r$ have smooth boundaries, we can write
\[
\Gamma(r) = \frac{1}{\pi}\int_{C_r}H_{ A_r}(1,w)H_{\disk}(w,1)|dw|.
\]
(The constant $1/\pi$ in the last equation is because of our choice of normalization for the Poisson kernel. It is normalized so that
$H_\disk(0,1) = 1/2$.)
Starting from the definition \eqref{eq5.2}, one can show that
\begin{equation}\label{eqidk}
\HH_r(z) = \frac{2}{z} + z\left(2\Gamma(r) -\frac{1}{r} - \frac{1}{6}\right) + \bigo(|z|^3),\qquad z \to 0.
\end{equation}
See Lemma 3.16 in \cite{greg_annulus} for more details. It follows that $\HH_r(z)$ is analytic on $\mathbb{C}\sm \{2k\pi + i2mr;\,k,m\in\mathbb{Z}\}$ and has poles of degree 1 at points $\{2k\pi + i2mr;\,k,m\in\mathbb{Z}\}$.
\item Let $\wp$ be the Weierstrass elliptic function with periods $2\pi, i2r$. Then 
\[
\partial_z \HH_r(z) = - 2\wp(z) + \zeta_r, 
\]
where $\zeta_r$ is a constant depending on $r$ \cite{dapeng_2004}.
\end{itemize} 
\subsection{Comparing annulus \texorpdfstring{$\sle$}{LG} to radial \texorpdfstring{$\sle$}{LG}}
 Suppose with respect to a probability space $(\Omega,\mathcal{F},\mathbb{P})$, $\gamma_t$ is a radial $\sle$ path from $\ubar$ to 0 in $\disk$. 
We assume $\gamma_t$ has the radial parametrization $\log\gg'(0) = at/2$. 
With respect to $\mathbb{P}$, $\xi_t$ is a Brownian motion and $\bar{\xi}_t=\bar{g}_t({\gamma}(t))=e^{i\xi_t}$.
Denote by $\mu_\disk(\bar{u}, 0)$  the distribution of $\gamma_t$ and let $\mu_{A_r}(\bar u,\bar w)$ be the distribution of  $\sle$ from $\bar u$ to $\bar w$ in $A_r$. Recall our notation in section \ref{secnotation} and let $\Wbar_t =\hbar_t(\wbar)$.
The following result is stated in section 7.2 of \cite{greg_annulus}.  Here, we give a proof with more details.
\begin{proposition} \label{proprad}
Let $\tau_r=\inf\{t;\,\gamma_t\cap C_r\neq \emptyset\}$  and suppose $t<\tau_r$. Then
\begin{equation}\label{eqidk1.9} 
\frac{d\mu_{A_r}(\ubar,\wbar)}{d\mu_\disk(\ubar,0)}(\gamma_t) = \frac{|\hbar'_t(\bar w)|^b|\bar {\phi}'_t(\bar{\xi}_t)|^b}{\bar{g}'_t(0)^{\tilde{b}}}\exp\left\{m_\disk(C_r,\gamma_t)\right\}\Psi_{A_{r(t)}}(\Ubar_t,\Wbar_t).
\end{equation}
\end{proposition}
\begin{proof}
Let $D\subset A_r$ be a simply connected domain that agrees with $A_r$ in neighborhoods of $\ubar,\,\wbar$. 
Let $D_t = \hbar_t(D\sm\gamma_t)$  and $\Psi_{A_{r(t)}}(\Ubar_t,\Wbar_t;\,D_t) = \lvert\lvert\mu_{A_{r(t)}}(\Ubar_t,\Wbar_t;\,D_t)\rvert\rvert$.
It is easy to see that
\[
\frac{d\mu_{A_r}(\ubar,\wbar;\,D)}{d\mu_{A_r}(\ubar,\wbar)}(\gamma_t) = 1\{\gamma_t\subset D\}\,\frac{\Psi_{A_{r(t)}}(\Ubar_t,\Wbar_t; D_t)}{\Psi_{A_{r(t)}}(\Ubar_t,\Wbar_t)}.
\]
By the definition of $\sle$ in $A_r$,
\[
\frac{d\mu_{A_r}(\ubar,\wbar;\,D)}{d\mu_D(\ubar,\wbar)}(\gamma_t) = \exp\left\{-\frac{\cc}{2}m_{A_r}(\gamma_t,A_r\sm D)\right\}1\{\gamma_t\subset D\} \frac{\Psi_{A_{r(t)}}(\Ubar_t,\Wbar_t; D_t)}{\Psi_{D_t}(\Ubar_t,\Wbar_t)}.
\]
Let $|\zbar|=1$ be another analytic boundary point of $D$ and let $\bar Z_t = \hbar_t(\zbar)$. Since $D$ is simply connected, we have
\[
\frac{d\mu_D(\ubar,\wbar)}{d\mu_D(\ubar,\zbar)}(\gamma_t)=\frac{\lvert\hbar_t'(\wbar)\rvert^b\,\Psi_{D_t}(\Ubar_t,\Wbar_t)}{\lvert\hbar_t'(\zbar)\rvert^b\,\Psi_{D_t}(\Ubar_t,\bar Z_t)}.
\]
Moreover, by comparing chordal $\sle$ in $D$ and $\disk$ we get
\[
\frac{d\mu_D(\ubar,\zbar)}{d\mu_\disk(\ubar,\zbar)}(\gamma_t) = \exp\left\{\frac{\cc}{2}m_{\disk}(\gamma_t,\disk\sm D)\right\}1\{\gamma_t\subset D\}\frac{\Psi_{\gg(D\sm\gamma_t)}(\bar\xi_t,\gg(\wbar))}{\Psi_\disk(\bar\xi_t,\gg(\zbar))}.
\]
Finally,  comparing chordal and radial $\sle$ in $\disk$ gives
\[
\frac{d\mu_\disk(\ubar,\zbar)}{d\mu_\disk(\ubar,0)}(\gamma_t) = \frac{\lvert\gg'(\zbar)\rvert^b\,\Psi_\disk(\bar \xi_t,\gg(\zbar))}{\gg'(0)^{\tilde b}}.
\]
Note that
\[
\Psi_{D_t}(\Ubar_t,\bar Z_t) = \lvert\bar\phi'_t(\bar\xi_t)\rvert^{-b}\,\lvert\bar\phi'_t(\gg(\zbar))\rvert^{-b}\,\Psi_{\gg(D\sm\gamma_t)}(\bar\xi_t,\gg(\zbar))
\]
and
\[
m_{\disk}(\gamma_t,\disk\sm D)=m_{A_r}(\gamma_t,A_r\sm D)+m_{\disk}(\gamma_t,C_r).
\]
Therefore,
\[
\frac{d\mu_{A_r}(\ubar,\wbar;\,D)}{d\mu_\disk(\ubar,0)}(\gamma_t) =\frac{|\hbar'_t(\bar w)|^b|\bar {\phi}'_t(\bar{\xi}_t)|^b}{\bar{g}'_t(0)^{\tilde{b}}} \exp\left\{\frac{\cc}{2}m_{\disk}(\gamma_t,C_r)\right\}1\{\gamma_t\subset D\}\,\Psi_{A_{r(t)}}(\Ubar_t,\Wbar_t; D_t).
\]
The result follows since this is true for any simply connected domain $D$.
\end{proof}
\begin{lemma}\label{lemplatten}
For $s>0,\,x\in\mathbb{R}$, define 
\[
\LL_s(x)=-\kappa\frac{\partial_x\Psi_{A_s}(\psi(x),e^{-s})}{\Psi_{A_s}(\psi(x),e^{-s})}. 
\]
If $\gamma_t$ has the distribution of $\sle$ from $\ubar$ to $\wbar$ in $A_r$, then 
\begin{equation}\label{eqidk3}
dU_t = \LL_{r-t}(W_t-U_t)dt + \sqrt{\kappa}dB_t,
\end{equation}
where $B_t$ is a Brownian motion.
\end{lemma}
\begin{proof}
Let $M_t$ denote the Radon-Nikodym derivative given in the statement of Proposition \ref{proprad} and note that $|\bar\phi_t'(\bar{\xi}_t)| = \phi_t'(\xi_t)$. With respect to $\mathbb{P}$, $M_t$ is a  martingale and
\[
dM_t = M_t \left[b\frac{{\phi}''_t({\xi}_t)}{{\phi}'_t({\xi}_t)}+ i{\phi}'_t({\xi}_t)\bar{\phi}_t(\bar{\xi}_t)\frac{\partial_1\Psi_{A_{r(t)}}(\Ubar_t,\bar{W}_t)}{\Psi_{A_{r(t)}}(\Ubar_t,\bar{W}_t)}\right]d\xi_t.
\]
Here, $\partial_1$ denotes the derivative  with respect to the first argument. Using the Girsanov's theorem,  there exists a probability measure $\mathbb{P}'$ such that
\begin{equation}\label{eqidk2}
d\xi_t = \left[b\frac{{\phi}''_t({\xi}_t)}{{\phi}'_t({\xi}_t)}+ i{\phi}'_t({\xi}_t)\bar{\phi}_t(\bar{\xi}_t)\frac{\partial_1\Psi_{A_{r(t)}}(\Ubar_t,\bar{W}_t)}{\Psi_{A_{r(t)}}(\Ubar_t,\bar{W}_t)}\right]dt + dB_t,
\end{equation}
where $B_t$ is a standard Brownian motion with respect to $\mathbb{P}'$.

For $z\in S_{r,t}$, the transformation $\tilde g_t$ satisfies the radial Loewner equation 
\[
\partial_t {\tilde{g}}_t(z) = \frac{a}{2}\cot \left(\frac{\tilde g_t(z)-\xi_t}{2}\right).
\]
By using the chain rule we get
\[
\partial_t {h}_t(z) = \dot{\phi}_t(\tilde g_t(z)) + \phi_t'(\tilde g_t(z)) \partial _t{\tilde g}_t(z).
\]
Hence,
\[
-\dot{r}(t)\HH_{r(t)}(h_t(z) - U_t) -\frac{a}{2} \phi_t'(\tilde g_t(z)) \cot \left(\frac{\tilde g_t(z)-\xi_t}{2}\right)= \dot{\phi}_t(\tilde g_t(z)).
\]
From Lemma \ref{lemback2} we know $\dot{r}(t) = -a\phi_t'(\xi_t)^2/2$.
Moreover, $\cot(z) = 1/z + \bigo(|z|)$ as $z\to 0$. By using  equation \eqref{eqidk} we can see that 
\[
\dot{\phi}_t(\xi_t) = -\frac{3a}{2} \phi''_t(\xi_t) = -(\frac{1}{2}+b) \phi_t''(\xi_t).
\]
Therefore,  \eqref{eqidk2} implies that
\begin{align*}
dU_t = d\phi_t(\xi_t) &= -b\phi''_t(\xi_t)dt+ \phi'_t(\xi_t)d\xi_t\\
&=  \left[i{\phi}'_t({\xi}_t)^2\bar{\phi}_t(\bar{\xi}_t)\frac{\partial_1\Psi_{A_{r(t)}}(\Ubar_t,\bar{W}_t)}{\Psi_{A_{r(t)}}(\Ubar_t,\bar{W}_t)}\right]dt + \phi_t'(\xi_t)dB_t.
\end{align*}
We can use Lemma \ref{lemback2} one more time to see that with annulus parametrization the last equation can be written as
\[
dU_t = i\kappa\Ubar_t\frac{\partial_1\Psi_{A_{r-t}}(\Ubar_t,\bar{W}_t)}{\Psi_{A_{r-t}}(\Ubar_t,\bar{W}_t)}dt + \sqrt{\kappa}dB_t.
\]
From conformal covaraince of the partition function, we get $
\Psi_{A_{r-t}}(\Ubar_t,\Wbar_t) = \Psi_{A_{r-t}}(\psi(W_t-U_t),e^{-r+t}).
$ Moreover,
\[
\partial_{(W_t-U_t)}\Psi_{A_{r-t}}(\psi(W_t-U_t),e^{-r+t}) = i\psi(W_t-U_t)\,\partial_1\Psi_{A_{r-t}}(\psi(W_t-U_t),e^{-r+t})
\]
and
\[
\partial_1\Psi_{A_{r-t}}(\Ubar_t,\Wbar_t) = -\psi(W_t-2U_t)\,\partial_1\Psi_{A_{r-t}}(\psi(W_t-U_t),e^{-r+t}),
\]
which completes the proof.
\end{proof}
It is not hard to verify that for $x\in(-\pi,\pi)$, $\Psi_{A_r}(\psi(x),e^{-r})$ is an even function that is decreasing in $|x|$.
Hence, $\LL_r(x)$ is an odd function with $\LL_r(x)\geq 0$ for $x\in[0,\pi)$.
For more details see section 5 of \cite{greg_annulus}. 

\section{Proof of Theorem 1}\label{proof}
\subsection{The case with two paths}

We consider the measure on paths $( \gamma, \gamma')$, where $\gamma,\gamma'$ are $\sle$ paths from $|\ubar| = 1$ to $|\wbar| = e^{-r}$ and  from $|\xbar|=1$ to $|\ybar| = e^{-r}$,  respectively. 
Let $\bfub = (\ubar,\xbar),\,\bfwb = (\wbar,\ybar)$.
By \eqref{eqdef2}, 
 the partition function can be written as 
\[
\Psi_{A_r}\left(\bfub,\,\bfwb\right)=\Psi_{A_r}\left((\u,\x),(\w,\y)\right)=\Psi_{A_r}(\ubar,\wbar)\EE\left[H_{A_r\sm\gamma}(\xbar,\ybar)^b\right],
\]
where $\EE$ denotes the expectation with respect to the distribution of $\gamma$.  Recall the function $\psi(z) = e^{iz}$ and choose $0\leq  \u,\x,\w,\y < 2\pi$ such that 
\[
\ubar = \psi(\u),\,\xbar = \psi(\x),\,\wbar = \psi(\w+ir),\,\ybar = \psi(\y+ir).
\]
Consider the function
\begin{equation}\label{eq2}
\pf(r,\bfu,\bfw)=\frac{\Psi_{A_r}(\bfub,\,\bfwb)}{\Psi_{A_r}(\ubar,\wbar)H_{A_r}(\xbar,\ybar)^b}=\EE\left[Q_{A_r}(\xbar,\ybar;A_r\sm\gamma)^b\right].
\end{equation}
We will show that $\pf$ is a smooth function of $r,\u,\x,\w,\y$. It is clear that $ \pf(r,\bfu,\bfw)$ can be written as a function of $r, (\x - \u), (\w - \u), (\y - \u)$. However, it is easier to prove the smoothness if we consider it as $ \pf(r,\bfu,\bfw)$.

Let $\eta_t\subset S_r$ be the unique continuous curve satisfying $\psi(\eta_t) = \gamma_t,\,\eta(0+) = \u$. Define $\hbar_t,\gg,h_t,g_t,\text{etc}.$ for $\gamma_t$ as in section \ref{secnotation}. Define,
\[
\U_t = h_t(\eta(t)),\,\X_t = h_t(\x),\, \W_t = h_t(\w+ir),\,\Y_t = h_t(\y+ir),
\]
\[
\Ubar_t = \hbar_t(\gamma(t)),\, \Xbar_t = \hbar_t(\xbar),\,\Wbar_t = \hbar_t(\wbar),\,\Ybar_t = \hbar_t(\ybar).
\]

Note that for $t<\tau_r$, we can write $H_{A_r\sm \gamma_t}(\xbar,\ybar)$ in terms of the Poisson kernel in the covering space $S_{r,t}$,
\begin{equation*}
H_{A_r\sm \gamma_t}(\xbar,\ybar) = e^r\sum _{k\in\mathbb{Z}} H_{ S_{r,t}}(\x,\y + 2k\pi).
\end{equation*}
For now, suppose $\gamma_t$ has the radial parametrization.
\begin{lemma}\label{lem1}
Suppose $t<\tau_r$ and let $z\in\mathbb{R}, z'\in\mathbb{R} + ir$ and $Z_t = h_t(z), Z'_t = h_t(z')$. 
Then
\begin{equation}\label{eq2.8}
\partial_t\log Q_{S_r}(z,z'; S_{r,t}) =2\,\dot r(t)\,\textbf{F}(r(t),Z_t - U_t,Z'_t - U_t),
\end{equation}
where $U_t = h_t(\eta(t))$, $r(t)$ are defined in \eqref{eq2.5} and
\[
\textbf{F}(r,z,z'):=\sum_{k\in\mathbb{Z}}\frac{H_{ S_{r}}(z ,2k\pi)\,H_{ S_{r}}(2k\pi,z')}{H_{ S_{r}}(z ,z')}.
\]
\end{lemma}
\begin{proof}
We only need to prove the claim for the right derivative with respect to $t$ since the right-hand side of \eqref{eq2.8} is continuous in $t$. Moreover, we only need to prove the claim for the derivative at $t=0$ because
\begin{align*}
\partial_t\log Q_{S_r}(z,z'; S_{r,t}) & = \lim _{s\downarrow0}\frac{1}{s} \log Q_{S_{r,t}}(z,z';S_{r,t+s})\\
& =\lim _{s\downarrow0}\frac{1}{s} \log Q_{S_{r(t)}}\left(Z_t,Z'_t;S_{r(t)}\sm h_t\circ \tilde{\eta} \,(t,t+s)\right).
\end{align*}
First, note that by using \eqref{eqhalf} and Lemma \ref{lemback2} we get
\[
\lim_{s\downarrow0}\frac{1}{s} \,\hcap\left[h_t\circ \eta \, (t,t+s)\right] = a\phi_t'(\xi_t)^2 = -2\dot{r}(t).
\]
At $t=0$ we have
\[
\partial_t \log  Q_{S_r}(z,z';S_{r,t}) = \partial_t Q_{S_r}(z,z';S_{r,t}) = - \partial_t [1- Q_{S_r}(z,z';S_{r,t})].
\]
The term $1- Q_{S_r}(z,z';S_{r,t})$ is the probability that  a Brownian excursion from $z$ to $z'$ in $S_r$ hits $\tilde{\eta}_t$.
We first calculate the probability that the Brownian excursion hits $\eta_t$.
We can write
\begin{align*}
1 -Q_{S_r}(z,z';S_r\sm\eta_t)&= \frac{E^z[H_{S_r}(B_\tau,z')1\{B_\tau\in \eta_t\}]}{H_{S_r}(z,z')} ,
\end{align*}
where $B_s$ is a Brownian excursion in $\hp$ started from $z$ and $\tau$ is the first time $B_s$ exits $S_r\sm\eta_t$.
Let $D\subset S_r$ be a half disk centered at $\u$ such that $z\not\in\bar D$.
Assume $t$ is small enough so that $\eta_t \subset D$ and let $d_t = 2\diam[\eta_t]$.
Using the exact form of the Poisson kernel in $D$, we can see that if $B_\tau\in\eta_t$, then for any $w\in \partial D \cap \hp$,
 \[
H_D(B_\tau,w) = \Im(B_\tau) H_{ D}(\u,w) [ 1 + O(d_t)].
 \]
Using this, we get
\begin{equation*}\label{eq3}
H_{S_r}(B_\tau,z')\, 1\{B_\tau\in \eta_t\} = \text{Im}(B_\tau)\,  H_{S_r}(u,z') \, 1\{B_\tau\in\eta_t\}\,[1 + O(d_t)].
\end{equation*}
Let $f(w)$ be the unique bounded harmonic function on $S_r\sm\eta_t$ with boundary condition $\Im(w)1\{w\in\eta_t\}$. 
Similarly to  Lemma \ref{lem007}, we can show that
\begin{equation}\label{eq3.4}
E^z[f(B_\tau)] = \hcap[\eta_t] H_{S_r}(z,u) [1 + O(d_t)].
\end{equation}
To see this, suppose $D_t\subset S_r$ is a half disk of radius $d_t$ centered at $\u$ and let $w\in S_r\cap \partial D_t$.  Lemma \ref{lem0.1} implies that
\[
H_{ S_r\sm \bar{D}_t}(z,w) = 2\sin\theta_w H_{ S_r}(z,\u)[1+O(d_t)],
\]
where $\theta_w = \arg w$. Using this, we get
\begin{equation}\label{eq3.5}
E^z[f(B_\tau)] = \frac{1}{\pi}H_{ S_r}(z,\u)[1+O(d_t)] \int _{S_r\cap \partial D_t} 2\sin\theta_w E^w[f(B_\tau)] |dw|.
\end{equation}
Let $\sigma$ be the first time $B_s$ exits $\hp\sm \eta_t$. Note that $\tau\leq \sigma$ and for $w\in S_r\cap\partial D_t$,
\begin{align*}
E^w[f(B_\tau)] &= E^w[\Im(B_\sigma)] - E^w[\Im(B_\sigma) 1\{\tau<\sigma\}],\\
E^w[\Im(B_\sigma) 1\{\tau<\sigma\}] & = O(d_t) \hcap[\eta_t], \\
\frac{1}{\pi}\int _{S_r\cap \partial D_t} 2\sin\theta_w E^w[\Im(B_\sigma)]|dw| &= \hcap[\eta_t].
\end{align*}
Plugging this into \eqref{eq3.5} proves  \eqref{eq3.4}. Therefore we can see that at $t=0$, for any $k\in\mathbb{Z}$
\begin{equation*}
\partial_t [1-Q_{S_r}(z,z';S_r\sm\eta_t + 2k\pi)] = a\,\frac{H_{ S_r}(z ,\u + 2k\pi)H_{ S_r}(\u + 2k\pi,z')}{H_{ S_r}(z,z')}.
\end{equation*}
One can use a similar argument to see that the probability of the Brownian excursion hitting at least two copies of $\eta_t$ is of order $O(\hcap[\eta_t]^2)$. Therefore, 
\[
1- Q_{S_r}(z,z';S_{r,t}) = \sum_{k=-\infty}^\infty \left[1-Q_{S_r}(z,z';S_r\sm\eta_t + 2k\pi)\right] + \bigo (\hcap[\eta_t]^2),
\]
and 
\[
\partial_t [1- Q_{S_r}(z,z';S_{r,t})] = \partial_t \sum_{k\in\mathbb{Z}} [1- Q_{S_r}(z,z';S_r\sm\eta_t + 2k\pi)].
\]
The result for general $t$ follows from this and Lemma \ref{lemback2}.
\end{proof}
If $\gamma_t$ has the annulus parametrization, then $\tau_r = r$ and for any $t<r$
 we can write \eqref{eq2.8} as
\begin{equation}\label{eq3.9}
\partial_t\log Q_{S_r}(z,z'; S_{r,t}) =-2\,\textbf{F}(r-t,h_t(z) - U_t,h_t(z') - U_t).
\end{equation}
\begin{proposition}\label{propn21}
For $s>0$, $ z,z'\in \mathbb{R}$,  define
\[
{\textbf{A}}(s, z, z') = \frac{H_{ A_s}(\psi(z),1) H_{ A_s}(1,\psi(z'+is))}{H_{ A_s}(\psi(z),\psi(z'+is))}.
\]
If $\gamma_t$ has the annulus parametrization, then for all $t<r$,
\begin{equation}\label{eq4}
\partial_t \log Q_{A_r}(\xbar,\ybar;A_r\sm\gamma_t) = -2\, {\textbf{A}}(r-t, \X_t - \U_t , \Y_t - \U_t).
\end{equation} 
\end{proposition}
\begin{proof}
We only need to prove the claim for the right derivative and $t=0$, since the right-hand side of \eqref{eq4} is continuous in $t$ and 
\[
\lim _{s\downarrow0} \frac{1}{s}[\log Q_{A_r}(\xbar,\ybar;A_r\sm\gamma_{t+s})- \log Q_{A_r}(\xbar,\ybar;A_r\sm\gamma_t)] = \partial _ s  \log Q_{A_{r-t}}(\Xbar_t,\Ybar_t; A_{r-t} \sm \bar\gamma_s)|_{s=0},
\]
where $\bar\gamma_s = h_t \circ \gamma(t,t+s]$ is a $\sle$ curve in $A_{r-t}$ starting from $\Ubar_t$.  

 At $t=0$ we have
\begin{align}\label{eq5}
\partial_t \log Q_{A_r}(\xbar,\ybar;A_r\sm\gamma_t) &= \partial_t\log H_{A_r}(\xbar,\ybar;A_r\sm\gamma_t) \\\nonumber
& = \frac{\partial_t H_{A_r\sm\gamma_t}(\xbar,\ybar)}{H_{A_r}(\xbar,\ybar)}.
\end{align}
By using \eqref{eq3.9} we get
\begin{align*}
\partial_t H_{A_r\sm\gamma_t}(\xbar,\ybar) & = \partial_t \,e^r \sum_{k'\in\mathbb{Z}} H_{S_{r,t}} (\x , \y + 2k'\pi) \\
& = -2\,e^r \sum_{k'\in\mathbb{Z}} \sum_{k\in\mathbb{Z}} H_{ S_r}(\x-\u, 2k\pi)\,H_{ S_r}(2k\pi,\y + 2k'\pi -\u)\\
&= -2\,\sum_{k\in\mathbb{Z}} H_{ S_r}(\x-\u, 2k\pi) \sum_{k'\in\mathbb{Z}} e^r\, H_{ S_r}(2k\pi,\y + 2k'\pi - \u)\\
& = -2\,H_{A_r}(1,\ybar/\ubar) \sum_{k\in\mathbb{Z}} H_{ S_r}(\x-\u, 2k\pi)\\
& = -2\,H_{A_r}(1,\ybar/\ubar)\,H_{A_r}(\xbar/\ubar,1).
\end{align*}
The result follows from this and \eqref{eq5}.
\end{proof}

Before we prove our main result, we recall the H\"ormander's theorem. Let $\Omega\subset \mathbb{R}^n$ be an open set. A linear differential operator $\mathcal{L}$ with $C^\infty$ coefficients on $\Omega$ is called \emph{hypoelliptic} if for every distribution $u$ on $\Omega$, $u$ is $C^\infty$ when $\mathcal{L}u$ is $C^\infty$. Assume $X_0,X_1,\ldots,X_k$ are first order homogeneous differential operators with $C^\infty$ coefficients on $\Omega$. Let
\[
\mathcal{L} = \sum_{j=1}^k X_j^2 + X_0 + c,
\]
where $c$ is a smooth function on $\Omega$. 
In \cite {hormander}, H\"ormander established a characterization of hypoelliptic second order differential operators with $C^\infty$ coefficients.  In particular, he proved the following theorem.
\begin{theorem}\label{hor}
If at all points in $\Omega$ the rank of the Lie algebra generated by the vector fields $X_0,\,X_1,\ldots,X_k$ equals $n$, then $\mathcal{L}$ is hypoelliptic. 
\end{theorem}
We say $\mathcal{L}$ satisfies the \emph{H\"ormander's condition} if it satisfies the requirements of Theorem \ref{hor}.
Having this result, we prove the smoothness of $\pf$.
\begin{proposition}\label{thmmain}
If $\bfu = (\u,\x),\,\bfw = (\w,\y)$, then 
$\pf(r,\bfu,\bfw)$ is a positive smooth function satisfying
\begin{align}
\partial_r\pf& = -2b\AA(r,\x-\u,\y-\u) \pf+ \LL_r(\w-\u) \partial_{\u}\pf+ \HH_r(\x-\u)\partial_{\x} \pf \\
&\,\,\,\,+\HH_r^R(\w-\u)\partial_{\w}\pf + \HH_r^R(\y-\u)\partial_{\y}\pf + \frac{\kappa}{2}\partial_{\u\u}\pf = 0.\nonumber
\end{align}
\end{proposition}
\begin{proof}
For $t< r$, we have
\begin{align*}
Q_{A_r}(\xbar,\ybar;A_r\sm\gamma_t)&=\exp\left\{\int_0^t\partial_t\log Q_{A_r}(\xbar,\ybar;A_r\sm\gamma_s)ds\right\}.
\end{align*}
This is also true for $t=r$ because $Q_{A_r}(\xbar,\ybar;A_r\sm\gamma_t)$ is continuous at $t=r$.
Proposition \ref{propn21} indicates that
\begin{equation*}
 \pf(r,\bfu,\bfw)=\EE\left[\exp\left\{-2b\int_0^r\textbf{A}(r-s,\X_s - \U_s,\Y_s - \U_s) ds\right\}\right].
\end{equation*}
If $\mathcal{F}_t$ denotes the $\sigma$-algebra generated by $\gamma_t$, then 
\begin{align*}
 M_t & := \EE\left[\exp\left\{-2b\int_0^r\textbf{A}(r-s,\X_s - \U_s,\Y_s - \U_s) ds\right\} \middle | \mathcal{F}_t\right]\\
 & = \exp\left\{-2b\int_0^t\textbf{A}(r-s, \X_s - \U_s,\Y_s - \U_s) ds\right\} \pf(r-t,\U_t,\X_t,\W_t,\Y_t) 
\end{align*}
is a martingale. 
Recall that as in Proposition \ref{proplow}, $\HH_t^R(z) = \Re[\HH_t(z + it)]$ for $z\in\mathbb{R}$. Using Lemma \ref{lemplatten} and Proposition \ref{proplow},
\begin{align*}
d\U_t &= {\LL}_{r-t}(\W_t -\U_t)dt + \sqrt{\kappa}dB_t,\\
d\X_t &= \HH_{r-t}( \X_t - \U_t)dt,\\
d\W_t &= \HH^R_{r-t}( \W_t - \U_t)dt,\\
d\Y_t &= \HH^R_{r-t}( \Y_t - \U_t)dt.
\end{align*}

Consider the process
\[
Z_t =\left(r-t,\,Q_{A_r}(\x,\y;A_r\sm\gamma_t),\,\U_t,\,\X_t,\,\W_t,\,\Y_t\right).
\]
This is a diffusion process with infinitesimal generator
\begin{align}\label{eqn2-1}
A\phi(\bfz)&=-2b\AA \,\partial_{z_2}\phi -\partial_{z_1}\phi+ \LL_{z_1}(z_5-z_3) \partial_{z_3}\phi+ \HH_{z_1}(z_4-z_3)\partial_{z_4}\phi  \\\nonumber
&\qquad + \HH^R_{z_1}(z_5-z_3)\partial_{z_5}\phi +\HH^R_{z_1}(z_6-z_3)\partial_{z_6}\phi  + \frac{\kappa}{2}\partial_{z_3z_3},
\end{align}
for any $C^2$ function $\phi\in \mathcal{D}_A$ and suitable $\bfz=(z_1,\ldots,z_6)\in\mathbb{R}^6$. Here, $\mathcal{D}_A$ is the domain of the generator $A$ and we write $\AA$ in short for $\AA(z_1,z_4-z_3,z_6-z_3)$.  Using the  Fokker--Planck equation, we can see that \eqref{eqn2-1} holds for any $\phi\in\mathcal{D}_A$ as long as the derivatives are interpreted in the weak sense. Let 
\[
f(z) = z_2^b\, V(z_1,z_3,z_4,z_5,z_6).
\]
Since $M_t = f(Z_t)$ is a martingale, we have $\EE(f(Z_t)) - f(Z_0) = 0$ for all $Z_0$ and $t\leq r$. In particular, $f\in \mathcal{D}_A$ and $Af =0$. Therefore, at least in the weak sense $\mathcal{L} V = 0$ for all $r,\u,\x,\y,\w$, where
\begin{align*}
\mathcal{L}&=-2b\AA(r,\x-\u,\y-\u) -\partial_r +\LL_r(\w-\u) \partial_{\u}+ \HH_r(\x-\u)\partial_{\x}  \\
&\quad + \HH^R_r(\w-\u)\partial_{\w}+ \HH^R_r(\y-\u)\partial_{\y}  + \frac{\kappa}{2}\partial_{\u\u}.
\end{align*}
We now prove that $\mathcal{L}$ satisfies the H\"ormander's condition, from which we conclude $\mathcal{L}$ is hypoelliptic and $\pf$ is a smooth function using Theorem \ref{hor}.
Note that 
\[
\mathcal{L} = \frac{1}{2}A_1^2 + A_0 + C,
\]
where 
\begin{align*}
A_1 &= \sqrt{\kappa}\partial_\u\\
A_0 = &-\partial_r +\LL_r(\w-\u) \partial_{\u}+ \HH_r(\x-\u)\partial_{\x} + \HH^R_r(\w-\u)\partial_{\w} +\HH^R_r(\y-\u)\partial_{\y} \\
C &= -2b\,\AA(r,\x-\u,\y-\u).
\end{align*}
We will show that the Lie algebra generated by the vector fields
\begin{equation}\label{eqmj}
A_0,\,A_1,\,[A_1,A_0],\,[A_1,[A_1,A_0]],\,\ldots
\end{equation}
has rank 5 for every $r,\u,\x,\w,\y$ satisfying $r>0,\,\u\neq \x,\,\w\neq\y$.
It is not hard to see that for every $n\in\mathbb{N}$, the $(n+2)$-th term in \eqref{eqmj} can be written as
\[
[\partial_{u^{(n)}} L_r(\w-\u)]\partial_{\u}+[\partial_{u^{(n)}}\HH_r(\x-\u)]\partial_{\x} + [\partial_{u^{(n)}}\HH^R_r(\w-\u)]\partial_{\w} +[\partial_{u^{(n)}}\HH^R_r(\y-\u)]\partial_{\y},
\]
where $\partial_{u^{(n)}}$ denotes the $n$-th derivative with respect to $u$.
Among the vector fields given in \eqref{eqmj}, $A_0$ is the only vector field with nonzero coefficient for $\partial_r$. 
Also in $A_1$, only $\partial_u$ has nonzero coefficient. 
Therefore, it is enough to show that the span of the vector fields $\partial_{\x},\,\partial_{\w},\,\partial_{\y}$ is a subspace of the span of
\[
[A_1,A_0],\,[A_1,[A_1,A_0]],\,[A_1,[A_1,[A_1,A_0]]],\,\ldots\,.
\] 
For a fixed $r>0$, define the  functions $f_0(z) = \HH_r(z),\, f_1(z) = \HH_r(z+ir)$. Note that $f_1'(z) = \partial_z \HH_r^R(z)$ for $z\in\mathbb{R}$, since $\Im[\HH_r(z+ir)] = 1$.
We want to show that for all $0<z_1<2\pi$ and $0\leq z_2<z_3<2\pi$, there exist three linearly independent vectors among
\[
v_k:=(f_0^{(k)}(z_1),\,f_1^{(k)}(z_2),\,f_1^{(k)}(z_3)),\qquad k\geq 1.
\]
Suppose the claim is not true and there exist constants $a_j,\,j\in\{1,2,3\}$ such that they are not all equal to $0$ and
$
a_1 f_0^{(k)}(z_1)+a_2 f_1^{(k)}(z_2)+a_3 f_1^{(k)}(z_3)=0
$
for all $k\geq 1$. If
\[
\tilde{f}(\epsilon) := a_1 f_0(z_1 + \epsilon)+a_2 f_1(z_2+\epsilon)+a_3 f_1(z_3+\epsilon),
\]
then $\tilde{f}^{(k)}(0)=0$ for all $k\geq1$.
Let 
\[
\epsilon_0:=\min\{z_1,\,(2\pi-z_1),\, |z_2+ir|,\,|z_2+ir-2\pi|,\, |z_3+ir|,\,|z_3+ir-2\pi|\}
\] 
and let $B_{\epsilon_0}(0)$ be the open ball of radius $\epsilon_0$ around the origin.
Since $\HH_r(z)$ is an elliptic function with periods $2\pi, 2ir$ and poles at $2k\pi + i2mr$, the function $\tilde{f}(\epsilon)$ is analytic on $B_{\epsilon_0}(0)$. 
Therefore, for some  constant $c$,  $\tilde f\equiv c$ on $B_{\epsilon_0}(0)$. 
But this is a contradiction because $\tilde f$ is continuous and there exists $\epsilon\in\partial B_{\epsilon_0}(0)$ such that $\tilde f(\epsilon)=\infty$.
\end{proof}
\subsection{The case \texorpdfstring{$n>2$}{LG}}
In this section we explain how the proof of Proposition \ref{thmmain} can be extended to the case $n>2$. 
Let $0\leq u^1,\ldots,u^n<2\pi$ and $0\leq w^1,\ldots,w^n<2\pi$ be distinct numbers and for $1\leq j\leq n$,  define $\ubar^j = \psi(\u^j),\, \wbar^j = \psi(ir+\w^j)$ as before. Let 
$\bfu = (u^1,\ldots,u^n),\,\bfw = (w^1,\ldots,w^n)$ and 
$\bfub=(\ubar^1,\ldots,\ubar^n),\,\bfwb = (\wbar^1,\ldots,\wbar^n)$.
For $1\leq j\leq n$, let $\gamma^j$ be a continuous path connecting $\ubar^j$ to $\wbar^j$ in $A_r$ and let $\bgamma=(\gamma^1,\ldots,\gamma^n)$. Define
\[
\pf(r,\,\bfu,\,\bfw) = \frac{\Psi_{A_r}(\bfub,\bfwb)}{\Psi_{A_r}(\ubar^1,\wbar^1)\prod_{j=2}^n H_{A_r}(\ubar^j,\wbar^j)^b}=\frac{\prod_{j=2}^n \Psi_{A_r}(\ubar^j,\wbar^j)}{\prod_{j=2}^n H_{A_r}(\ubar^j,\wbar^j)^b} \EE[Y(\bgamma)],
\]
where $\EE$ denotes the expectation with respect to $\mu^\#_{A_r}(\ubar^1,\wbar^1)\times\ldots\times \mu^\#_{A_r}(\ubar^n,\wbar^n)$. 
\begin{lemma}\label{lemwc}
Let $m_t$ denote the Brownian loop measures of loops in $A_r\sm \gamma^1_t$ and let $m=m_0$. 
Let $\bfub_t = (\gamma^1(t),\ubar^2,\ldots,\ubar^n)$ and define $\bgamma_t$ to be $n$-tuples of paths connecting $\bfub_t$ to $\bfwb$ in $A_r\sm\gamma_t$. 
Then
\[
\sum_{j=2}^nm[K_j(\bgamma)] = \sum_{j=2}^nm_t[K_j(\bgamma_t)] + \sum_{j = 2}^n m(\gamma^1_t,\,\gamma^j).
\]
\end{lemma}
\begin{proof} 
We prove this by induction. It is easy to see that the claim is true for $n=2$. Assuming the claim holds for $n-1$, we prove it for $n$.
Let $\bgamma=(\bgamma',\gamma_n)$ and $\bgamma_t=(\bgamma'_t,\gamma_n)$. Using Lemma \ref{lemdef1},
\[
\sum_{j=2}^nm[K_j(\bgamma)]  = \sum_{j=2}^{n-1}m[K_j(\bgamma')] + m(\gamma^1_t,\gamma^n) + m_t(\bgamma'_t,\gamma^n).
\]
Hence, by the induction hypothesis for $n-1$,
\[
\sum_{j=2}^nm[K_j(\bgamma)]  = \sum_{j=2}^{n-1}m_t[K_j(\bgamma'_t)] +\sum_{j=2}^{n} m(\gamma_t^1,\gamma^j) + m_t(\bgamma'_t,\gamma^n).
\]
 Using Lemma \ref{lemdef1} one more time gives the result.
\end{proof}
\begin{proof}[Proof of Theorem \ref{thmucsd}]
Since $\Psi_{A_r}(\ubar^j,\wbar^j),\,H_{A_r}(\ubar^j,\wbar^j)$ are smooth, it is  enough to to prove that $\pf$ is smooth and this can be proved in a similar way to Proposition \ref{thmmain}.
Suppose $\gamma^1_t$ has the annulus parametrization. Let $\hbar_t:A_r\sm\gamma^1_t\to A_{r-t}$ be our usual conformal transformation. Let $U^j_t = h_t(u^j),\, W^j_t = \Re[h_t(ir+w^j)],\, \textbf{U}_t = (U^1_t,\ldots,U^n_t),\, \textbf{W}_t = (\W_t,\ldots,W^n_t)$ and $\Ubar_t^j = \hbar(\ubar^j),\,\Wbar_t^j = \hbar (\wbar^j)$. 
Using Lemma \ref{lemwc} and equations \eqref{theq6}, \eqref{eqdef0}, we can see that
\begin{align*}
M_t&:=\frac{\prod_{j=2}^n \Psi_{A_r}(\ubar^j,\wbar^j)}{\prod_{j=2}^n H_{A_r}(\ubar^j,\wbar^j)^b} \EE[Y(\bgamma)|\gamma^1_t] \\
&= \frac{\prod_{j=2}^n \Psi_{A_r}(\ubar^j,\wbar^j)}{\prod_{j=2}^n H_{A_r}(\ubar^j,\wbar^j)^b} \, \frac{\prod_{j=2}^n \Psi_{A_r\sm\gamma^1_t}(\ubar^j,\wbar^j)}{\prod_{j=2}^n \Psi_{A_r}(\ubar^j,\wbar^j)} \,\EE[Y(\bgamma_t)|\gamma_t^1]\\
&= \frac{\prod_{j=2}^n \Psi_{A_r\sm\gamma^1_t}(\ubar^j,\wbar^j)}{\prod_{j=2}^n H_{A_r}(\ubar^j,\wbar^j)^b} \,\frac{\prod_{j=2}^n H_{A_{r-t}}(\Ubar^j_t,\Wbar^j_t)^b}{\prod_{j=2}^n \Psi_{A_{r-t}}(\Ubar^j_t,\Wbar^j_t)} \,\pf(r-t, \textbf{U}_t,\,\textbf{W}_t)\\
&= \prod_{j=2}^n Q_{A_r}(\ubar^j,\wbar^j; A_r\sm\gamma^1_t)^b\, \pf(r-t, \textbf{U}_t,\,\textbf{W}_t),
\end{align*}
is a martingale. At least in the weak sense, $M_t$ satisfies a  differential equation related to the infinitesimal generator of the process
\[
(r-t,Q_{A_r}(\u^2,\w^2; A_r\sm\gamma^1_t),\ldots,Q_{A_r}(u^n,w^n; A_r\sm\gamma^1_t),\textbf{U}_t,\textbf{W}_t).
\] 
Proposition \ref{propn21} gives the time derivative of $Q_{A_r}(\ubar^j,\wbar^j; A_r\sm\gamma^1_t)$. The derivatives of $V(r-t,{\bf U}_t,{\bf W}_t)$ can be described using Proposition \ref{proplow}. 
Using these, we can see that if 
\begin{align*}
A_1 & = \sqrt{\kappa}\partial_{u^1},\\
A_0 & = -\partial_r + L_r(w^1-u^1)\partial_{u^1} + \sum _{j=2}^n \HH_r(u^j-u^1)\partial_{u^j} + \sum_{j=1}^n \HH_r^R(w^j-u^1)\partial_{w^j},\\
C & = -2b\sum_{j=2}\AA(r,u^j-u^1,w^j-u^1),\\
\mathcal{L} &= \frac{1}{2}A_1^2+A_0+C,
\end{align*}
then $\mathcal{L} \pf = 0$ for all $r>0$ and distinct $\bfu,\bfw$. 
Verifying the H\"ormander's conditions
can be done as in the proof of Proposition \ref{thmmain}. 
\end{proof}
\section{Two-Sided Measures}\label{twosided}
In this section, we start with reviewing the boundary perturbation property for radial $\sle$. This is similar to the same property for chordal $\sle$ described in \cite{greg_restrict,Parkcity}. Similar to the chordal case, this allows for obtaining radial $\sle$ in a smaller domain by a change of measure.

Next, we construct  two-sided $\sle$ using two independent radial $\sle$ and a change of measure.
Finally, we make a connection between two-sided $\sle$ in $\disk$ and bi-chordal annulus $\sle$. In particular, we show that before reaching the boundary, the two measures are absolutely continuous and we give an estimate for the Radon-Nikodym derivative.

\subsection{Boundary perturbation for radial \texorpdfstring{$\sle$}{LG}}
Suppose $D\subset\disk$ is a simply connected domain such that $0\in D$ and $D$ agrees with $\disk$ in a neighborhood of $1$. Let $K=\disk\sm D$.
Let $\bar{G}: D\to \disk$ be the unique conformal transformation with $\bar G(0) = 0$ and $\bar G'(0)>0$. 
Let $\gamma_t$ be a radial $\sle$ from $1$ to 0 in $\disk$ (continuity at 0 is shown in \cite{greg_twosided}) and let $\gg:\disk\sm\gamma_t\to \disk$ be the unique conformal transformation satisfying $\gg(0) =0,\,\gg'(0) >0$. 
Let $\bar\gamma_t = \bar G(\gamma_t)$,  define $\bar\varphi_t:\disk\sm \bar\gamma_t\to \disk$ to be the conformal transformation satisfying $\varphi_t(0) = 0,\,\varphi'_t(0)>0$ and let $\bar \Phi_t :=\bar\varphi_t\circ \bar{G}\circ \gg^{-1}$. Let $\tg_t$ be the unique conformal transformation that is continuous in $t$ and satisfies
\[
\gg(e^{iz}) = e^{i\tg_t(z)},\qquad \tg_0(z) = z.
\]
Define $G,\,\varphi_t,\,\Phi_t$ in a similar way for $\bar G, \,\bar\varphi_t,\,\bar\Phi_t$ (see figure \ref{pic0}). As before, let  $\psi(z)= e^{iz}$. 
Let $U_t$ be the continuous process satisfying $\psi(U_t) = \gg (\gamma(t)),\,U_0 = 0$. If $\gamma_t$ has the radial parameterization, then we know that $\gg$ satisfies the radial Loewner equation
\[
\partial_t\gg(z) = \frac{a}{2}\,\gg(z)\,\frac{e^{iU_t}+\gg(z)}{e^{iU_t}-\gg(z)},\qquad \bar{g}_0(z) = z,
\]
for any $z\in \disk\sm \gamma_t$. Equivalently, if $\cot_2(z):=\cot(z/2)$, then
\[
\partial_t \tg_t(z) = \frac{a}{2}\cot_2\left(\tg_t(z) - U_t\right), \qquad \tg_0(z) = z.
\]
If $\gamma_t$ does not have radial parameterization, then the term $a/2$ is substituted with $\partial_t \log\gg'(0)$. 
\begin{figure}
  \centering
    \includegraphics[width=0.9\textwidth]{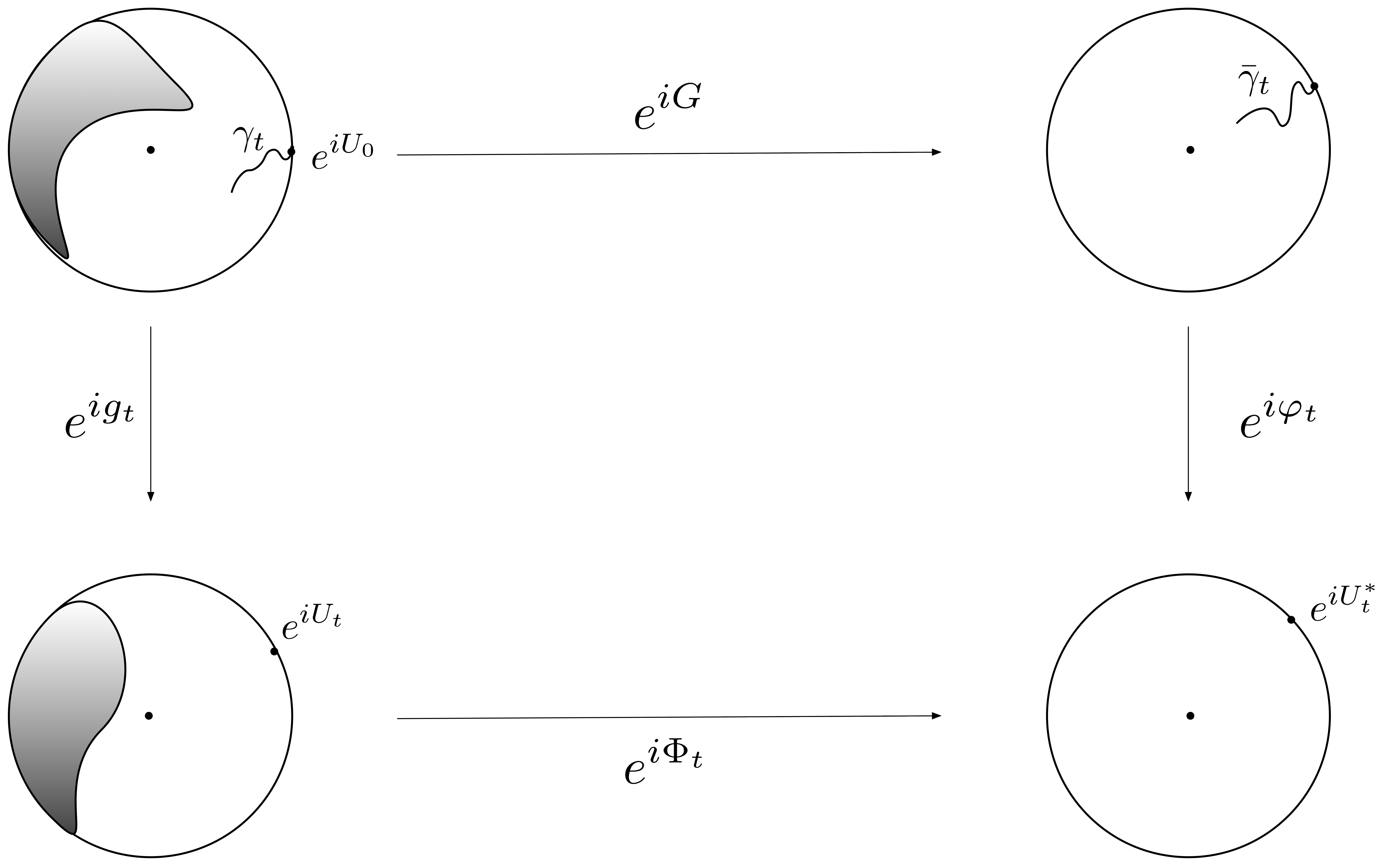}
    \caption{Shaded area in unit disk on the top left represents $K=\disk\sm D$. }
    \label{pic0}
\end{figure}
\begin{lemma}\label{lem2sided1}
The Brownian loop measure of loops with nonzero winding number in $\disk$ that intersect $\disk\sm D$ is 
\[
\frac{\log \bar G'(0)}{6}.
\]
\end{lemma}
\begin{proof}
Using conformal invariance and infinitesimal perturbations
it is not hard to see that the measure must be a
linear function of $\log \bar G'(0)$.  One can compute
the constant using the Brownian bubble measure; see
 corollary 3.12 in \cite{greg_annulus}.
\end{proof}
\begin{lemma}\label{lemextra1}
Let $\gamma_t\subset D$ be a  curve with $\gamma(0+)
 \in \p \Disk$ such that $\log\gg'(0) = at/2$ and $\gamma(t)\to 0$ as $t\to \infty$. If $\tau_r=\inf\{t:\,\gamma_t\cap C_r\neq \emptyset\}$, then
\[
 \lim_{t\to \infty} \bar\Phi_t'(0) =1,\qquad \lim_{r\to \infty} \Phi_{\tau_r}'(U_{\tau_r}) =1.
\]
\end{lemma}
\begin{proof} For the first limit, note that the 
Schwarz lemma   implies that
\[
1\leq \bar\Phi_t'(0)\leq \frac{1}{\dist(0,\gg(K))}.
\]
Hence it suffices to show that the harmonic measure
of $\gg(K)$ from
$0$ in $\Disk \setminus \gg(K)$ tends to zero.  But this
is the same as the harmonic measure of
$K$  from $0$ in $D \setminus \gamma_t$ and this goes to 0
monotically by the Beurling estimate.  

To prove the second equality, we need the following fact (see Lemma 2.6 in \cite{greg_twosided} for more details and a proof). 
Choose $d$ such that $e^{-d}\leq  \dist(0,K)$ and let $r>d$. 
There exists a unique open connected arc $\eta(0,1)\subset C_d$ such that $\eta(0+),\eta(1-)\in\gamma_{\tau_r}$ and $\eta(0,1)$ disconnects $K$ from 0 in $\disk\sm\gamma_{\tau_r}$. 
Consider the disjoint union $C_0 = l_1\cup l_2\cup l_3\cup \{\psi(U_{\tau_r})\}$, where $l_3$ is the unique closed connected arc  with endpoints $\bar{g}_{\tau_r}(\eta(0+)),\bar{g}_{\tau_r}(\eta(1-))$ and $l_1,l_2$ are open connected arcs. Then 
\[
\diam[\bar{g}_{\tau_r}\circ\eta(0,1)]\leq c_0\,e^{-(r-d)/2} \min\{|l_1|,\,|l_2|\},
\]
where $| \cdot |$ denotes length and $c_0$ is an absolute constant. Let $\bar K_r = \bar g_{\tau_r}(K)$. Since $\eta(0,1)$ disconnects $K$ from $0$ in $\disk\sm \gamma_{\tau_r}$, the curve $\bar g_{\tau_r}\circ\eta(0,1)$ disconnects $\bar K_r$ from 0 in $\disk$. Therefore, for large enough $r$,
\begin{equation}\label{eqextra2}
\diam(\bar K_r)\leq c e^{-(r-d)/2} \min\{|l_1|,\,|l_2|\},\qquad \min\{|l_1|,\,|l_2|\}\leq \pi\,\dist(\psi(U_{\tau_r}),\bar K_r),
\end{equation}
where $c$ is an absolute constant. Note that $\Phi'_{\tau_r}(\psi(U_{\tau_r}))$ is the probability that a Brownian motion from 0 to $\psi(U_{\tau_r})$ in $\disk$ does not hit $\bar K_r$.  One can verify that for some constant $c_*$,
\[
1-\Phi'_{\tau_r}(U_{\tau_r}) \leq c_* \diam(\bar K_r)^2 \dist(U_{\tau_r},\bar K_r)^{-2}.
\]
To see this let $\sigma$ be the first time Brownian motion $B$ starting from the origin exits $\disk\sm \bar K_r$. Then 
\[
1-\Phi'_{\tau_r}(U_{\tau_r}) = \frac{\EE[H_\disk(B_\sigma,U_{\tau_r})1\{B_\sigma\in\disk\}]}{H_\disk(0,U_{\tau_r})}.
\]
Moreover, 
\[
\PP[B_\sigma\in \disk] = \bigo(\diam(\bar K_r)). 
\]
Using the exact form of Poisson kernel in $\disk$, there exists a constat $c$ such that for any $w\in \bar K_r$,
\[
H_\disk(w,U_{\tau_r}) < c \frac{\diam(\bar K_r)}{\dist(\psi(U_{\tau_r}),\bar K_r)^2}.
\]
The result follows from this and equation \eqref{eqextra2}.
\end{proof}
\begin{proposition}\label{prop2sided1}
Suppose $\gamma_t$ is a radial $\sle$ from $1$ to 0 in $\disk$. Let $\tau = \inf\{t;\,\gamma_t\not\subset D\}$ and denote by $U_t$  the continuous process satisfying $e^{iU_t} = \gg (\gamma(t)),\,U_0 = 0$. Then
\begin{equation}\label{eq2sided-}
M_t =1\{t<\tau\}\;  \Phi'_t(U_t)^b\; \bar{\Phi}'_t(0)^{\tilde{b}}\;\exp\left\{\frac{\cc}{2}\,m_\disk(\gamma_t,K)\right\},\qquad \tilde{b}= \frac{b}{6} +\frac{\cc}{12},
\end{equation}
is a uniformly integrable martingale and
\[
\frac{d\mu_{D}(1,0)}{d\mu_\disk(1,0)}(\gamma_t) = M_t.
\]
Moreover, 
\[
M_\infty := \lim_{t\to \infty} M_t = 1\{\gamma\subset D\}\,\exp\left\{\frac{\cc}{2}\,m_\disk(\gamma,\disk\sm D)\right\}.
\]
\end{proposition}

The proof follows the outline   of 
  boundary perturbation for chordal $\sle$
  with one
extra consideration. 
\begin{itemize}
\item Use stochastic calculus to find an appropriate
local martingale.
\item Identify some of the terms in the martingale as
Brownian loop terms.  For the radial case one separates
the loops of zero winding number from those of
nonzero winding number.
\item  Use Girsanov theorem to see that the process
obtained by tilting by the local martingale is the same
as $\sle$ in the smaller domain and use this to conclude
that the martingale is uniformly integrable.
\end{itemize}
What the proof will show is that radial $\sle$ in $D$
is radial $\sle$ in $\disk$ ``weighted locally by
$\Phi_t'(U_t)^b$''  Indeed the local martingale $M_t$ can be
viewed as $C_t \, \Phi_t'(U_t)^b$ where $C_t$ is the
differentiable process needed to make it a local martingale.
\begin{proof}
Assume $\gamma_t$ is a radial $\sle$ with respect to a probability space $(\Omega,\mathcal{F},\PP)$. 
With respect to this probability space, $U_t$ is a standard Brownian motion. If $t<\tau$, then
 we can see from Lemma \ref{lemback2} that
\begin{equation}\label{eq2sided0.7}
\log\bar\varphi_t'(0) = \frac{a}{2}\int_0^t \Phi'_s(U_s)^2\,ds.
\end{equation}
Note that if $h_t:=\tg_t^{-1}$, then $\Phi_t=\varphi_t\circ G\circ h_t$. Using this, the chain rule and the radial Loewner equation, we can see that 
\begin{align*}
\dot\Phi_t(z) &= \dot{\varphi}_t(G\circ h_t(z)) + \varphi_t'(G\circ h_t(z))\,G'(h_t(z))\,\dot h_t(z),\\
& = \frac{a\Phi_t'(U_t)^2}{2}\cot_2(\Phi_t(z)-U_t^*) + \Phi_t'(z)\,\tg_t'(h_t(z))\,\dot h_t(z),\\
& = \frac{a\Phi_t'(U_t)^2}{2}\cot_2(\Phi_t(z)-U_t^*) - \frac{a\Phi_t'(z)}{2}\cot_2(z-U_t).
\end{align*}
Taking limit as $z\to U_t$ gives 
\begin{equation}\label{eq2sided0.8}
\dot\Phi_t(U_t) = -3a\,\Phi''_t(U_t)/2.
\end{equation}
 Moreover, we can take derivative (with respect to $z$) of the right-hand-side of the equation above and let $z\to U_t$ to get 
\begin{equation}\label{eq2sided0.9}
\dot{\Phi}_t'(U_t) = \frac{a}{2}\left[\frac{\Phi_t''(U_t)^2}{2\Phi'_t(U_t)}-\frac{4\Phi_t'''(U_t)}{3}\right].
\end{equation}
Therefore, an application of the It\^o's formula gives
\begin{equation}\label{eq2sided1}
dU^*_t =d\Phi_t(U_t)= -b\, \Phi''_t(U_t) \,dt + \Phi'_t(U_t)\, dU_t.
\end{equation}
If $\gamma_t$ is a $\sle$ in $D$, then $\bar\gamma_t$ is a (time change of) $\sle$ in $\disk$ and $U^*_t$ is a Brownian motion (with an appropriate time change). Let $Z_t = \Phi_t'(U_t)^b$. Using the It\^o's formula and \eqref{eq2sided0.9}, we get
\[
\frac{dZ_t}{Z_t}= \left[\frac{a\cc}{12}S\Phi_t(U_t) + \frac{ab}{12}(1-\Phi'_t(U_t)^2)\right]dt + b\frac{\Phi''_t(U_t)}{\Phi'_t(U_t)}dU_t.
\]
Here, $S$ denotes the Schwarzian derivative
\[
Sf(z) = \frac{f'''(z)}{f'(z)} -\frac{3f''(z)^2}{2f'(z)^2}.
\]

So far we have done straightforward stochastic calculus.
   Now we view these quantities in terms
of Brownian loops that intersect both $K$ and $\gamma_t$.
We recall  (see \cite{greg_loop} or \cite[Proposition 5.22]{greg_book}), that if $K' \subset \Half$ and $\gamma$ is
a chordal $SLE$ starting at the origin, then the measure
of loops that intersect both $K'$ and $\gamma_t$ is
$-S\Phi (0)at/6 + o(t)$ where $\Phi: \Half \setminus K'
\rightarrow \Half$ is a conformal transformation.
Using this we can see that 
if $K=\disk\sm D$, then
\[
\hat m_\disk(\gamma_t,K) := -\frac{ a}{6}\int_0^t S\Phi_s(U_s)ds
\]
is the Brownian loop measure of loops in $\disk$ that have zero winding number and intersect both $\gamma_t$ and $K$ (having
zero winding number means that it is an image under $\psi$
of a loop in $\Half$). 
By using  Lemma \ref{lem2sided1} and an straightforward inclusion-exclusion argument we can see that the Brownian loop measure of loops with nonzero winding number that intersect both $\gamma_t$ and $K$ is
\[
\frac{1}{6}\,[ \log \bar{G}'(0) - \log \bar\Phi_t'(0)]= \frac{a}{12}\int_0^t(1-\Phi_s'(U_s)^2)ds.
\]
Here for the second equality, we have used \eqref{eq2sided0.7} and the fact  $\bar\Phi_t\circ \gg= \bar\varphi_t\circ \bar G$.
Therefore, for $t<\tau$,
\[
\left[\frac{a\cc}{12}S\Phi_t(U_t) + \frac{ab}{12}(1-\Phi'_t(U_t)^2)\right] = -\frac{\cc}{2}\,m_\disk(\gamma_t,\,K) + \tilde{b}\,[ \log \bar\Phi_t'(0)-\log \bar{G}'(0)]
\]
and
\[
M_t = \Phi'_t(U_t)^b\, \bar{\Phi}'_t(0)^{\tilde{b}}\,\exp\left\{\frac{\cc}{2}\,m_\disk(\gamma_t,K)\right\},\qquad t<\tau,
\]
is a local martingale satisfying 
\[
dM_t = b\frac{\Phi''_t(U_t)}{\Phi'_t(U_t)}\,M_t \, dU_t,\qquad M_0 = G'(0)^b \,\bar{G}'(0)^{\tilde{b}}.
\]
Let $\tau_n = \min(n,\,\inf\{t;\,\dist(\gamma_t,K)<e^{-n}\})$.
It is easy to see that 
for all $n\in \mathbb{Z}^+$, the local martingale $M_{t\wedge\tau_n}$ is uniformly bounded and hence is a martingale.

Let $\hat\PP_n$ be the probability measure obtained from weighting $\PP$ by $M_{\tau_n}/M_0$. With respect to $\hat\PP_n$, equation \eqref{eq2sided1} becomes
\[
dU^*_t = \Phi'_t(U_t)\,  dW_t,\qquad t<\tau_n,
\]
where $W_t$ is a standard Brownian motion. That is, with respect to $\hat\PP_n$ the curve $\gamma_{t\leq \tau_n}$  has the distribution of (a time change of) radial $\sle$ from 1 to 0 in $D$. 
Let $\hat\PP=\hat\PP_\infty$ denote the probability measure obtained from applying the Kolmogorov extension theorem to the consistent measures $\hat\PP_n$. Since $\tau_n\uparrow\tau$, with respect to $\hat\PP$ and for any $t<\tau$ the curve $\gamma_t$ has the distribution of radial $\sle$ in D. Since radial $SLE_{\kappa\leq 4}$ in $D$ stays at a positive distance from $K$ and goes to 0, we have $m_\disk(\gamma_\infty,K)<\infty$. It follows from this and Lemma \ref{lemextra1} that with $\hat\PP$-probability one $M_\tau=M_\infty<\infty$. From this, we get
\[
\frac{d\hat\PP}{d\PP} = \frac{M_\tau}{M_0}
\]
(e.g. see Theorem 5.3.3 in \cite{durrett}).
Since $\kappa\leq 4$,  we have $M_\tau =0$ in case $\tau<\infty$. If $\tau = \infty$, the result follows from  Lemma \ref{lemextra1}. 
\end{proof}

\subsection{Two-sided \texorpdfstring{$\sle$}{LG}}
Suppose $\gamma^1_t,\gamma^2_s$ are two independent radial $\sle$ in $\disk$, as shown in figure \ref{pic1}
each parametrized by its own capacity. Define $\tau = \inf\{t:\,\gamma^1_t\cap\gamma^2_t \neq \emptyset\}$ to be the first time the curves intersect.
Let $t,s<\tau$ and define the following
(see Figure \ref{pic1}). 

\begin{itemize}
\item Let $\gg:\disk\sm\gamma^1_t \to \disk$ be the unique conformal transformation satisfying $\gg(0) = 0,\,\gg'(0)>0$.

\item Let $ \tg_t$ be the conformal transformation that is continuous in $t$ and  
\[
\gg(e^{iz}) = e^{i\tg_t(z)},\qquad \tg_0(z) = z.
\]

\item Let $\bar\gamma^2_{t,s} = \gg(\gamma^2_s)$. 

\item Similarly  to $\gg,\,\tg_t$, define the conformal transformations $\bar{G}_s,\, G_s$ for $\gamma^2_s$ and let $\bar\gamma^1_{t,s}  = G_s(\gamma^1_t)$.  

\item Define $\bar \Phi_{t,s} = e^{i \Phi_{t,s}}$
and $ \bar  \phi_{t,s} = e^{i \phi_{t,s}}$ as in the 
figure and let
\[    \bar g_{t,s} = \bar \Phi_{t,s} \circ \bar g_t
   = \bar \phi_{t,s} \circ \bar G_s.\]

\item  We assume $\gamma_t^1,\gamma_t^2$ each
have independent radial parametrizations so that 
 $\log \gg'(0) = at/2,$ $\,\log \bar{G}'_s(0) = as/2$.
 
\item Let $U_t,X_s,U_{t,s}^*,X_{t,s}^*$ be as indicated by
Figure \ref{pic1}.  Under this parametrization, $U_t,\,X_s$ are independent standard Brownian motions with respect to a probability space $(\Omega,\mathcal{F},\PP)$. 

\item We will write $\Phi_t,\,\varphi_t,\,U^*_t,\,X^*_t,\,Z_t,\,\text{etc.}$ for $\Phi_{t,t},\,\varphi_{t,t},\,U^*_{t,t},\,X^*_{t,t},\,Z_{t,t},\,\text{etc.}$.

\item Let 
\begin{equation}\label{eq2sided1.5}
Z_{t,s} =  \Phi'_{t,s}(U_t)^b\,\varphi'_{t,s}(X_s)^b,
\end{equation}
\begin{equation}\label{fundiego}
\Theta_t = X^*_t - U^*_t.
\end{equation}

\end{itemize}

\begin{figure}
  \centering
    \includegraphics[width=0.9\textwidth]{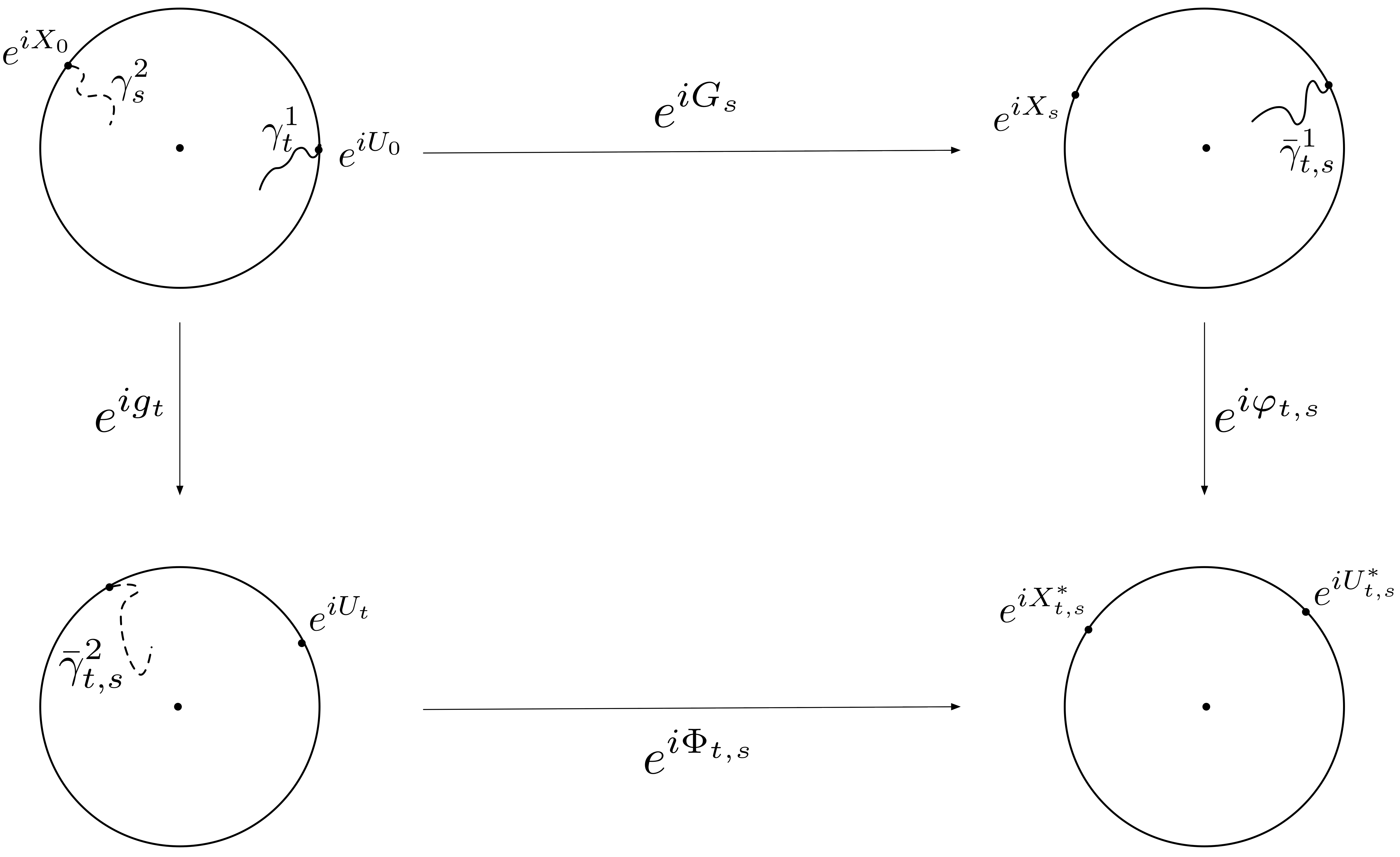}
    \caption{Two-sided $\sle$}
    \label{pic1}
\end{figure}

We will now consider a process that can be called
``locally independent $\sle$''.  It is the process such
that given $\gamma^1_t$ and $\gamma^2_s$, the
paths are moving like independent radial $\sle$ paths
in $\Disk \setminus (\gamma^1_t \cup \gamma^2_s)$. 
In analogy to the boundary perturbation case, we can
describe the process as independent $\sle$ paths ``weighted
locally by $Z_{t,s}$''.  To make this precise, we first find an appropriate compensator to make this a local martingale.

\begin{lemma}\label{lem2sided2}
The Brownian loop measure of loops in $\disk$ that intersect both $\gamma^1_t$ and $\gamma^2_s$ and have nonzero winding number is
\[
\frac{\log\bar G_s'(0)-\log \bar \Phi'_{t,s}(0)}{6} = \frac{\log\gg'(0)-\log \bar \varphi'_{t,s}(0)}{6}.
\]
\end{lemma}
\begin{proof}
This follows from Lemma \ref{lem2sided1} and an easy inclusion-exclusion argument.
\end{proof}
\begin{lemma}
Assume $t<\tau$, $U_t,X_t$ are independent standard
Brownian motions, and let $K_t = \Phi'_t(U_t)^2 + \varphi'_t(X_t)^2$. Then
\[
M_t = Z_t \bar{\Phi}'_t(0)^{\tilde{b}}e^{-\tilde{b}at/2} \exp\left\{\frac{\cc}{2}m_\disk(\gamma^1_t,\gamma^2_t) + \frac{ab}{4}\int_0^t\frac{K_r}{\sin_2(\Theta_r)^2} dr\right\},
\]
is a local martingale satisfying 
\[
dM_t = bM_t\left[\frac{\Phi''_t(U_t)}{\Phi_t'(U_t)}dU_t +\frac{\varphi''_t(X_t)}{\varphi'_t(X_t)}dX_t\right].
\]
Here, $\tilde{b}$ is the interior scaling exponent defined in \eqref{constants}.
\end{lemma}
\begin{proof}
The It\^o's formula and similar calculations as in the last section give
\[
\frac{d{\Phi}'_{t,s}(U_t)^b}{{\Phi}'_{t,s}(U_t)^b} = \left[\frac{a\cc}{12} S\Phi_{t,s}(U_t) + \frac{ab}{12}(1-\Phi_{t,s}'(U_t)^2)\right]dt -\frac{ab\,\varphi'_{t,s}(X_s)^2}{4\sin_2(U_{t,s}^*-X_{t,s}^*)^2} ds + b\frac{\Phi''_{t,s}(U_t)}{\Phi'_{t,s}(U_t)}dU_t.
\]
Here, $\sin_2(x) = \sin(|x|/2)$.
We can derive a similar formula for $\varphi'_{t,s}(X_s)^b$. Using the two formulas and the
independence of $U_t,X_t$, we have
\begin{align*}
\frac{dZ_t}{Z_t} &= \left[\frac{a\cc}{12} S\,\Phi_t(U_t) + \frac{ab}{12}(1-\Phi_t'(U_t)^2)-\frac{ab\,\Phi'_t(U_t)^2}{4\sin_2(\Theta_t)^2}\right] dt\\
&\; + \left[\frac{a\cc}{12} S\,\varphi_t(X_t) + \frac{ab}{12}(1-\varphi_t'(X_t)^2)-\frac{ab\,\varphi'_t(X_t)^2}{4\sin_2(\Theta_t)^2}\right] dt \\
&\; + b\frac{\Phi''_t(U_t)}{\Phi'_t(U_t)}dU_t + b\frac{\varphi''_t(X_t)}{\varphi'_t(X_t)}dX_t.
\end{align*}
Here, we are also using the fact that $\Phi'_{t,s}(U_t),\,\varphi'_{t,s}(X_s)$ are $C^1$ in $t,s$.
 Note that
\[
m_1(t):=\frac{-a}{6}\int_0^t S\,\Phi_{r}'(U_r)dr
\]
is the Brownian loop measure of loops $l$ in $\disk$ that have the following properties:
\begin{itemize}
\item Winding number of $l$ is zero.
\item $l$ intersects  both $\gamma_t^1$ and $\gamma_t^2$.
\item If $T\leq t$ is the first time $l$ intersects $\gamma^1_t$, then $l\cap \gamma^2_{T}\neq \emptyset$.
\end{itemize}
The term
\[
m_2(t):=-\frac{ a}{6}\int_0^t S\varphi_{r} (X_r)dr 
\]
has a similar interpretation for $\gamma^2_t$. Hence, $\hat m_\disk(\gamma^1_t,\gamma^2_t):=m_1(t)+m_2(t)$ is the Brownian loop measure of loops in $\disk$ that have zero winding number and intersect $\gamma^1_t,\gamma^2_t$ (we are using the fact that measure of the loops hitting $\gamma^1_t,\gamma^2_t$ at the same time is zero). Moreover, 
\[
\frac{a}{12}\int_0^t(1-\Phi'_r(U_r)^2) dr
\]
is the Brownian loop measure of loops $l$ in $\disk$ with the following properties:
\begin{itemize}
\item $l$ has nonzero winding number.
\item $l$ intersects  both $\gamma_t^1$ and $\gamma_t^2$.
\item  $l$ intersects $\gamma^1$ first, that is,
if $s\leq t$ is the smallest time with $\gamma^1(s)
 \in l$, then  $l\cap \gamma^2_{s}\neq \emptyset$.
\end{itemize}
To see this, let $\tilde m_1(t)$ be the measure of loops having the properties above. Using Lemma \ref{lem2sided2} and  \eqref{eq2sided0.7},  the Brownian loop measure of loops in $\disk\sm \gamma^1_t$ that intersect both $\gamma^2_t$ and $\gamma^1(t,t+\epsilon)$ is
\[
\frac{a}{12}\left[\epsilon - \int_0^\epsilon \Phi'_{t+r,t}(U_{t+r})^2dr\right].
\]
Moreover, the Brownian measure of loops in $\disk\sm \{\gamma_t^1\cup \gamma_t^2\}$ that intersect both $\gamma^1(t,t+\epsilon)$ and $\gamma^2(t,t+\epsilon)$ is $O(\epsilon^2)$. Therefore,
\[
\partial_t \tilde m_1(t) = \frac{a(1 - \Phi'_t(U_t)^2)}{12},
\]
and by integrating the claim follows.  
Using a similar argument for $\gamma^2_t$, we can see that
\begin{equation}\label{eq2sided2}
\tilde{m}_\disk(\gamma^1_t,\gamma^2_t)=\frac{a}{12}\left[\int_0^t(1-\Phi'_r(U_r)^2) dr + \int_0^t(1-\varphi'_r(X_r)^2)dr\right]
\end{equation}
is the Brownian loop measure of loops that intersect both $\gamma^1_t$ and $\gamma^2_t$ and have nonzero winding number.
Note that $ab/12 = a\tilde{b}/2 - a\cc/24$. It follows from \eqref{eq2sided2} and Lemma \ref{lem2sided2} that
\[
\frac{ab}{12}\left[\int_0^t\left(2-\Phi'_r(U_r)^2 -\varphi'_r(X_r)^2\right)dr\right] = -\frac{\cc}{2}\tilde{m}_\disk(\gamma^1_t,\gamma^2_t)+\frac{\tilde b at}{2} - \tilde{b}\,\bar\Phi'_t(0).
\]
Therefore,
\[
M_t = Z_t\bar{\Phi}'_t(0)^{\tilde{b}}e^{-\tilde{b}at/2} \exp\left\{\frac{\cc}{2}m_\disk(\gamma^1_t,\gamma^2_t)+\frac{ab}{4}\int_0^t\frac{K_r}{\sin_2(\Theta_r)^2} dr\right\}
\]
is a local martingale satisfying the claim. 
\end{proof}


Let
\[
\tau_n = \inf\{t;\,\dist(\gamma^1_t,\gamma^2_t)\leq e^{-n})\}
\]
and recall that $\tau = \tau_\infty$.
Although $M_{t\wedge \tau}$ is only a supermartingale (positive local martingale), one can see that $M_{t\wedge \tau_n}$ is actually a martingale. To see this, note that $Z_t\leq 1,\,K_t\leq 2,\,\bar\Phi_t'(0)\leq e^{at/2}$ and $m_\disk(\gamma^1_t,\gamma^2_t)$ is uniformly bounded for all $t\leq \tau_n$. 
Let $\PP^*$ be the probability measure obtained from weighting $\PP$ by $M_{t\wedge \tau_n}$. 
Using the radial Loewner equation and equations \eqref{eq2sided0.7}, \eqref{eq2sided0.8}, we can see that
\begin{align*}
dU^*_{t,s} &= -b\,\Phi''_{t,s}(U_t) \,dt+ \frac{a}{2}\,\varphi'_{t,s}(X_s)^2\cot_2\left({U_{t,s}^*-X_{t,s}^*}\right) ds + \Phi'_{t,s}(U_t) \,dU_t,\\
dX^*_{t,s} &=   -b\,\varphi''_{t,s}(X_s)\,ds + \frac{a}{2}\Phi'_{t,s}(U_t)^2\cot_2\left({X_{t,s}^*-U_{t,s}^*}\right)\,ds + \varphi'_{t,s}(X_s) dX_s.
\end{align*}
Using the Girsanov's theorem and the It\^o's formula, we have
\[
d\Theta_t = \frac{a}{2}K_t\cot_2(\Theta_t)dt+\sqrt{K_t}dW_t,
\]
where $W_t$ is a Brownian motion with respect to $\PP^*$. By comparing to a radial Bessel process, we can see that with $\PP^*$-probability one, $\sin_2(\Theta_t)>0$ for all times $t<\infty$.
 Therefore,  $M_{t\wedge \tau_n}<\infty$ with $\PP^*$-probability one and the claim follows.

The curves $\gamma^1_t,\gamma^2_t$ have interesting distributions under the measure $\PP^*$. 
Fix $t>0$ and assume $\epsilon>0$ is small. With respect to measure $\PP$, the curve $\gamma^1(t,t+\epsilon)$ grows like radial $\sle$ from $\gamma^1(t)$ to $0$ in $\disk\sm\gamma^1_t$. 
Equivalently, $\gg(\gamma^1(t,t+\epsilon))$  has the distribution of radial $\sle$ from $U_t$ to 0 in $\disk$. 
According to Proposition \ref{prop2sided1}, weighing this process by $\Phi'_t(U_t)^b$ yields a radial $\sle$ from $\gamma^1(t)$ to 0 in $\disk\sm \{\gamma^2_t\cup\gamma^1_t\}$. 
Similarly, weighing $\gamma^2(t,t+\epsilon)$ by $\varphi_t'(X_t)^2$ gives radial $\sle$ from $\gamma^2(t)$ to 0 in $\disk\sm \{\gamma^2_t\cup\gamma^1_t\}$. 
Therefore under the probability measure $\PP^*$, at each time $t$ the curves $\gamma^1_t,\gamma^2_t$ grow like independent radial $\sle$ in $\disk\sm \{\gamma^2_t\cup\gamma^1_t\}$. 
We can call this process {\em locally independent
$\sle$}.

The calculation above was straightforward but a little complicated. It is easier to view locally independent
$\sle$ in a slightly different
parametrization.
  Let us suppose   that $(\gamma^1_t,\gamma^2_t)$ are curves as above but with a different
parametrization that is absolutely continuous with
respect to the individual capacities.  We replace 
the sixth bullet with
\begin{itemize}
\item  We assume that the parametrization is such that
 $\log\bar g_{t,t}'(0) = at$ and
\[      \log\bar g_{t+\delta,t}'(0) = at + \frac{a\delta}{2} + o(\delta), \;\;\; \delta \downarrow 0.\]
That is, each curve increases its capacity at rate $a/2$ where
the capacity is measured in $\Disk \setminus (\gamma_t^1 \cup
\gamma_t^2)$. 
Then in this new parametrization,
\[    \p_t g_{t,t}(z) = \frac{a}{2} \, \cot_2( g_{t,t}(z)
 - U_t^*) +  \frac{a}{2} \,\cot_2(g_{t,t}(z) - X_t^*) .
 \]
where
\[     dU_t^* = \frac{a}{2} \, \cot_2(U_t^*
  - X_t^*) \, dt + dB_t^1,\]
  \[  dX_t^* = 
  \frac{a}{2} \, \cot_2(X_t^*
  - U_t^*) \, dt + dB_t^2. \]  
We note that the above equations are the same as one would
get if one started with independent Brownian motions
$B_t^1,B_t^2$ and then tilted by the local martingale
$  N_t =  C_t\, |\sin_2(X_t- U_t)|^a$ where $C_t$
is a $C^1$ compensator.

\end{itemize}

Before taking our next steps, we briefly recall the definition of two-sided $\sle$. 
Roughly speaking two-sided $\sle$ from $\ubar=e^{iU_0}$ to $\xx=e^{iX_0}$ is chordal $\sle$ conditioned to go through the origin. 
  it can be defined precisely as chordal $\sle$ weighted by the Green's function, which is proportional to $\sin_2(g_t(X_0)-U_t)^{4a-1}$. Equivalently, it can be considered as  radial $\sle$ from $\ubar$ to 0 weighted by $\sin_2(g_t(X_0)-U_t)^a$. 
After reaching the origin (say at time $T$), the rest of the process has the distribution of chordal $\sle$ from $0$ to $\xx$ in $\disk\sm\gamma_T$ (see \cite{greg_twosided} for more details). Straightforward calculation using the It\^o's formula show that 
\begin{equation}\label{eq2sided0}
\sin_2\left({g_t(X_0) - U_t}\right)^a g'_t(X_0)^b e^{3a^2t/8}
\end{equation} 
is a martingale and tilting radial $\sle$ by this martingale gives two-sided $\sle$.

 We consider weighting the measure $\PP^*$ by $\sin_2(\Theta_t)^a$. Straightforward calculation using the It\^o's formula shows  that 
\[
N_t = \sin_2(\Theta_t)^a\exp\left\{\frac{3a^2}{8}\int_0^tK_r\,dr-\frac{ab}{4}\int_0^t\frac{K_r}{\sin_2 (\Theta_r)^2} dr\right\}
\]
is a local martingale satisfying
\[
dN_t = \frac{a}{2}N_t \cot_2\left({\Theta_t}\right)\sqrt{K_t}\,dW_t.
\]
Since $K_r\leq 2$ for any $r$, $N_{t\wedge \tau}$ is actually a martingale.
Let $\hat\PP$ be the probability measure obtained from weighting $\PP^*$ by $N_t$. Equivalently, $\hat\PP$ is the probability measure obtained from weighting $\PP$ by $O_t:=M_tN_t$. 
Using Lemma \ref{lem2sided2} and equation \eqref{eq2sided2}, we can see that
\begin{equation}\label{eq2sided3}
O_t = Z_t\sin_2(\Theta_t)^a\,\bar\Phi'_t(0)^{(\tilde{b}+3a/4)}\exp\left\{\frac{\cc}{2}m_\disk(\gamma^1_t,\gamma^2_t)+ \frac{3a^2t}{8}-\frac{a\tilde{b}t}{2}\right\}.
\end{equation}
Using the Girsanov's theorem, we can see that there exists a Brownian motion $ B_t$ such that with respect to the measure $\hat \PP$,
\[
d\Theta_t = aK_t \cot_2\left({\Theta_t}\right)dt + \sqrt{K_t}\,dB_t.
\]

\begin{proposition}\label{prop2sided3}
Let $\bar u = e^{iU_0},\, \xx = e^{iX_0}$ and define the measure $\nu_t(\bar u, \xx)$ with
\[
\frac{d\,\nu_t(\bar u, \xx)}{d \,\mu_\disk(\bar u,0)\times \mu_\disk( \xx,0)}(\gamma^1_t,\gamma^2_t) = O_t\,1\{t<\tau\},
\]
where $O_t$ is defined in \eqref{eq2sided3}.
Then with respect to the measure $\nu_t(\ubar,\xx)$,
\begin{itemize}
\item Marginal distribution of $\gamma^1_t$ is two-sided $\sle$ from $\ubar$ to $\xx$.
\item Marginal distribution of $\gamma^2_t$ is two-sided $\sle$ from $\xx$ to $\ubar$.
\item Conditional on $\gamma^1_t$, the process $\gamma^2_t$ has the distribution of two-sided $\sle$ from $\xx$ to $\gamma^1(t)$ in $\disk\sm\gamma^1_t$. 
\item Conditional on $\gamma^2_t$, the process $\gamma^1_t$ has the distribution of two-sided $\sle$ from $\ubar$ to $\gamma^2(t)$ in $\disk\sm\gamma^2_t$. 
\end{itemize}
\end{proposition}
\begin{proof}
It suffices to show that marginal measure induced on $\gamma^2_t$ by $\nu_t(\ubar,\xx)$ is two-sided $\sle$ and conditioned on $\gamma^2_t$, the distribution of $\gamma^1_t$ is  two-sided $\sle$ from $\xx$ to $\gamma^2(t)$ in $\disk\sm\gamma^2_t$. 
This is because the local martingale given in \eqref{eq2sided3} is symmetric with respect to $\gamma^1_t,\gamma^2_t$.

From Proposition \ref{prop2sided1} we know that weighting $\mu_\disk(\ubar,0)$ by 
\[
G'_t(U_0)^{-b}\,\bar\Phi'_t(0)^{\tilde{b}} \Phi'_t(U_t)^b e^{-a\tilde{b}t/2} \exp\left\{\frac{\cc}{2}m_\disk(\gamma^1_t,\gamma^2_t)\right\}1\{t<\tau\}
\]
gives $\mu_{\disk\sm \gamma^2_t}(\ubar,0)$ on paths up to time $t$. 
Moreover, we can see from an appropriate time change of \eqref{eq2sided0} that weighting $\mu_{\disk\sm \gamma^2_t}(\ubar,0)$ by 
\[
\bar\varphi'_t(0)^{3a/4}\,\varphi_t'(X_t)^b\sin_2(\Theta_t)^a
\]
gives two-sided $\sle$ from $\ubar$ to $\gamma^2(t)$ in $\disk\sm \gamma^2_t$.
Let $\EE,\,\hat \EE$ be expectations with respect to $\mu_\disk(\ubar,0),\,\mu_{\disk\sm \gamma^2_t}(\ubar,0)$. 
Then using the fact that $\bar\varphi'_t(0) = \bar\Phi_t'(0)$ and that the process given in \eqref{eq2sided0} is a martingale, we get
\begin{align*}
\EE[O_t] &= \hat\EE\left[G'_t(U_t)^b\varphi_t'(X_t)^b\,\bar\varphi'_t(0)^{3a/4}\sin_2(\Theta_t)^a\,e^{3a^2t/8}\right] \\
&= \sin_2(X_t-G_t(U_0))^a\,G'_t(U_t)^b\,e^{3a^t/8}.
\end{align*}
The proof follows from comparing this to \eqref{eq2sided0}.
\end{proof}
\subsection{Bi-chordal annulus \texorpdfstring{$\sle$}{LG}}
In this section, we derive asymptotic estimates for the partition function of bi-chordal annulus $\sle$. 
Let $\ubar,\xx\in C_0$ be distinct boundary points of $\disk$ and let $\gamma_t\subset \disk$ be a simple curve with $\gamma(0) = \ubar$. 
Define $\gg,\,g_t$ to be our usual transformations for $\gamma_t$ and choose  $0\leq u,x<2\pi$  such that $\ubar = \psi(u),\,\xx = \psi(x)$. 
\begin{lemma}\label{2slelem1}
Suppose $\kappa<8$ and $\gamma_t$ is a radial $\sle$ from $\ubar$ to 0 in $\disk$ with radial parametrization.  Then there exist constants $C,\beta>0$ such that for all $t>1$,
\[
\EE\left[H_{\disk\sm \gamma_t}(0,\xx)^b\right] =  C \sin_2(x-u)^ae^{-3a^2t/8}[1 + O(e^{-\beta t})].
\]
As before,  $\sin_2(\theta) = \sin(|\theta|/2)$.
\end{lemma}
\begin{proof}
Although we will only use this lemma when $\kappa\leq 4$, we will prove it here for $\kappa<8$. 
Let $U_t,\,X_t$ be the unique $t$-continuous processes satisfying $U_0 = u, \, X_0 = x$ and $\gg(\gamma(t)) = \psi(U_t),\,\gg(\xx) = \psi(X_t)$. Let  $\Theta_t = X_t - U_t$ and 
define $\sigma=\inf\{t;\, \sin_2({\Theta_t}) = 0\}$ be the first time that $\Theta_t\in\{0,2\pi\}$. 
Since $\gamma_t$ is a radial $\sle$,  $U_t=-B_t$ is a standard Brownian motion and 
\[
d\Theta_t = \frac{a}{2}\cot_2(\Theta_t)dt + dB_t,\qquad t<\sigma,
\]
 where as before $\cot_2(\theta) = \cot (\theta/2)$. Note that when $\kappa\leq 4$,  with probability one $\sigma = \infty$ and the last equation is well-defined for all times $t$. As discussed in \eqref{eq2sided0}, let
\begin{equation}\label{eq2sided4.5}
M_t = \sin_2\left(\Theta_t\right)^a g'_t(x)^b e^{3a^2t/8}
\end{equation}
be a martingale satisfying 
\[
dM_t = \frac{a}{2}\cot_2(\Theta_t) M_t dB_t.
\]
Let $\hat\PP$ be the probability measure obtained from using the Girsanov theorem with the martingale $M_t$. Under the probability measure $\hat{\PP}$, there exists a Brownian motion $W_t$ such that
\begin{equation}\label{eq2sided5}
d\Theta_t = a\cot_2(\Theta_t) dt + dW_t.
\end{equation}
Since $\kappa<8$, we have $2a>1/2$. Comparing this to a radial Bessel process, we can see that with $\hat\PP$-probability one $\sigma = \infty$. 
Since $H_{\disk\sm \gamma_t}(0,\xx) = g'_t(x)$ we can write
\begin{align*}
\EE\left[H_{\disk\sm \gamma_t}(0,\xx)^b\right]  &= \EE\left[g'_t(x)^b;\,t<\sigma\right]\\
&=  \EE\left[M_t\sin_2(\Theta_t)^{-a}e^{-3a^2t/8};\,t<\sigma\right]\\
&=M_0\, e^{-3a^2t/8}  \,\hat\EE\left[\sin_2(\Theta_t)^{-a};\,t<\sigma\right]\\
&= \sin_2(\Theta_0)^ae^{-3a^2t/8}\,\hat\EE\left[\sin_2(\Theta_t)^{-a}\right].
\end{align*}
The last equation holds because $\hat\PP\{\sigma =\infty\} =1 $. It only remains to compute $\hat\EE\left[\sin_2(\Theta_t)^{-a}\right]$. The function 
\[
f(x) = c_{4a} \sin_2(x)^{4a},\qquad c_{4a} = \left[\int_0^{2\pi} \sin_2(x)^{4a} dx\right]^{-1}
\]
satisfies the adjoint equation of \eqref{eq2sided5} and therefore is the invariant density of $\Theta_t$. 
Let $\tilde f_t(\theta,x)$ be the transition density of $\Theta_t$ starting at $\Theta_0 = \theta$.  It follows from the properties of radial Bessel processes (see section 4 of \cite{Japan}) that there exists $\beta>0$ such that for all $\theta,\,x$ and $t>1$,
\[
\tilde{f}_t(\theta,x) = f(x)[1 + O(e^{-\beta t})].
\]
Therefore,
\[
\EE\left[H_{\disk\sm \gamma_t}(0,\xx)^b\right] = C \sin_2(\Theta_0)^ae^{-3a^2t/8}[1 + O(e^{-\beta t})],
\]
where 
\[
C = c_{4a}\int_0^{2\pi}\sin_2(x)^{3a}dx .
\]
\end{proof}

\begin{lemma}\label{2slelem2}
For every $\epsilon_0>0$ and $r_0>\pi$, there exists $c_0>0$ such that the following holds. Assume $0\leq u,x,w,y< 2\pi$ and $\pi<r<r_0$. 
Let
\[
\epsilon = \min\{|u-x+2k\pi|,\,|w-y+2m\pi|;\,m,k\in\{-1,0,1\}\}.
\]
Recall the partition function $\pf$ defined in \eqref{eq2}.
If $\epsilon>\epsilon_0$, then $\pf(r,(u,x),(w,y))>c_0$.
\end{lemma}
\begin{proof}
Assume $u = 0$ and let $\gamma_t$ be a $\sle$ curve from $1$ to $\bar w=\psi(w+ir)$ in $A_r$ with annulus parametrization. 
We can assume $0<x$, $0\leq w\leq\pi$ and $w<y<2\pi+w$, since the other cases can be proved in a similar way. Let $\xx = \psi(x),\,\yy=\psi(y +ir)$.
From the definition, 
\[
\pf(r,(u,x),(w,y)) = \EE\left[Q_{A_r}(\xx,\bar y;\gamma_r)^b\right],
\]
where $\EE$ denotes the expectation with respect to the distribution of $\gamma_t$. The goal is to show that there exist $p^*>0$ and $c^*>0$ such that $Q_{A_r}(\xx,\bar y;\gamma_r)>c^*$ with probability at least $p^*$.

Let $\eta_t\subset S_r$ be the unique continuous curve starting from 0, ending at $w +2k\pi+ir$ and satisfying $\gamma_t = \psi (\eta_t)$ for $0\leq t\leq r$. Here, $k$ is uniquely determined by the winding number of $\gamma$.
Let $D_w$ denote the parallelogram created by the intersections of $\mathbb{R},\, \mathbb{R} + ir$, the line connecting $\epsilon_0/2,\, w+\epsilon_0/2+ir$ and the line connecting $-\epsilon_0/2,\,w-\epsilon_0/2+ir$.
First, there exists $p_1>0$ such that for all $0\leq w\leq \pi$, the probability that $\gamma_r\subset D_w$ is at least $p_1$. This is because uniformly over $0\leq w\leq \pi$ and $\pi\leq r\leq r_0$, there is a positive probability that the winding number of $\sle$ from $\ubar$ to $\wbar$ in $A_r$ is zero and therefore, $\eta(r) = w+ir$.
Given this, the distribution of $\eta_t$ is absolutely continuous with respect to the distribution of a chordal $\sle$ from 0 to $w+ir$ in $S_r$. 
The Radon-Nikodym is bounded if $\eta_r\subset D_w$. 
Moreover, there exists $p_0>0$ such that for all $0\leq w\leq\pi$ and $\pi\leq r\leq r_0$, the probability that chordal $\sle$ from 0 to $w+ir$ in $S_r$ does not exit $D_w$ is at least $p_0$. 
To see this, let $f_w:S_r\to\hp$ be a conformal transformation with $f_w(w+ir) = \infty, f_w(0) =0$. 
The domain $f_w(D_w)$ is a simply connect subdomain of $\hp$ and $\hp\sm f_w(D_w)$ is bounded. Note that if $\eta_t^*$ is a chordal $\sle$ from 0 to $\infty$ in $\hp$, then $|\eta^*(t)|\to\infty$ as $t\to \infty$. Hence, for each $w\in[0,\pi]$, the $\sle$ curve $\eta^*$ has a positive probability of staying in $f_w(D_w)$ (see \cite{greg_book} for a proof).  In addition, we can see that this probability is a continuous function of $w$. From this, the claim follows. 


Let $K_{w}$ denote the  connected component of $S_r\sm\,\{D_w\cup D_w+2\pi\}$ that has $x,\,y+ir$ on its boundary. Then there exists a constant $c_1>0$ such that for  all $\gamma_r\subset \psi(D_w)$, $0\leq w\leq \pi,\, \epsilon_0\leq x\leq 2\pi-\epsilon_0$ and $w+\epsilon_0\leq y\leq w+\epsilon_0+2\pi$,
\[
Q_{A_r}(\xx,\bar y;\gamma_r) \geq c_1 \,Q_{S_r}(x,y+ir;\,K_w).
\]
Finally, there exists a constant $c_2$ such that
\[
H_{K_w}(x,y+ir) > c_2.
\]
\end{proof}

\begin{lemma}\label{2sleeq1}
Let $\gamma_t\subset A_r$ be a simple curve with $\gamma(0+) = 1,\, |\gg'(0)| = e^{at/2},\, t\leq \frac{2(r-4)}{a}$ and let $\xx\in C_0,\,\yy\in C_r$. If $r_t = r-at/2$, then
\begin{equation*}
Q_{A_r}(\yy,\xx;\gamma_t) = \frac{2r}{r_t}H_{\disk\sm\gamma_t}(0,\xx)\left[1+\bigo\left(r_te^{-r_t}\right)\right].
\end{equation*}
\end{lemma}
\begin{proof}
We can write
\begin{equation}\label{eq2sletemp}
H_{\disk\sm\gamma_t}(0,\xx) = \frac{1}{\pi}\,\int_{C_{r}}G_{\disk\sm\gamma_t}(0,\zbar)\,H_{A_r\sm\gamma_t}(\zbar,\xx)\,|d\zbar|.
\end{equation}
Here, $G_{\disk\sm \gamma_t}(0,\zbar)$ denotes the Brownian  Green's function in $\disk\sm\gamma_t$.
Using the distortion estimates (in a similar way to the proof of Lemma \ref{lemback3}),  we can see that for every $\zbar\in C_r$,
\[
|r_t+\log|\gg(\zbar)||=\bigo(e^{-r_t}). \]   Therefore, 
\[
G_{\disk\sm\gamma_t}(0,\zbar) = \frac{r_t}{2}+\bigo(e^{-r_t}),
\]
\[
H_{A_r\sm\gamma_t}(\zbar,\xx)= H_{A_r\sm\gamma_t}(\yy,\xx)\left[1 + \bigo(r_te^{-r_t})\right],
\]
where the second equality follows from Lemma \ref{lemback0}.
Using this and  \eqref{eq2sletemp} we get
\[
H_{\disk\sm\gamma_t}(0,\xx)=r_t\,e^{-r}\,H_{A_r\sm\gamma_t}(\xx,\yy) \left[1 + \bigo(r_te^{-r_t})\right].
\]
Finally, from Lemma \ref{lemback0} we know that
\[
H_{A_r}(\xx,\yy) = \frac{e^r}{2r}\left[1 + \bigo(re^{-r})\right].
\]
\end{proof}
\begin{proposition}\label{2sleprop1}
Let $\pf(r,(u,x),(w,y))$ be as in \eqref{eq2}. There exist constants $0<c_*,r_*<\infty$ such that for all $0\leq u,x,w,y< 2\pi$ and $r>r_*$,
\[
\pf(r,(u,x),(w,y)) \leq c_*\,r^b\,e^{-3ar/4}\,\sin_2(x-u)^a.
\]
Furthermore, for any $\epsilon>0$, there exists a constant $c_{\epsilon}>0$ such that if $\min\{|y-w+2k\pi|;\,k\in\{-1,0,1\}\}>\epsilon$, then
\[
c_{\epsilon}\,r^b\,e^{-3ar/4}\,\sin_2(x-u)^a\leq\pf(r,(u,x),(w,y)).
\]
\end{proposition}
\begin{proof}
Suppose $\gamma_t$ is a $\sle$ curve from $\ubar$ to $\wbar$ and let $\mathcal{F}_t$ be the $\sigma$-algebra generated by $\gamma_t$. 
We assume $\gamma_t$ has radial parametrization and let $\tau$ be the hitting time of $C_r$. Let $\hbar_t:A_r\sm\gamma_t \to A_{r(t)},\, h_t:S_{r,t}\to S_{r(t)}$ be as in section \ref{seclow}.
As before, let $\Ubar_t = \hbar(\gamma(t)),\,\XX_t=\hbar_t(\xx),\,\Wbar_t = \hbar_t(\wbar),\,\YY_t = \hbar_t(\yy)$ and  $U_t = h_t(\eta(t)),\,X_t = h_t(x),\,W_t=\Re[h_t(w+ir)],\,Y_t=\Re[h_t(y+ir)]$.  Define 
\[
\epsilon^x_t = \min\{|U_t-X_t+2k\pi|;\,m,k\in\{-1,0,1\}\},
\]
\[
\epsilon^y_t = \min\{|W_t-Y_t+2m\pi|;\,m,k\in\{-1,0,1\}\},
\]
\[
\epsilon_t = \min\{\epsilon^x_t,\,\epsilon^y_t\}.
\]
For a fixed $t<\tau$, let $ \hat\gamma = \hbar_t(\gamma(t,\tau))$. We can write
\begin{align*}
\pf(r,(u,x),(w,y))&= \EE\left[Q_{A_r}(\xx,\yy;\gamma_\tau)^b\right] \\
&= \EE\left[Q_{A_r}(\xx,\yy;\gamma_t)^b \EE\left[Q_{A_{r(t)}}(\XX_t,\YY_t;\hat\gamma)^b\middle|\mathcal{F}_t\right]\right].
\end{align*}
Here, $r(t)$ is as in \eqref{eq2.5} and conditioned on $\mathcal{F}_t$, $\hat\gamma$ is a $\sle$ from $\Ubar_t$ to $\Wbar_t$ in $A_{r(t)}$.
Using  \eqref{ucsd2} and Corollary \ref{corback2}, we can find a constant $s_0>5$ such that if $t = \frac{2(r-s_0)}{a}$, then for all $z\in C_r$, 
\begin{equation}\label{eqluke}
\left\lvert \partial_z\,\arg\hbar_t(z)-1\right\rvert< \frac{1}{2}
\end{equation}
  and
\begin{equation}\label{eqjedi}
\left\lvert \frac{d\mu_{A_r}^\#(\ubar,\wbar)}{d\mu_{\disk}^\#(\ubar,0)}(\gamma_t)- 1\right\rvert < \frac{1}{2}.
\end{equation}
 Koebe-$1/4$ theorem implies that $|s_0 -r(t)| \leq \ln(4)$ and $C_{r-2}\subset \disk\sm\gamma_t$.
Using Lemma \ref{2slelem2}, we can see that there exists $c_0=c_{\epsilon_0/2}$ such that $1>c_0>0$ and
\begin{equation}\label{eqchew}
c_0 1\left\{\epsilon_{t} >\frac {\epsilon_0} 2\right\}<\EE\left[Q_{A_{r(t)}}(\XX_t,\YY_t;\hat\gamma)^b\middle|\mathcal{F}_t\right]<1.
\end{equation}
Let $\bar\PP$ be a probability measure under which $\gamma_t$ is a radial $\sle$ and let $\bar\EE$ denote the expectation with respect to $\bar\PP$. Using Lemma \ref{2sleeq1} and \eqref{eqjedi}, there exists a constant $c$ such that
\[
\frac{1}{c}\leq\frac{\pf(r,(u,x),(w,y))}{r^b\,\bar\EE\left[H_{\disk\sm\gamma_t}(0,\xx)^b\,1\{2\epsilon_{t} >\epsilon_0\}\right]}\leq c.
\]
Let $g_s,\,\xi_s$ be as  in section \ref{secnotation} and define $\tilde{X}_s = g_s(x)$. Equation  \eqref{eqluke} implies that $\epsilon_t^y>\epsilon_0/2$ and therefore
\[
1\left\{\epsilon_t>\frac{\epsilon_0}{2}\right\}=1\left\{\epsilon^x_t>\frac{\epsilon_0}{2}\right\}.
\]
Let 
\[
\tilde\epsilon_t = \min\{|\tilde{X}_t-\xi_t+2k\pi|;\,k\in\{-1,0,1\}\}.
\]
Considering our choice of $t$, Corollary \ref{corback1} implies
\[
\{\tilde\epsilon_t>\epsilon_0\}\subset \left\{\epsilon^x_t>\frac{\epsilon_0}{2}\right\}.
\] 
Let $\Theta_s = \tilde{X}_s - \xi_s$ and note that 
\[
d\Theta_s = \frac{a}{2}\cot_2(\Theta_s) ds + dB_s,
\]
where $B_s$ is a standard Brownian motion with respect to $\bar\PP$.
Consider the martingale
\[
M_s = \sin_2\left(\Theta_s\right)^a H_{\disk\sm\gamma_t}(0,\xx)^b e^{3a^2t/8}
\]
defined in \eqref{eq2sided4.5}.
Let $\tilde{\EE}$ be the expectation with respect to the probability measure obtained from weighing $\bar\PP$ by $M_s$. 
It follows from the Girsanov's theorem  that with respect to the new measure, there exists a Brownian motion $\tilde{B}_s$  such that
\[
d\Theta_s = a\cot_2(\Theta_s) ds + d\tilde B_s.
\]
Using properties of radial Bessel processes in a similar way to the proof of Lemma \ref{2slelem1}, we can see that 
\[
f(x) = c_{4a} \sin_2(x)^{4a},\qquad c_{4a} = \left[\int_0^{2\pi} \sin_2(x)^{4a} dx\right]^{-1},
\]
is the invariant density of $\Theta_s$. Moreover, there exists $0<c_0^*<\infty$ such that if $\tilde f_s(\theta,x)$ denotes the density of $\Theta_s$ starting at $\Theta_0 = \theta$, then for all $\theta,\,x$ and $s>1$ 
\begin{equation}\label{eq2sle3}
\frac 1 {c_0^*}\,\tilde{f}_s(\theta,x)\leq f(x)\leq c_0^*\,\tilde{f}_s(\theta,x).
\end{equation}
See section 4 of \cite{Japan} for more details.
Note that,
\begin{align*}
\bar\EE\left[H_{\disk\sm\gamma_t}(0,\xx)^b\,1\{\tilde\epsilon_t>\epsilon_0\}\right] & = \sin_2(x-u)^a\,e^{-3a^2t/8}\,\bar\EE\left[\frac{M_t}{M_0}\sin_2(\Theta_t)^{-a}1\{\tilde{\epsilon}_t>\epsilon_0\}\right]\\
& = \sin_2(x-u)^a\,e^{-3a^2t/8}\,\tilde\EE\left[\sin_2(\Theta_t)^{-a}1\{\tilde{\epsilon}_t>\epsilon_0\}\right].
\end{align*}
Now we can use \eqref{eq2sle3} to see that there exists a constant $c_0>0$ such that for large enough $t$
\[
\frac{1}{c_0}<\tilde\EE\left[\sin_2(\Theta_t)^{-a}1\{\tilde{\epsilon}_t>\epsilon_0\}\right] < c_0.
\]
Hence, there exist  constants $c_1,c_2>0$ such that
\[
c_1\leq \frac{\pf(r,(u,x),(w,y))}{r^b\,e^{-3ar/4}\,\sin_2(x-u)^a}\leq c_2.
\]
Finally, because only the lower bound in \eqref{eqchew} depends on $\epsilon_0$, we can see $c_2$ in the last inequality does not depend on $\epsilon_0$.

\end{proof}
\begin{corollary}\label{2slecor2}
There exist constants $0<c_*,r_*<\infty$ such that for all $0\leq u,x,w,y< 2\pi$ and $r>r_*$,
\[
\pf(r,(u,x),(w,y)) \leq c_*\,r^b\,e^{-3ar/4}\,\sin_2(y-w)^a.
\]
Furthermore, for any $\epsilon>0$, there exists a constant $c_{\epsilon}>0$ such that if $\min\{|x-u+2k\pi|;\,k\in\{-1,0,1\}\}>\epsilon_0$, then
\[
c_{\epsilon}\,r^b\,e^{-3ar/4}\,\sin_2(y-w)^a\leq\pf(r,(u,x),(w,y)).
\]
\end{corollary}
\begin{proof}
This follows from Proposition \ref{2sleprop1}, reversibility of $\sle$ and the fact that $f(z) = -e^r/z$ is a conformal transformation mapping $A_r$ to itself.
\end{proof}
\begin{proposition}\label{proplast}
Let $\pf(r,(u,x),(w,y))$ be as in \eqref{eq2}. Then uniformly over
 all $0\leq u,x,w,y< 2\pi$,
 \[
{\pf(r,(u,x),(w,y))}\asymp{r^b\,e^{-3ar/4}\,\sin_2(x-u)^a\sin_2(w-y)^a}.
 \]
 In other words, there exist constants $c_*,r_*>0$ such that for all $0\leq u,x,w,y< 2\pi$ and $r>r_*$,
\[
\frac{1}{c_*}\leq \frac{\pf(r,(u,x),(w,y))}{r^b\,e^{-3ar/4}\,\sin_2(x-u)^a\sin_2(w-y)^a}\leq c_*.
\]
\end{proposition}
\begin{proof}
Fix $1/2>\epsilon_0>1/3$ and let $\gamma_t$ be a $\sle$ curve from $\ubar$ to $\wbar$ in $A_r$. 
Define
\[
\epsilon^y_t = \min\{|Y_t-W_t+2k\pi|;\,k\in\{-1,0,1\}\},
\]
\[
\epsilon^x_t = \min\{|X_t-U_t + 2k\pi|;\,k\in\{-1,0,1\}\}.
\]
If $\epsilon^y_0>\epsilon_0$, the result follows from Proposition \ref{2sleprop1}. So we assume $\epsilon^y_0\leq\epsilon_0$. 
It follows from  \eqref{ucsd2} and  Corollary \ref{corback2} that for sufficiently large $r$ and $t=r/a$,
\begin{equation}\label{eq2sle3.5}
\left\lvert \frac{d\mu_{A_r}^\#(\ubar,\wbar)}{d\mu_{\disk}^\#(\ubar,0)}(\gamma_t)- 1\right\rvert \leq \frac{1}{2}
\end{equation}
and 
\begin{equation}\label{eq2sle4}
\frac{1}{2}<\frac{\epsilon^y_t}{\epsilon^y_0}<2.
\end{equation}
Define $\tau$ to be the hitting time of $C_r$ by $\gamma$ and  let $t=r/2$. Recall that $r(t)$ is the unique number satisfying $\hbar_t(A_r\sm\gamma_t) = A_{r(t)}$. 
Let $ \hat\gamma = \hbar_t(\gamma(t,\tau))$ and denote by $\mathcal{F}_t$ the sigma-algebra generated by $\gamma_t$. We can write
\begin{align*}
\pf(r,(u,x),(w,y))&= \EE\left[Q_{A_r}(\xx,\yy;\gamma_\tau)^b\right] \\
&= \EE\left[Q_{A_r}(\xx,\yy;\gamma_t)^b \EE\left[Q_{A_{r(t)}}(\XX_t,\YY_t;\hat\gamma)^b\middle|\mathcal{F}_t\right]\right].
\end{align*}
Conditioned on $\mathcal{F}_t$, $\hat\gamma$ is a $\sle$ from $\Ubar_t$ to $\Wbar_t$ in $A_{r(t)}$.
It follows from \eqref{eq2sle4} and Corollary \ref{2slecor2}   that there exists a constant $c = c({\epsilon_0})$ such that
\begin{align*}
\frac{1}{c}1\{\epsilon^x_t>2\epsilon_0\}  \leq \frac{\EE_{\tilde \gamma}\left[Q_{A_{r(t)}}(\XX_t,\YY_t;\tilde{\gamma})^b\middle| \mathcal{F}_t\right]}{r(t)^b\,e^{-3ar(t)/4}\,\sin_2(y-w)^a}\leq c.
\end{align*}
Moreover, lemmas \ref{2slelem1}, \ref{2sleeq1} and equation  \eqref{eq2sle3.5} give
\[
\EE\left[Q_{A_r}(\xx,\yy;\gamma_t)^b\right]\leq c \left(\frac{r}{r-at/2}\right)^b e^{-3a^2t/8} \sin_2(x-u)^a.
\]
We can see from an argument similar to the proof of Lemma \ref{2slelem1} that
\[
\frac{1}{c}\left(\frac{r}{r-at/2}\right)^b e^{-3a^2t/8} \sin_2(x-u)^a\leq\EE\left[Q_{A_r}(\xx,\yy;\gamma_t)^b1\{\epsilon_t^x>2\epsilon_0\}\right].
\]
Moreover, Koebe-$1/4$ theorem implies that 
\[
|r - at/2 - r(t) |\leq \log(4),
\]
from which the result follows.
\end{proof}

\begin{theorem}\label{handy}
There exist constants $0<c_*,r_*<\infty$ such that for all $0\leq u,x,w,y< 2\pi, \,r>r_*$,
\[
\frac{1}{c_*}\leq \frac{\Psi_{A_r}((\ubar,\xx),(\wbar,\yy))}{r^{\cc/2}\,\sin_2(x-u)^a\sin_2(w-y)^a\,e^{r\left(2b-\tilde{b}-{3a}/{4}\right)}}\leq c_*.
\]

\end{theorem}
\begin{proof}
Recall that 
\[
\Psi_{A_r}((\ubar,\xx),(\wbar,\yy)) = H_{A_r}(\xx,\yy)^b\,\Psi_{A_r}(\ubar,\wbar)\,\pf(r,(u,x),(w,y)).
\]
Using this, the result follows from Proposition \ref{proplast}, Lemma \ref{lemback0} and \eqref{ucsd1}.
\end{proof}

For $0\leq u,x,w,y<2\pi,\,r>0$, let $\gamma^1_t,\,\gamma^2_t$ be $\sle$ curves from $\ubar=\psi(u)$ to $\wbar=\psi(w+ir)$ and from $\xx=\psi(x)$ to $\yy=\psi(y+ir)$ in $A_r$ with radial parametrization.
\begin{corollary}
Suppose $\kappa=8/3$.  There exist constants $0<c_*,r_*<\infty$ such that for all $0\leq u,x,w,y< 2\pi, \,r>r_*$,
\[
\frac{1}{c_*}\leq \frac{\PP\left[\gamma^1\cap\gamma^2 = \emptyset\right]}{e^{-11r/24}\,\sin_2(x-u)^a\,\sin_2(y-w)^a}\leq c_*.
\]
\end{corollary}
\begin{proof}
This is an straightforward consequence of \eqref{ucsd1} and Theorem \ref{handy}.
\end{proof}

\begin{theorem}
There exist constants $0<c_*,r_*<\infty$ such that  the following holds.   Suppose $\nu_t(\ubar,\xx)$ is the measure defined in Proposition \ref{prop2sided3} and let $\bgamma_t = (\gamma_t^1,\gamma_t^2)$. For all $0\leq u,x,w,y< 2\pi, \,r>r_*$ and $ t<\frac{2(r-r_*)}{a}$, if
\[
M_t := \frac{d\mu_{A_r}((\ubar,\xx),\,(\wbar,\yy))}{d\nu_t(\ubar,\xx)}(\bgamma_t),
\]
then 
\[
\frac{1} {c_*}\leq \frac{M_t}{r^{\cc/2}\,\sin_2(w-y)^a\,e^{r\left(2b-\tilde{b}-{3a}/{4}\right)}} \leq c_*.
\]
\end{theorem}
\begin{proof}
Let $\bgamma=(\gamma^1,\gamma^2)$  and define
\[
R:=\frac{d\mu_{A_r}((\ubar,\xx),\,(\wbar,\yy))}{d\,\mu_{A_r}(\ubar,\wbar)\times\mu_{A_r}(\xx,\yy)}(\bgamma) = e^{\frac{\cc}{2}m_{A_r}(\gamma^1,\gamma^2)}1\{\gamma^1\cap\gamma^2=\emptyset\}.
\]
For $i\in\{1,2\}$, let $\tau_i$ be the time $\gamma^i$ hits $C_r$. For $t<\min\{\tau_1,\tau_2\}$, define the conformal transformations $\hbar^1_t:A_r\sm\gamma_t^1 \to A_{r_1(t)},\,\hbar^2_t:A_r\sm\gamma^2_t\to A_{r_2(t)}$ and
let $\tilde{\gamma}^2_t = \hbar^1_t(\gamma^2_t),\,\tilde{\gamma}^1_t = \hbar^2_t(\gamma^1_t)$. In addition, define the conformal transformations $\tilde{h}^1_t: A_{r_2(t)}\sm\tilde{\gamma}^1_t\to A_{r(t)},\,\tilde{h}^2_t: A_{r_1(t)}\sm\tilde{\gamma}^2_t\to A_{r(t)}$.
Let
\[
 \Ubar_t = \hbar^1_t(\gamma^1(t)),\quad \XX_t = \hbar^2_t(\gamma^2(t)),\quad \Wbar_t = \hbar^1_t(\wbar),\quad \YY_t = \hbar^2_t(\yy),\]
 \[
\tilde{U}_t = \tilde{h}^2_t(\Ubar_t) ,\quad \tilde{X}_t = \tilde{h}^1_t(\XX_t),\quad \tilde{W}_t = \tilde{h}^2_t(\Wbar_t),\quad \tilde{Y}_t = \tilde{h}^1_t(\YY_t).\]
Note that 
\begin{align*}
m_{A_r}(\gamma^1,\gamma^2) &= m_{A_r}(\gamma^1_t,\gamma^2_t) + m_{A_r\sm\gamma_t^1}(\gamma^1(t,\tau_1),\gamma^2_t) + m_{A_r\sm\gamma^2_t}(\gamma^1_t,\gamma^2(t,\tau_2))\\
&\quad +m_{A_r\sm\{\gamma^1_t\cup\gamma^2_t\}}(\gamma^1(t,\tau_1),\gamma^2(t,\tau_2))
\end{align*}
and
\begin{align*}
\left\{\gamma^1\cap\gamma^2=\emptyset\right\} &= \left\{\gamma^1_t\cap\gamma^2_t=\emptyset\right\} \cap \left\{\gamma^1_t\cap\gamma^2(t,\tau_2)=\emptyset\right\} \cap \left\{\gamma^1(t,\tau_1)\cap\gamma^2_t=\emptyset\right\}\\
&\quad\cap \left\{\gamma^1(t,\tau_1)\cap\gamma^2(t,\tau_2)=\emptyset\right\}.
\end{align*}
If 
\[
R_1 = \exp\left\{\frac{\cc}{2}m_{A_r\sm\gamma_t^1}(\gamma^1(t,\tau_1),\gamma^2_t)\right\}1\{\gamma^1(t,\tau_1)\cap\gamma^2_t=\emptyset\},
\]
\[
R_2 = \exp\left\{\frac{\cc}{2}m_{A_r\sm\gamma_t^2}\left(\gamma^1_t,\gamma^2(t,\tau_2)\right)\right\}1\{\gamma^1_t\cap\gamma^2(t,\tau_2)=\emptyset\},
\]
\[
R_{1,2}=\exp\left\{\frac{\cc}{2}m_{A_r\sm\{\gamma^1_t\cup\gamma_t^2\}}\left(\gamma^1(t,\tau_1),\gamma^2(t,\tau_2\right)\right\}1\{\gamma^1(t,\tau_1)\cap\gamma^2(t,\tau_2)=\emptyset\},
\]
then 
\[
R = R_1\,R_2\,R_{1,2}\,\exp\left\{\frac{\cc}{2}m_{A_r}(\gamma^1_t,\gamma^2_t)\right\}1\{\gamma^1_t\cap\gamma^2_t=\emptyset\}.
\]
Let $\mathcal{F}_t = \sigma(\gamma^1_t,\gamma^2_t)$ be the $\sigma$-algebra generated by the curves up to time $t$ and denote by $\EE$ the expectation with respect to the product measure $\mu^\#_{A_r}(\ubar,\wbar)\times\mu^\#_{A_r}(\xx,\yy)$. 
Conditioning on $F_t$, equation \eqref{theq6} implies that $\gamma^1(t,\tau_1)$ weighted by $R_1$ has the distribution of $\sle$ from $\gamma^1(t)$ to $\wbar$ in $A_r\sm\{\gamma^1_t\cup\gamma^2_t\}$.
 Similarly,  $\gamma^2(t,\tau_2)$ weighted by $R_2$ has the distribution of $\sle$ from $\gamma^2(t)$ to $\yy$ in $A_r\sm\{\gamma^1_t\cup\gamma^2_t\}$. 
If $\EE^1,\EE^2$ denote the expectations with respect to $\mu^\#_{A_r}(\ubar,\wbar),\,\mu^\#_{A_r}(\xx,\yy)$, then equations \eqref{eqdef0}, \eqref{theq6} give us
\begin{align*}
\EE^1[R_1|F_t] & = \frac{|\tilde{h}^{2\,\prime}_t(\Ubar_t)|^b|\tilde{h}^{2\,\prime}_t(\Wbar_t)|^b\,\Psi_{A_{r(t)}}(\tilde{U}_t,\tilde{W}_t)}{\Psi_{A_{r_1(t)}}(\Ubar_t,\Wbar_t)},\\
\EE^2[R_2|F_t] & = \frac{|\tilde{h}^{1\,\prime}_t(\XX_t)|^b|\tilde{h}^{1\,\prime}_t(\YY_t)|^b\,\Psi_{A_{r(t)}}(\tilde{X}_t,\tilde{Y}_t)}{\Psi_{A_{r_2(t)}}(\XX_t,\YY_t)}.
\end{align*}
Using this and definition \ref{def1}, we get
\begin{align}\label{eqproof1}
 N_t = \EE[R\,|\,\mathcal{F}_t]&=\Psi_{A_{r(t)}}((\tilde{U}_t,\tilde{X}_t),\,(\tilde{W}_t,\tilde{Y}_t))\exp\left\{\frac{\cc}{2}m_{A_r}(\gamma^1_t,\gamma^2_t)\right\}1\{\gamma^1_t\cap\gamma^2_t=\emptyset\}\\
 &\quad \times \frac{|\tilde{h}^{2\,\prime}_t(\Ubar_t)|^b|\tilde{h}^{2\,\prime}_t(\Wbar_t)|^b\,|\tilde{h}^{1\,\prime}_t(\XX_t)|^b|\tilde{h}^{1\,\prime}_t(\YY_t)|^b }{\Psi_{A_{r_1(t)}}(\Ubar_t,\Wbar_t)\Psi_{A_{r_2(t)}}(\XX_t,\YY_t)}.\nonumber
\end{align}
From Lemma \ref{lemback4} we have 
\begin{equation}\label{eqproof2}
r_1(t) = r-\frac{at}{2} + \bigo\left(e^{-r + at/2}\right),\qquad r_2(t) = r-\frac{at}{2} + \bigo\left(e^{-r + at/2}\right). 
\end{equation}
Lemma \ref{lemback3} implies that
\[
|\tilde{h}^{2\,\prime}_t(\Wbar_t)| = e^{r_1(t)-r(t)}\left[1 + O(e^{-r(t)})\right],\qquad |\tilde{h}^{1\,\prime}_t(\YY_t)| = e^{r_2(t)-r(t)}\left[1 + O(e^{-r(t)})\right].
\]
Let $\Theta_t,\,Z_t$ be as in \eqref{fundiego}, \eqref{eq2sided1.5} (also see Figure \ref{pic1}). Using Corollary \ref{corback1}, we can see that
\[
|\tilde{h}^{2\,\prime}_t(\Ubar_t)|^b\,|\tilde{h}^{1\,\prime}_t(\XX_t)|^b = Z_t \left[1 + \bigo(e^{-r(t)})\right].
\]
Let $U_t,X_t,W_t,Y_t$ be the unique continuous processes satisfying $\tilde U_t = \psi(U_t),\,\tilde{X}_t = \psi(X_t),\,\tilde{W}_t = \psi(W_t+ir(t)),\,\tilde{Y}_t = \psi(Y_t+ir(t))$ and $U_0 = u,\,X_0 = x,\,W_0 = w,\, Y_0 = y$.
Corollaries \ref{corback1}, \ref{corback2} imply that for large enough $r_*$,
\[
\frac{1}{2}<\frac{\sin_2(W_t-Y_t)}{\sin_2(w-y)}<2,\qquad \frac{1}{2}<\frac{\sin_2(X_t-U_t)}{\sin_2(\Theta_t)}<2.
\]
It is not hard to see that
\[
\mathcal{E}_{A_r\sm\gamma^1_t}(C_r,C_0\cup\gamma_t^1) + \mathcal{E}_{A_r\sm\gamma^2_t}(C_r,C_0\cup\gamma_t^2)\geq \mathcal{E}_{A_r\sm\{\gamma_t^1\cup\gamma_t^2\}}(C_r,C_0\cup\{\gamma_t^1\cup\gamma_t^2\}).
\]
Considering this, we can use Koebe-$1/4$ theorem to see that  \eqref{eqproof2}
\[
\frac{4\pi}{r-at/2-\log(4)}\geq \frac{2\pi}{r(t)}
\]
and 
\[
r(t)\geq \frac{r-at/2-\log(4)}{2}>\frac{r_*}{2}-1.
\]
Therefore, we can use
 Theorem \ref{handy} to see
\[
\Psi_{A_{r(t)}}((\tilde{U}_t,\tilde{X}_t),\,(\tilde{W}_t,\tilde{Y}_t)) \asymp r(t)^{\cc/2}\, e^{r(t)\,(-3a/4+2b-\tilde{b})}\,\sin_2( \Theta_t)^a\,\sin_2(y-w)^a.
\]
By \eqref{ucsd1} and \eqref{eqproof2},
\[
\frac{1}{\Psi_{A_{r_1(t)}}(\Ubar_t,\Wbar_t)\,\Psi_{A_{r_2(t)}}(\XX_t,\YY_t)} = \left(\frac{1}{r_1(t)\,r_2(t)}\right)^{\cc/2} e^{(at-2r)\,(b-\tilde{b})} \left[ 1+ \bigo\left(e^{-q(r-at/2)}\right)\right],
\]
for some constant $q$. 
Plugging all these estimates into \eqref{eqproof1} gives 
\begin{align}\label{eqproof3}
 N_t &\asymp \left(\frac{r(t)}{r_1(t)r_2(t)}\right)^{\cc/2}\exp\left\{\frac{\cc}{2}m_{A_r}(\gamma^1_t,\gamma^2_t)\right\}1\{\gamma^1_t\cap\gamma^2_t=\emptyset\}Z_t \\
&\quad\times\, e^{2r\tilde b -\tilde{b}at-r(t)(\tilde{b}+3a/4)}\sin_2( \Theta_t)^a\,\sin_2(y-w)^a.\nonumber
\end{align}
Let
\begin{align*}
M_t^1:=\frac{d\mu_{A_r}(\ubar,\wbar)}{d\mu_\disk(\ubar,0)}(\gamma^1_t),
\qquad M_t^2:=\frac{d\mu_{A_r}(\xbar,\yy)}{d\mu_\disk(\xx,0)}(\gamma^2_t) .\nonumber
\end{align*}
Define $\hat M_t$ to be the martingale satisfying 
\[
\frac{d\mu_{A_r}((\ubar,\xx),\,(\wbar,\yy))}{d\,\mu_{\disk}(\ubar,0)\times\mu_{\disk}(\xx,0)}(\bgamma_t) = \hat M_t.
\]
That is, $ \hat M_t = M_t^1\,M_t^2\,N_t$. 
Using \eqref{ucsd2}, there exist constants $0<\hat c,q<\infty$ such that
\[
M_t^1= \hat c\,r^{\cc/2}e^{(b-\tilde{b})r}\left[1+\bigo\left(e^{-q(r-at/2)}\right)\right],\qquad M_t^2 =  \hat c\,r^{\cc/2}e^{(b-\tilde{b})r}\left[1+\bigo\left(e^{-q(r-at/2)}\right)\right].
\]
Using \eqref{eqproof1} and \eqref{eqproof3}, we get
\begin{align*}
 \hat M_t &\asymp \left(\frac{r^2\,r(t)}{r_1(t)r_2(t)}\right)^{\cc/2}\exp\left\{\frac{\cc}{2}m_{A_r}(\gamma^1_t,\gamma^2_t)\right\}1\{\gamma^1_t\cap\gamma^2_t=\emptyset\}Z_t \\
&\quad\times\, e^{2rb -\tilde{b}at-r(t)(\tilde{b}+3a/4)}\sin_2( \Theta_t)^a\,\sin_2(y-w)^a.\nonumber
\end{align*}
By using Lemma \ref{lemback3} we can see that 
\[
O_t \asymp \exp\left\{\frac{\cc}{2}m_{\disk}(\gamma^1_t,\gamma^2_t)\right\}1\{\gamma^1_t\cap\gamma^2_t=\emptyset\}\, Z_t\, \sin_2(\Theta_t)^a\, e^{(r-at/2-r(t))(\tilde{b}+3a/4)+at/2(3a/4-\tilde{b})}.
\]
It follows from Lemma \ref{lemback4} that
\[
m_\disk(\gamma^1_t,\gamma^2_t)- m_{A_r}(\gamma_t^1,\gamma_t^2) = \log\frac r{r_1(t)} + \log \frac r {r_2(t)}  - \log\frac r {r(t)} + \bigo(e^{-r(t)}).
\]
Therefore,
\[
M_t = \frac{\hat{M}_t}{O_t}\asymp r^{\cc/2}\,e^{r(2b-\tilde{b}-3a/4)}\sin_2(y-w)^a.
\]
\end{proof}

\end{document}